\documentclass{amsart}
\usepackage{pgf,amsthm,amscd,amssymb,verbatim,epsf,amsmath,eurosym,amsfonts,mathrsfs,graphicx,tikz-cd,lscape}
\usepackage{amsthm,amscd,amssymb,verbatim,epsf,amsmath,amsfonts,mathrsfs,graphicx,dirtytalk,pdfpages,mdframed,tikz-cd,xy}
\usepackage[colorlinks=true,linkcolor=blue,citecolor=blue]{hyperref}
\usepackage{epigraph}
\begin{document}
\theoremstyle{plain}
\newtheorem{Thm}{Theorem}
\newtheorem{Cor}[Thm]{Corollary}
\newtheorem{Ex}[Thm]{Example}
\newtheorem{Con}[Thm]{Conjecture}
\newtheorem{Main}{Main Theorem}
\newtheorem{Lem}[Thm]{Lemma}
\newtheorem{Prop}[Thm]{Proposition}

\theoremstyle{definition}
\newtheorem{Def}[Thm]{Definition}
\newtheorem{Note}[Thm]{Note}
\newtheorem{Question}[Thm]{Question}

\theoremstyle{remark}
\newtheorem{notation}[Thm]{Notation}
\renewcommand{\thenotation}{}

\errorcontextlines=0
\renewcommand{\rm}{\normalshape}%
\newcommand{\transv}{\mathrel{\text{\tpitchfork}}}
\makeatletter
\newcommand{\tpitchfork}{%
  \vbox{
    \baselineskip\z@skip
    \lineskip-.52ex
    \lineskiplimit\maxdimen
    \m@th
    \ialign{##\crcr\hidewidth\smash{$-$}\hidewidth\crcr$\pitchfork$\crcr}
  }%
}
\makeatother

\title[Carath\'eodory Conjecture]%
   {Proof of the\\ Carath\'eodory Conjecture}
\author{Brendan Guilfoyle}\address{Brendan Guilfoyle\\
          Department of Computing and Mathematics\\
          Institute of Technology, Tralee \\
          Clash \\
          Tralee  \\
          Co. Kerry \\
          Ireland.}
\email{brendan.guilfoyle@@ittralee.ie}
\author{Wilhelm Klingenberg}
\address{Wilhelm Klingenberg\\
 Department of Mathematical Sciences\\
 University of Durham\\
 Durham DH1 3LE\\
 United Kingdom}
\email{wilhelm.klingenberg@@durham.ac.uk }

\date{31st July 2013}

\maketitle

\tableofcontents

\newpage

\section{{\bf Introduction}}

In this paper we prove a conjecture attributed\footnote{Hans Hamburger, Berliner Mathematische Gesellschaft, Berlin, 26th March 1924.} 
to Constantin Carath\'eodory:

\vspace{0.1in}

\noindent{\bf Main Theorem.} {\it Every closed strictly convex surface in Euclidean 3-space has at least two umbilic points.}

\vspace{0.1in}

Recall that an umbilic point is a point where the
second fundamental form of the surface (represented by a symmetric 2-by-2 matrix) has a double
eigenvalue. Since the eigen-directions of the second fundamental form determine a foliation of the surface with singularities precisely at 
the umbilic points, for topological reasons a closed convex surface must have at least one umbilic point.

While the conjecture applies to $C^2$-smooth surfaces, we prove it for $C^{3+\alpha}$-smooth surfaces. 
Our proof depends upon a reformulation in terms of
complex points on Lagrangian surfaces in the space of oriented geodesics of Euclidean 3-space ${\mathbb E}^3$, which is identified with $TS^2$. 
Here complex and Lagrangian refer to the
neutral K\"ahler structure on $TS^2$ introduced by the authors in \cite{gak4} and neutral means that the signature of the metric is $(++--)$.

More specifically, the reformulated conjecture states that every closed Lagrangian section of $TS^2$ has at least two complex points. In this paper we
prove this conjecture for $C^{2+\alpha}$ sections.

We first show that a $C^{2+\alpha}$-smooth Lagrangian section of $TS^2$ with just one complex point, if such exists, lies in an open subset ${\mathcal U}$
of a Banach manifold. Surjectivity of the Cauchy-Riemann operator implies that, in a dense open subset of ${\mathcal U}$, the dimension 
of the space of holomorphic discs with edge lying on a Lagrangian section is determined by the Keller-Maslov class of the edge curve. 

The neutral geometry introduced by the authors in \cite{gak4} identifies the Keller-Maslov index with the number of complex points on the 
boundary surface enclosed by the edge curve. Thus there cannot exist a holomorphic disc with edge enclosing regions without complex points on a
$C^{2+\alpha}$ generic Lagrangian boundary surface near to the section with only one complex point. 

Every $C^{2+\alpha}$ Lagrangian boundary surface 
near to a section with only one complex point, should such exist, contains a totally real Lagrangian hemisphere.  Therefore, were the
Conjecture false, there would exist a totally real Lagrangian hemisphere which could not be the boundary for any holomorphic disc.

But we prove that it is possible to attach a holomorphic disc to {\it any} $C^{2+\alpha}$ totally real Lagrangian hemisphere. 
This implies that the set of $C^{2+\alpha}$-smooth Lagrangian sections of $TS^2$ with 
just one complex point must be empty. Noting the drop in differentiability going from ${\mathbb E}^3$ to $TS^2$, we have therefore proven
the Carath\'eodory conjecture for $C^{3+\alpha}$-smooth surfaces in ${\mathbb E}^3$.

The existence of a holomorphic disc uses mean curvature flow with boundary and a compactness result on spaces of $J$-holomorphic discs. The 
flow requires $C^{2+\alpha}$-smoothness of the boundary condition for long-time existence. 

Our proof is organised as follows. The next section contains the reformulation and proof of the Main Theorem. In the following sections we supply the 
technical details of the proof. In section 3 we consider mean curvature flow in indefinite spaces of any dimension. The neutral geometry of $TS^2$ 
is summarized in section 4, while mean curvature flow with boundary in $TS^2$ is considered in detail in section 5. Long-time and short-time 
existence for the flow, along with the  existence of 
holomorphic discs with Lagrangian boundary conditions, which completes the proof, are established in the section 6. In the final section
we make some concluding remarks on the Conjecture and our proof.

\vspace{0.1in}

\noindent{\bf Note}: 
Here and throughout we refer to the {\em edge} of the flowing disc and reserve the phrase {\em boundary surface} or {\em boundary condition} 
for the fixed surface on which the edge of the flowing disc is restricted to lie.

\vspace{0.2in}


\section{{\bf Strategy and Proof}}

In this section we give a reformulation of the Carath\'eodory conjecture in terms of complex points on Lagrangian surfaces
in $TS^2$. We then prove the reformulated Conjecture, referring to results established in later sections.

\subsection{Reformulation of the Conjecture in $TS^2$}\label{s:reform}

It is well-known that the space of oriented geodesics in Euclidean 3-space ${\mathbb E}^3$ may be identified with the total space of the tangent 
bundle to the 2-sphere. 
This 4-manifold is endowed with a natural neutral K\"ahler structure
$({\mathbb J},\Omega,{\mathbb G})$ which is invariant under the action induced on $TS^2$ by the Euclidean group acting on
${\mathbb E}^3$ \cite{gak4}. 

Throughout this paper, we denote the neutral K\"ahler surface $(TS^2,{\mathbb J}$,$\Omega$,${\mathbb G})$ simply by $TS^2$.
The term neutral refers to the fact that the metric ${\mathbb G}$ is indefinite, having signature $(++--)$.
We now briefly summarize some properties of this structure, which will be considered more fully in section \ref{s:neutral}.  

Given an oriented $C^k$-smooth surface S in ${\mathbb E}^3$ with
$k\geq 1$, the set of oriented lines normal to S forms a
$C^{k-1}$-smooth surface $\Sigma$ in $TS^2$. Such a surface
$\Sigma$ is Lagrangian: $\Omega|_\Sigma=0$. Indeed, the well-known converse
holds by Frobenius integrability (see for example \cite{gak2}):

\begin{Prop}\label{p:lag}
A surface $\Sigma$ in $TS^2$ is Lagrangian iff there exists a surface S in ${\mathbb E}^3$ which is orthogonal to the oriented lines of
$\Sigma$.
\end{Prop}

Given one surface $S$ in ${\mathbb E}^3$  orthogonal  to the oriented lines of $\Sigma$, we have a 1-parameter family of parallel surfaces 
which are also orthogonal. Moreover, a point on $S$ is umbilic iff the corresponding points on the parallel surfaces are also umbilic. 
Indeed, it is precisely this property that allows us to reformulate the Conjecture entirely in $TS^2$, as we now show.

Let $S$ be an oriented surface in ${\mathbb E}^3$, and  $\Sigma$ the corresponding surface in $TS^2$ formed by the oriented 
normal lines to $S$. The canonical projection
$\pi:TS^2\rightarrow S^2$ restricted to $\Sigma$ is just the Gauss map of the surface, and so we have:

\begin{Prop}\label{p:Gauss}
The surface $S$ is non-flat (has non-zero Gauss curvature) iff the Lagrangian surface
$\Sigma$ is the graph of a section of the bundle $\pi:TS^2\rightarrow S^2$. In particular, the surface $\Sigma$ in
$TS^2$ formed by the oriented normal lines of a convex
surface S is the graph of a section.
\end{Prop}

In general, a point $\gamma$ on a surface $\Sigma$ in an almost complex 4-manifold
$({\mathbb M}, J)$ is said to be {\it complex} if
$J:T_\gamma{\mathbb M}\rightarrow T_\gamma{\mathbb M}$ leaves $T_\gamma\Sigma$ invariant. In our setting:

\vspace{0.1in}

\begin{Prop}\cite{gak2}\label{p:index}
A point $\gamma$ on a Lagrangian surface $\Sigma$ in $TS^2$
is complex iff the point on the orthogonal surface $S$ in ${\mathbb
E}^3$ with oriented normal line $\gamma$ is umbilic.

Moreover, the index $i(p)$ of an isolated umbilic point $p$ on $S$ is related to the Keller-Maslov index of a simple closed curve 
about the oriented normal $\gamma$ on $\Sigma$ by:
\[
\mu(TS^2,T\Sigma)=4i(p).
\]
\end{Prop}

\vspace{0.1in}

Here, an umbilic point on $S\subset{\mathbb E}^3$ is a point where
the second fundamental form has a double eigenvalue, the ${\textstyle{\frac{1}{2}}}{\mathbb Z}$-valued index $i(p)$ of an 
isolated umbilic point $p$ is the winding number of the eigen-directions about the point and the Keller-Maslov index $\mu(TS^2,T\Sigma)$ is the 
relative 1st Chern class of the pair over a simple closed curve on $\Sigma$ about $\gamma$.

Propositions \ref{p:lag} to \ref{p:index} prove that the Carath\'eodory conjecture is equivalent to:

\vspace{0.1in}

\noindent{\bf Reformulation.} {\it Every closed Lagrangian section of $TS^2\rightarrow S^2$ has at least two complex points.}

\vspace{0.1in}

This reformulation in the 1-jet involves a drop in differentiability: it holds for a $C^k$-smooth Lagrangian section of $TS^2$
if and only if it holds for a $C^{k+1}$-smooth surface in ${\mathbb E}^3$.

Before proceeding with the proof, we exhibit a smooth family of surfaces in ${\mathbb E}^3$ with certain properties that are 
germane to the Carath\'eodory conjecture. We return to this example in section \ref{s:nkm}.

\begin{Ex}\label{ex1}
Consider the 1-parameter family of surfaces in ${\mathbb E}^3$ parameterized by 
$\xi\mapsto (x^1(\xi,\bar{\xi}),x^2(\xi,\bar{\xi}),x^3(\xi,\bar{\xi}))$, for $\xi\in{\mathbb C}$ and:
\[
x^1+ix^2=\frac{2[\bar{\xi}(1+2\xi\bar{\xi}-\xi^4)+C\xi]}{1+\xi\bar{\xi}}
\qquad\qquad
x^3=\frac{-(1+3\xi\bar{\xi})(\xi^2+\bar{\xi}^2)+C(1-\xi\bar{\xi})}{1+\xi\bar{\xi}}.
\]

This family exhibits the following features:
\begin{enumerate}
\item The surfaces are parameterized by the inverse of their Gauss maps, with stereographic projection from 
the south pole.
\item For different values of $C$ the surfaces are parallel: they can be obtained by moving a fixed distance along their normal lines.
\item For large $C$ the surfaces are convex. More specifically, the surface will be convex at $\xi=Re^{i\theta}$ if $C>1+2(1+2\sin^2\theta)R^2+3R^4$.
 For example, the surfaces with $C>10$ are convex for a whole Gauss hemisphere about 0.
\item The surfaces have no umbilic points.
\end{enumerate}

While the surfaces are not closed, for large enough $C$ we can construct a smooth convex surface without umbilic points 
with arbitrary large Gauss area and such surfaces can always be smoothly closed. Above is the umbilic-free hemisphere obtained by putting $C=11$.
\end{Ex}

\vspace{0.1in}

From this example we conclude that the umbilic points on a closed convex surface can occur arbitrarily close 
together. One interpretation of Carath\'eodory's conjecture is
that, despite this, they cannot be brought to a single point while closing the surface differentiably.

\vspace{0.1in}
\includegraphics{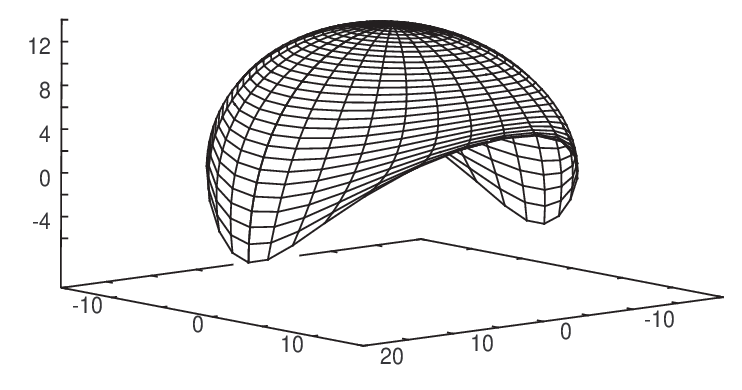}
\vspace{0.1in}

In what follows, such  \say{umbilic-free hemispheres} play an important role. They motivate the following definition:

\begin{Def}
A local section of $TS^2\rightarrow S^2$ is a {\it hemisphere} if its domain is a hemisphere in $S^2$. A {\it totally real Lagrangian hemisphere} is a 
hemisphere in $TS^2$ which is Lagrangian and contains no complex points.
\end{Def}

By Propositions \ref{p:lag} and \ref{p:Gauss}, a Lagrangian hemisphere $\Sigma$ in $TS^2$ corresponds to a strictly convex surface $S$ in ${\mathbb E}^3$
whose image under the Gauss map contains a hemisphere. Moreover, should $\Sigma$ be totally real (i.e. without complex points) then $S$ 
is umbilic-free - and vice-versa. 
In summary, the oriented lines normal to an umbilic-free hemisphere (as pictured above) form a totally real Lagrangian hemisphere in $TS^2$.

Our proof can be summarized as follows. Every putative closed convex surface with a single umbilic point contains a umbilic-free hemisphere.
We demonstrate using mean curvature flow that given any umbilic-free hemisphere one can always find a holomorphic 2-parameter family
of oriented lines which agree with a curve of oriented normals to the hemisphere. 

As we will see in the next section, global arguments then preclude the possibility of closing the hemisphere differentiably with the creation of only 
a single umbilic point. Thus a closed convex surface with a single umbilic point does not exist.

\vspace{0.1in}

\subsection{The manifold of holomorphic discs with boundary}\label{s:mnfldjdiscs}

We now analyze holomorphic discs in the complex surface $(TS^2,{\mathbb J})$ with edge lying on a boundary surface. 
Fix  $\alpha\in(0,1)$, $s\geq 1$ and denote by $C^{k+\alpha}$ and $H^{1+s}$ the usual H\"older and Sobolev spaces, respectively.

Define the space of {\it H\"older boundary conditions} by
\[
{\mathcal L}\equiv \left\{ \;\Sigma\subset TS^2 \; \;\left| \;\;  \Sigma\; {\mbox{is a}}\; C^{2+\alpha} \;{\mbox{-section}}\;\right.\right\}  .
\]
Endow ${\mathcal L}$ with the $C^{2+\alpha}$ topology  for global sections of $TS^2$.

For a fixed $\Sigma\in{\mathcal L}$, the differentiable structure of the set of $C^{2+\alpha}$ sections of the normal bundle 
$N\Sigma=T_\Sigma TS^2/T\Sigma$ exponentiates to 
give an infinite-dimensional manifold structure for an open neighbourhood of $\Sigma$ in ${\mathcal L}$. If $\Sigma$ were totally real, then
$N\Sigma\cong {\mathbb J}T\Sigma\cong T\Sigma$, which gives a canonical Banach manifold structure modeled on sections of $T\Sigma$:
\[
\Gamma T \Sigma \cong \Gamma J T\Sigma  \hookrightarrow \Gamma N \Sigma \stackrel{{\mbox{exp}}}{\rightarrow} {\mathcal L} .
\] 

In our situation, the section always has at least one complex point and so we must modify this argument. In particular, fix a point $\gamma_0\in TS^2$
and define
\[
{\mathcal L}_0\equiv \left\{ \;\Sigma\in{\mathcal L}\;\; \left| \;\;  \gamma_0\in\Sigma \;\right.\right\}  ,
\]
endowed with the induced  H\"older space $C^{2,\alpha}$ -topology.

Note that ${\mathcal L}_0$ can be identified with the
quotient of ${\mathcal L}$ by the translation that takes the point $\pi^{-1}(\pi(\gamma_0))\cap\Sigma\in{\mathcal L}$ to $\gamma_0$. That is, 
${\mathcal L}_0={\mathcal L}/\tilde Aut(TS^2)$, where $\tilde Aut(TS^2) = Aut(TS^2)/Aut_{\gamma_0}(TS^2) $ is the quotient of the isometry group 
$Aut(TS^2)$ of $(TS^2,{\mathbb G})$ by the stabilizer subgroup of $\gamma_0 \in TS^2$.
In addition, $Aut(TS^2)$ acts holomorphically and symplectically on $TS^2$ and continuously on ${\mathcal L}$ \cite{gak4}. Therefore it takes complex points
of $\Sigma$ to complex points of its image, and preserves the Lagrangian property for surfaces.

Let $\Sigma_1\in{\mathcal L}_0$ have exactly one complex point at $\gamma_0$. Then there exists an open neighbourhood of $\Sigma_1$ in
${\mathcal L}_0$ which is a Banach manifold. We prove this as follows.

\vspace{0.1in}

\begin{Lem}
Let $\Sigma_1 \in {\mathcal L}_0 $ be a surface with exactly one complex point at $\gamma_0$.
Then there exists a canonical embedding 
\[
\tilde\Phi : \Gamma ({\mathbb J}T\Sigma_1/T\Sigma_1)   \hookrightarrow {\mathcal L}/\tilde Aut(TS^2)\cong{\mathcal L}_0 , 
\]
In particular,  there exists an open neighbourhood of  $\Sigma_1$ in ${\mathcal L}_0$ 
which has a canonical Banach manifold structure.
\end{Lem}

\begin{proof}
Let $\pi_N : T_\Sigma TS^2 \rightarrow T_\Sigma TS^2/T\Sigma\cong  N\Sigma $ be the normal bundle of $\Sigma$.
 
Note that for $\Sigma \in \mathcal L$, the linearization of $ \tilde Aut(TS^2)$
at the identity gives rise to the quotient projection
$\tilde\pi_N : T_\Sigma TS^2 / d \tilde Aut_{\gamma_0}(TS^2)  \to N\Sigma / d\tilde Aut_{\gamma_0}(TS^2).$ 
Since the exponential map commutes with the isometries of $Aut(TS^2)$, we have:
\[
\begin{array}[c]{ccc}
 \Gamma T_\Sigma TS^2/d\tilde Aut_{\gamma_0}(TS^2) &\stackrel{\tilde{{\mbox{exp}}}}
{\rightarrow}& {\mathcal L}/\tilde Aut(TS^2) \\
\downarrow\scriptstyle{\tilde\pi_N}&&\downarrow\scriptstyle{{\mbox{id}}}\\
 \Gamma N\Sigma/dAut_{\gamma_0}(TS^2) &\stackrel{\tilde{{\mbox{exp}}}}{\rightarrow}& {\mathcal L}/\tilde Aut(TS^2).
\end{array}
\]

The key point is that if $\Sigma_1$ has only one complex point at $\gamma_0$, then
${\mathbb J}T\Sigma_1/T\Sigma_1 \stackrel{\cong}\hookrightarrow N\Sigma_1 /d\tilde Aut_{\gamma_0}(TS^2) $ is an  
isomorphism of plane bundles over $\Sigma_1$ with 
vanishing fibre at one point only, namely the complex point, where ${\mathbb J}T_{\gamma_0} \Sigma_1 = T_{\gamma_0}\Sigma_1$. 
This is  no longer an isomorphism if $\Sigma_1$ has more than one complex point. 

We  conclude that the composition
\[
\tilde\Phi : \Gamma ({\mathbb J}T\Sigma_1/T\Sigma_1)
 \hookrightarrow \Gamma N \Sigma_1 /d\tilde Aut_{\gamma_0}(TS^2) 
 \stackrel{\tilde{{\mbox{exp}}}}{\rightarrow} {\mathcal L}/\tilde Aut(TS^2) 
\] 
is an embedding of the Banach space of $C^{2+\alpha} $ - smooth sections of  ${\mathbb J}T\Sigma_1/T\Sigma_1$
into $ {\mathcal L}/\tilde Aut(TS^2)$, giving rise to an open Banach manifold
of variations of  $ \Sigma_1 \in {\mathcal L}_0$.

\end{proof}

\vspace{0.1in}

In a similar way we define the spaces of Lagrangian boundary conditions
\[
{\mathcal L}{\mbox{ag}}\equiv \left\{ \;\Sigma\subset TS^2 \; \;\left| \;\;  \Sigma\; {\mbox{is a}}\; C^{2+\alpha} \;{\mbox{ Lagrangian section}}\;\right.\right\}  
\] 
\[
{\mathcal L}{\mbox{ag}}_0\equiv \left\{ \;\Sigma\in{\mathcal L}{\mbox{ag}}\;\; \left| \;\;  \gamma_0\in\Sigma \;\right.\right\}\equiv {\mathcal L}{\mbox{ag}}/\tilde Aut(TS^2).
\]
\vspace{0.1in}
\begin{Lem}
Let $\Sigma_1 \in {\mathcal L}{\mbox{ag}}_0 $ be a surface with exactly one complex point at $\gamma_0$. Then there exists an open neighbourhood ${\mathcal U}$ of  
$\Sigma_1$ in ${\mathcal L}{\mbox{ag}}_0$ 
which has a canonical Banach manifold structure.
\end{Lem}
\begin{proof}
The proof follows that of the previous Lemma, with the manifold modeled on the Banach space of Lagrangian sections
\[
\Gamma^{lag}=\left\{\;v\in\Gamma({\mathbb J}(T\Sigma_1))\;\;\left|\;\; v\in C^{2+\alpha},\;\;d({\mathbb J}(v)\lrcorner\Omega)=0\;\right.\right\}.
\] 
\end{proof}
\vspace{0.1in}

For $s > 1, (2 - \alpha)/2 = 1/s $ and a relative class $A\in\pi_2(TS^2,\Sigma)$, the space of
{\it parameterized Sobolev-regular discs with Lagrangian boundary condition} is defined by
\[
{\mathcal{F}}_A\equiv \left\{\;(f, \Sigma)\in H^{1+s}(TS^2)\times{\mathcal U}\;\;\left|\;\; [f]=A,\; f(\partial D)\subset {\mbox{ totally real part of }} \Sigma   
\;\right.\right\},
\]
where ${\mathcal U}$ is the Banach manifold neighbourhood of $\Sigma_1$ as above.
The space ${\mathcal{F}}_A$ is a Banach manifold and the so the projection 
$\pi:{\mathcal{F}}_A\rightarrow {\mathcal{U}}$: $\pi(f, \Sigma)= \Sigma$ is a Banach
bundle. 

For $(f, \Sigma)\in{\mathcal F}_A$ define 
$\bar{\partial}f={\textstyle{\frac{1}{2}}}(df\circ j- J \circ df)$, where
$j$ is the complex structure on $D$. Then
$\bar{\partial}f\in H^s(f^*T ^{01}TS^2) \equiv H^s(f^*TTS^2) $ and we define the space of sections
\[
H^s\equiv\bigcup_{(f, \Sigma)\in{\mathcal F}_A}H^s(f^*TTS^2).
\]
This is a Banach vector bundle over ${\mathcal F}_A$ and the operator $\bar{\partial}$ is
a section of this bundle.

\vspace{0.1in}
\begin{Def}
The {\it set of holomorphic discs with Lagrangian boundary condition} is defined by
\[
{\mathcal M}_A\equiv\left\{\;(f, \Sigma)\in{\mathcal F}_A\;\;\left|\;\; \bar{\partial}f=0 \;\right.\right\}.
\]
\end{Def}
\vspace{0.1in}

As before let $\Sigma_1$ be a Lagrangian section with a single complex point at $\gamma_0$. 

\vspace{0.1in}
\begin{Prop}\label{p:banach}
There exists a neighbourhood of $\Sigma_1$, denoted
${\mathcal V}\subset{\mathcal L}ag_0$, such that ${\mathcal M}_A\cap\pi^{-1}({\mathcal V})$ is a Banach submanifold of ${\mathcal F}_A\cap\pi^{-1}({\mathcal V})$.
\end{Prop}
\begin{proof}
This is established by considering the smooth map defined by
\[
\Delta:{\mathcal F}_A\cap\pi^{-1}({\mathcal U})\rightarrow H^s\times \Omega(TS^2),
\]
by $\Delta(f,\Sigma)=(\bar{\partial}f,\Phi^{-1}_\Sigma\circ f|_{\partial D})$, where 
$\Omega(TS^2)$ is the set of loops in $TS^2$ and $\Phi_\Sigma$ is the ambient Hamiltonian isotopy which takes $\Sigma_1$ to $\Sigma$. For ease
of notation, we suppress the composition $\Phi^{-1}_\Sigma\circ f|_{\partial D}$ and simply write $f|_{\partial D}$.

Thus
\[
{\mathcal M}_A\cap\pi^{-1}({\mathcal U})=\Delta^{-1}(\{0\}\times\Omega_A({\mathcal U})).
\]
To prove the proposition we must show that $\Delta$ is transverse to the submanifold
\[
\{0\}\times\Omega_A({\mathcal U})\subset H^s\times \Omega(TS^2),
\]
at $\Sigma_1$. Transversality at $\Sigma_1$ is proved exactly as in the tame case (Theorem 1 in Oh \cite{oh}), which we now outline.

Let $(f,\Sigma)\in{\mathcal M}_A$, which means that
\[
\bar{\partial}f=0
\qquad\qquad
f|_{\partial D}\subset\Sigma_1
\qquad\qquad
[f]=A {\mbox{ in }}\pi_2(TS^2,\Sigma_1).
\]
To show transversality, we need that to prove that
\[
{\mbox{Im}}\left(D_{(f,\Sigma)}\Delta\right)+\{0\}\oplus T_{f|_{\partial D}}\Omega_A({\mathcal U})=
T_{(0,f|_{\partial D})} H^s\times \Omega(TS^2),
\]
or, denoting the $L^2$-adjoint by $^{\perp_{L^2}}$, equivalently
\begin{align}
0&=\left({\mbox{Im}}\left(D_{(f,\Sigma)}\Delta\right)+\{0\}\oplus T_{f|_{\partial D}}\Omega_A({\mathcal U})\right)^{\perp_{L^2}}\nonumber\\
&=\left({\mbox{Im}}\left(D_{(f,\Sigma)}\Delta\right)\right)^{\perp_{L^2}}\cap\left(\{0\}\oplus T_{f|_{\partial D}}\Omega_A({\mathcal U})\right)^{\perp_{L^2}}\nonumber\\
&=\left({\mbox{Im}}\left(D_{(f,\Sigma)}\Delta\right)\right)^{\perp_{L^2}}\cap\left(H^{-s}\oplus \left(T_{f|_{\partial D}}\Omega_A({\mathcal U})\right)^{\perp_{L^2}}\right)\label{e:adj}.
\end{align}
A point in $T_{(f,\Sigma_1)}\left({\mathcal F}_A\cap\pi^{-1}({\mathcal U})\right)$ can be represented by $(\zeta,X_h)$, where $X_h$ is the Hamiltonian vector field associated 
with some $h\in C_0^\infty(\Sigma_1)$ and we find that
\[
D_{(f,\Sigma_1)}\Delta(\zeta,X_h)=({\nabla}^+_{\mathbb J}\zeta,X_h(f|_{\partial D})-\zeta(f|_{\partial D})),
\]
where we have introduced the connection
\[
\nabla^{\pm}_{\mathbb J}={\textstyle{\frac{1}{2}}}\left(\frac{D}{dx}\pm{\mathbb J}\frac{D}{dy}\right).
\]
Now let $(\psi,\alpha)\in\left({\mbox{Im}}\left(D_{(f,\Sigma)}\Delta\right)+\{0\}\oplus T_{f|_{\partial D}}\Omega_A({\mathcal U})\right)^{\perp_{L^2}}$. We show
that  $(\psi,\alpha)=(0,0)$ as follows. By definition we have 
\[
\int_D({\nabla}^+_{\mathbb J}\zeta,\psi)+\int_{\partial D}(X_h(f|_{\partial D})-\zeta(f|_{\partial D}),\alpha)=0,
\]
for all $(\zeta,X_h)\in T_{(f,\Sigma_1)}\left({\mathcal F}_A\cap\pi^{-1}({\mathcal U})\right)$.

Integrating by parts and rearranging terms
\[
-\int_D(\zeta,{\nabla}^-_{\mathbb J}\psi)+\int_{\partial D}(\zeta,\tilde{\psi}-\alpha)+\int_{\partial D}(X_h\circ f|_{\partial D}),\alpha)=0,
\]
where $\tilde{\psi}$ is the ${\mathbb J}$-adjoint of $\psi$. As this holds for all $\zeta$ and $h$, we have
\[
{\nabla}^-_{\mathbb J}\psi=0
\qquad\qquad
\tilde{\psi}-\alpha=0 {\mbox{ on }}\partial D
\qquad\qquad
\alpha^\perp=0 {\mbox{ on }}\partial D .
\]
Since, by equation (\ref{e:adj}) $\alpha\in\left(T_{f|_{\partial D}}\Omega_A({\mathcal U})\right)^{\perp_{L^2}}$, we have $\alpha=\alpha^\perp$, which vanishes by 
the last equation above. Substituting this in the second equation, we get $\psi=0$ on $\partial D$, and finally, from the first equation with this
boundary condition, $\psi=0$. Thus $(\psi,\alpha)=(0,0)$, as claimed, and we have established transversality.
\end{proof}

\vspace{0.1in}

Consider the linearization of $\bar{\partial}$ at $(f, \Sigma) \in \mathcal{F}_A$ with
respect to any $J$-parallel connection on $H^{1+s}({\mathcal F}_A)$:
\[
\nabla_{(f, \Sigma)} \bar{\partial}: H^{1+s}(f^*TTS^2 \otimes f^*T \Sigma) \to H^s(f^*TTS^2).
\]
The following key points about this operator are standard:
\vspace{0.1in}

\begin{Prop}\cite{oh}\label{p:oh}
$\nabla_{(f,\Sigma)}\bar{\partial}$ is Fredholm and therefore has finite dimensional
kernel and cokernel. The analytic index of this operator
$I={\mbox{dim ker}}(\nabla_{(f, \Sigma)} \bar{\partial}) -{\mbox{dim
coker}}(\nabla_{(f, \Sigma)} \bar{\partial})$ is related to
the Keller-Maslov index of the edge by
\[
I=\mu(TS^2,T\Sigma)+2.
\]
\end{Prop}
                    
\vspace{0.1in} 

\begin{Prop} \cite{oh}\label{p:fredholm}
There exists a dense open set ${\mathcal W}\subset{\mathcal V}$ containing $\Sigma_1$ such that any 
(not multiply covered) holomorphic disc
with edge in $\Sigma\in{\mathcal W}$ is Fredholm-regular i.e. 
${\mbox{dim coker}}(\nabla_{(f, \Sigma)} \bar{\partial})=0$.
\end{Prop}
\begin{proof}
The proof follows from ellipticity of $\bar{\partial}$, the fact 
that the projection map $\pi$ is Fredholm and the Sard-Smale theorem. 

\end{proof}

\vspace{0.1in}

\subsection{Proof of the Conjecture}\label{s:bvp}

We now prove that it is not possible for a closed $C^{3+\alpha}$ strictly convex surface in ${\mathbb E}^3$ to have only one umbilic point. By
the reformulation in section \ref{s:reform}, this is equivalent to showing that it is not possible for the graph of a global $C^{2+\alpha}$ 
Lagrangian section of $\pi:TS^2\rightarrow S^2$ to have a single complex point. 

Our proof goes as follows. First, by Proposition \ref{p:banach}, should such a section exist, call it $\Sigma_1$, then 
$\Sigma_1$ lies in an open 
subset of the $C^{2+\alpha}$ Lagrangian sections passing through a single fixed point of $TS^2$. This neighbourhood ${\mathcal U}$ is a 
Banach manifold modeled on the 
Lagrangian vector fields which arise as ${\mathbb J}$ times the tangent space to $\Sigma_1$. 

Moreover, this model is exactly the one required for a surface
to be a good boundary condition for the Cauchy-Riemann operator. That is, Proposition \ref{p:banach} establishes that, 
passing to a smaller neighbourhood ${\mathcal V}$, the space of holomorphic discs 
with edge lying on points of ${\mathcal V}$ is a submanifold of the space of all smooth maps.

Since the Cauchy-Riemann operator with this boundary condition is Fredholm, an application of the Sard-Smale theorem in Proposition \ref{p:fredholm} 
proves that for a dense open subset ${\mathcal W}$ of ${\mathcal V}$ the Cauchy-Riemann operator with this boundary condition is surjective. We conclude
that the dimension of the space of parameterized holomorphic discs with edge on a section in ${\mathcal W}$ is equal to the analytic index of the 
operator.  By Proposition \ref{p:oh} this index is related to the Keller-Maslov index of the edge curve by  $I=\mu+2$.

Consider a graphical holomorphic disc with edge lying on the fixed section $\Sigma\in{\mathcal W}$. The Keller-Maslov index of the edge is the relative first 
Chern class of the edge which, for graphs over the same domain, counts the number of complex points inside the edge on $\Sigma$ - 
see Proposition \ref{p:index}. 

Suppose that the disc in $\Sigma$ bounded by the edge of the holomorphic disc is totally real, so that $\mu=0$. Thus $I=\mu+2=2$, and, quotienting by 
the M\"obius group of the  disc, the space of unparameterised holomorphic discs is $I-3=2-3=-1$. This means that, for $C^{2+\alpha}$ boundary conditions 
$\Sigma\in{\mathcal W}$, there cannot exist a holomorphic disc with edge on $\Sigma$.

In particular, as every $\Sigma\in{\mathcal W}$ contains a totally real Lagrangian hemisphere, we conclude that, should $\Sigma_1$ exist, then so would 
a totally real Lagrangian hemisphere over which it is not possible to attach a holomorphic disc.

However, we prove that a graphical holomorphic disc always exists with edge lying on any $C^{2+\alpha}$ 
totally real Lagrangian hemisphere. The construction of the holomorphic disc is carried out by mean curvature flow. In particular, 
we consider the following initial boundary value problem:

\begin{center}\fbox{\parbox{4.8in}{
\begin{center}{\Large{ \bf I.B.V.P.}}\end{center}
{\it
Consider a family of positive sections $f_s:D\rightarrow TS^2$ such that
\[
\frac{d f}{ds}^\bot=H,
\]
with initial and boundary conditions:
\vspace{0.1in}
\begin{enumerate}
\item[(i)] $f_0(D)=\Sigma_0,$
\item[(ii)]$f_s(\partial D)\subset \tilde{\Sigma}$,
\item[(iii)] the hyperbolic angle $B$ between $Tf_s(D)$ and $T\tilde{\Sigma}$ is constant along $f_s(\partial D)$,
\item[(iv)] $f_s(\partial D)$ is asymptotically holomorphic: $|\bar{\partial}f_s|=C/(1+s)$,
\end{enumerate} 
\vspace{0.1in}
where $H$ is the mean curvature vector of $f_s(D)$, and $\Sigma_0$ and
$\tilde{\Sigma}$ are given positive sections.
\vspace{0.1in}
}
}
}
\end{center}
\vspace{0.2in}

Here positive means spacelike: the induced metric is positive definite. In what follows we refer to 
$f_s(\partial D)$ as the {\it edge} of the flowing surface, which lies on the boundary surface $\tilde{\Sigma}$.

This is a quasilinear parabolic system and short time existence for the flow is established in Theorem \ref{t:ste}. The proof 
consists of checking the Lopatinskii-Shapiro conditions for the boundary conditions and then using standard Schauder theory. 

Long time existence under certain conditions is then proven in Theorem \ref{t:lte}. In particular, having added a sufficiently large
holomorphic twist to the boundary surface to make it positive at the origin, we can choose an initial surface and hyperbolic angle
so that the flow exists for all time. This is proven as follows.

Long-time existence is  ensured by uniform positivity of the flowing disc. Uniform positivity in the interior of the flowing disc
is established by showing that the 
conditions required in the compact case, Theorem \ref{t:indflow}, namely the timelike convergence condition and containment in a compact region, 
hold for this flow. We then establish uniform positivity at the edge by careful consideration of the boundary conditions and finding a priori
bounds on the 2-jet of the flowing surface. 

It is at this juncture that our differentiability assumption enters. Derivatives of the boundary conditions relate second derivatives of
the flowing surface with those of the boundary surface. As a consequence, we must assume that the boundary surface is at least 
$C^{2+\alpha}$-smooth as a section in $TS^2$, or that the underlying surface in ${\mathbb E}^3$ is at least $C^{3+\alpha}$-smooth.

We prove in Theorem \ref{t:ashol} that, given a non-umbilic hemisphere $S$ in ${\mathbb E}^3$, then there exists a 
holomorphic disc in $TS^2$ whose edge lies on the set of oriented normals of $S$, considered as a surface $\Sigma$ in $TS^2$.
 
In summary, if there exists a $C^{2+\alpha}$ Lagrangian section with exactly one complex point, then there must exist totally real $C^{2+\alpha}$
Lagrangian hemispheres which cannot be the boundary condition for a holomorphic disc. However, we prove that one can attach a 
holomorphic disc to any $C^{2+\alpha}$ totally real Lagrangian hemisphere. 
Thus a $C^{2+\alpha}$ Lagrangian section with one complex point does not exist.

Noting the jump in derivative in the passage from points in ${\mathbb E}^3$ to points in $TS^2$, we conclude that a closed convex $C^{3+\alpha}$-smooth 
surface in ${\mathbb E}^3$ cannot have just one umbilic point.

\vspace{0.2in}


\section{{\bf Mean Curvature Flow of Compact Spacelike Submanifolds}}\label{s:compactmcf}

In this section we establish a long-time existence result for mean
curvature flow of compact spacelike surfaces in indefinite manifolds. Throughout we utilize the summation convention on repeated indices, except for the
quantity $\psi_\alpha$, defined below. In some instances we include summation signs for clarity. Note that raising and lowering normal indices
(Greek indices) changes the sign of the component, while raising and lowering tangent indices (Latin indices) does not change the sign.

\subsection{Immersed spacelike submanifolds}\label{s:setting}

Let ${\mathbb M}$ be an $n+m$-dimensional manifold endowed with a metric ${\mathbb G}$ of signature ($n,m$).
We assume throughout that there exists a {\it multi-time function} $t:{\mathbb M}\rightarrow {\mathbb R}^m$ of maximal rank with
components $t_\alpha$ for $\alpha=1...m$ such that
\[
{\mathbb G}(\overline{\nabla}t_\alpha,\overline{\nabla}t_\alpha)<0 \qquad\qquad \forall \alpha=1...m,
\]
and $\{\overline{\nabla}t_\alpha\}_1^m$ form a mutually orthogonal basis for a spacelike plane,
where all geometric quantities associated with ${\mathbb G}$ will be denoted with a bar.

\begin{Def}
The manifold $({\mathbb M},{\mathbb G})$ is said to satisfy the {\it timelike curvature condition} if, for any spacelike $n$-plane 
$P$ at a point in ${\mathbb M}$, the Riemann curvature tensor satisfies
\begin{equation}\label{e:tcc}
{\mathbb G}(\overline{R}(X,\tau_i)X,\tau_i)\;\geq k \;{\mathbb G}(X,X),
\end{equation}
for some positive constant $k$, where $\{\tau_i\}_{i=1}^n$ form an orthonormal basis for $P$ and $X$ is any timelike vector orthogonal to $P$.
\end{Def}

\begin{Note}
This generalises the timelike convergence condition of the codimension one case employed in \cite{EaH}:
\[
\overline{R}ic(X,X)\geq 0.
\]
\end{Note}

\vspace{0.1in}

We fix an orthonormal frame on (${\mathbb M},{\mathbb G}$):
\[
\{e_i,T_\alpha\}_{i,\alpha=1}^{n,m} \qquad \mbox{ s.t.}\qquad {\mathbb G}(e_i,e_j)=\delta_{ij}
\qquad {\mathbb G}(T_\alpha,T_\beta)=-\delta_{\alpha\beta} \qquad {\mathbb G}(e_i,T_\alpha)=0,
\]
with
\[
T_\alpha=-\psi_\alpha {\mathbb G}(\overline{\nabla}t_\alpha,\cdot)
\qquad\qquad \psi_\alpha^{-2}=-{\mathbb G}(\overline{\nabla}t_\alpha,\overline{\nabla}t_\alpha).
\]

\begin{Def}\label{d:norm}
Given a contravariant tensor $B$ on ${\mathbb M}$ we define its norm by
\[
\|B\|^2=\sum_{i_1,...i_l=1}^n [B(e_{i_1},e_{i_2},...,e_{i_l})]^2+\sum_{\beta_1,...\beta_l=1}^m [B(T_{\beta_1},T_{\beta_2},...,T_{\beta_l})]^2.
\]
Similarly, for a covariant tensor $B$ we dualise it with the metric ${\mathbb G}$ and define its norm as above.

Higher derivative norms are also defined:
\[
\|B\|^2_k=\sum_{j=0}^k\|\overline{\nabla}^jB\|^2.
\]

For a mixed tensor, we occasionally use the induced metric on the spacelike components to define a norm on the timelike components. That is,
if $B_{\alpha\beta ijk}$ is a tensor of the indicated type, then we define
\[
|B_{\alpha\beta}|^2=\sum_{i=1}^n [B_{\alpha\beta}(e_i,e_i,e_i)]^2.
\]
 
\end{Def}

Let $f:\Sigma\rightarrow {\mathbb M}$ be a spacelike immersion of an
$n$-dimensional manifold $\Sigma$, and let $g$ be the metric induced
on $\Sigma$ by ${\mathbb G}$. 

\begin{Def}\label{d:adptframe}
A second orthonormal basis $\{\tau_{i},\nu_\alpha\}$ for (${\mathbb
M},{\mathbb G}$) along $\Sigma$ is {\it adapted} to the
submanifold if:
\[
\{\tau_i,\nu_\alpha\}_{i,\alpha=1}^{n,m} \qquad \mbox{ s.t.}\qquad {\mathbb G}(\tau_i,\tau_j)=\delta_{ij}
\qquad {\mathbb G}(\nu_\alpha,\nu_\beta)=-\delta_{\alpha\beta} \qquad {\mathbb G}(\tau_i,\nu_\alpha)=0,
\]
where $\{\tau_i\}_{i=1}^{n}$ form an orthonormal basis for
($\Sigma,g$), and $\{\nu_\alpha\}_{\alpha=1}^m$ span the normal space.

\end{Def}

The {\it second fundamental form} of the immersion is
\[
A_{ij\alpha}={\mathbb G}(\overline{\nabla}_{\tau_i}\nu_\alpha,\tau_j)=-{\mathbb G}(\overline{\nabla}_{\tau_i}\tau_j,\nu_\alpha),
\]
while the {\it mean curvature vector} is
\[
H_\alpha=g^{ij}A_{ij\alpha}.
\]
We have the following two equations for the splitting of the
connection
\begin{equation}\label{e:connsplit1}
\overline{\nabla}_{\tau_i}\tau_j=\nabla^\parallel_{\tau_i}\tau_j-A_{ij}^\alpha\nu_\alpha
\end{equation}
\begin{equation}\label{e:connsplit2}
\overline{\nabla}_{\tau_i}\nu_\alpha=A^j_{i\alpha}\tau_j+C_{i\alpha}^\beta\nu_\beta,
\end{equation}
where $C_{i\alpha}^\beta$ are the components of the normal
connection
\[
\nabla^\bot_{\tau_i}\nu_\alpha=C_{i\alpha}^\beta\nu_\beta.
\]

\subsection{Multi-angles}\label{s:anglesgen}

We now consider how to use orthonormal frames to define a matrix of angles between two positive $n$-planes in an $n+m$-manifold.

For frames $\{e_{i},T_\alpha\}$ and $\{\tau_{i},\nu_\alpha\}$ as above, introduce the notation
\[
X_{ij}={\mathbb G}(\tau_i,e_j)
\quad
W_{i\beta}={\mathbb G}(\tau_i,T_\beta)
\quad
U_{\alpha j}=-{\mathbb G}(\nu_\alpha,e_j)
\quad
V_{\alpha\beta}=-{\mathbb G}(\nu_\alpha,T_\beta).
\]
Thus
\[
e_i=X_{ij}\tau_j+U_{\alpha i}\nu_\alpha
\quad
T_\beta=W_{i\beta}\tau_i+V_{\alpha\beta}\nu_\alpha,
\]
and the $(n+m)\times (n+m)$ dimensional matrix
\[
M=\left(
     \begin{array}{cc}
      X & W \\
      -U & -V
     \end{array}\right),
\]
is an element of the orthogonal group $O(n,m)$.

\begin{Prop}
The $O(n,m)$ condition on M is
\begin{equation}\label{e:onm1a}
X^TX=I_n+U^TU
\qquad
V^TV=I_m+W^TW
\qquad
U^TV=X^TW.
\end{equation}
\end{Prop}
\begin{proof}
This follows from the requirement that
\[
M^T\left(
     \begin{array}{cc}
      I_n & 0 \\
       0 & -I_m
     \end{array}\right)M=\left(
     \begin{array}{cc}
      I_n & 0 \\
       0 & -I_m
     \end{array}\right).
\]
\end{proof}

The vectors $\{\tau_i\}_1^n$ span the tangent space of $\Sigma$, while $\{\nu_\alpha\}_1^m$ span the normal bundle. We are free to rotate these frames
within these two spaces, and this corresponds to left action of $O(n)$ and $O(m)$ within $O(n,m)$.

Similarly, we consider rotations of $\{e_i\}_1^n$
that preserve the $n$-dimensional vector space that they span, along with rotations of $\{T_\beta\}_1^m$ 
that preserves the $m$-dimensional space they span.
These correspond to right actions of $O(n)$ and $O(m)$ within $O(n,m)$. Note that the positive definite norm in Definition \ref{d:norm} is preserved by
these rotations.

\begin{Prop}\label{p:gauge}
By rotations of the frames $\{e_i,T_\alpha\}$ and $\{\tau_j,\nu_\beta\}$ which preserve the tangent and normal bundles of $\Sigma$, as well as the
tensor norm of Definition \ref{d:norm}, we can simplify the matrix $M\in O(n,m)$ for $n\geq m$ to
\[
M=\left(
     \begin{array}{ccc}
      I_{n-m} & 0 & 0\\
       0 &  D_1 & \pm D_4A^T \\
       0 & D_3A & D_2
     \end{array}\right),
\]
where $A\in O(m)$ is a transposition matrix, $D_1$, $D_2$, $D_3$ and $D_4$ are diagonal matrices satisfying
\[
D_1^2=I_m+D_3^2
\qquad
D_2^2=I_m+D_4^2
\qquad
|D_1|^2=|D_2|^2,
\]
and $\pm$ of a diagonal matrix means a free choice of sign on the entries of the matrix. 
\end{Prop}
\begin{proof}
Consider first the matrix $X_{ij}=<\tau_i,e_j>$. The matrix $X^TX$
is symmetric and non-negative definite and so it has a
well-defined square root, namely a symmetric $n\times n$ matrix which we
denote by $\sqrt{X^TX}$. By the first equation of (\ref{e:onm1a}), $X$
is invertible since det$(X)\ge 1$ and so we can define the
$n\times n$ matrix $A=\sqrt{X^TX}X^{-1}$. Then
\[
A^TI_nA=(X^{-1})^T\sqrt{X^TX}\sqrt{X^TX}X^{-1}=(X^{-1})^TX^TXX^{-1}=I_n,
\]
so that $A\in O(n)$. Define a new frame by $\{A_{ij}\tau_j,\nu_\alpha\}$ and then
\[
\tilde{X}_{ij}=A_{ik}<\tau_k,e_j>=\sqrt{X^TX}X^{-1}X=\sqrt{X^TX},
\]
which is symmetric. Now we can act on both the left and right of
$\tilde{X}$ by $O(n)$ to diagonalise it.

A similar argument yields a diagonalisation of $V_{\alpha\beta}$.

After diagonalisation of $X$, the
first of equations (\ref{e:onm1a}) implies that the matrix $U^TU$
is diagonal. Thus the $n$ $m$-dimensional vectors $\{U_{\alpha
i}\nu_\alpha\}_{i=1}^n$ are mutually orthogonal and, since $n\geq
m$, we conclude that $n-m$ of these vectors must be zero.

After a reordering of the basis elements, the matrix $M$ then decomposes into
\[
M=\left(
     \begin{array}{ccc}
      I_{n-m} & 0 & W_2\\
       0 & X_1 & W_1 \\
       0 & U_1 & V
     \end{array}\right).
\]
The last of equations (\ref{e:onm1a}) now implies that $W_2=0$ and
we reduce the problem to the square case:
\[
X_1^TX_1=I_m+U_1^TU_1
\qquad
V^TV=I_m+W_1^TW_1
\qquad
U_1^TV=X_1^TW_1.
\]

In fact, to indicate that $X^1$ and $V$ are diagonal, let us write $X_1=D_1$ and $V=D_2$. Thus
\begin{equation}\label{e:onm1}
D_1^2=I_m+U_1^TU_1,
\end{equation}
\begin{equation}\label{e:onm2}
D_2^2=I_m+W_1^TW_1,
\end{equation}
\begin{equation}\label{e:onm3}
U_1^TD_2=D_1W_1.
\end{equation}
Equations (\ref{e:onm1}) and (\ref{e:onm2}) imply that there
exists diagonal matrices $D_3$ and $D_4$ (with entries defined up
to a sign) such that
\[
U_1=D_3A \qquad W_1=D_4B \qquad {\mbox{ for }}A,B\in O(m).
\]
Thus equations (\ref{e:onm1}), (\ref{e:onm2}) and (\ref{e:onm3}) now read
\begin{equation}\label{e:onm4}
D_1^2=I_m+D_3^2,
\end{equation}
\begin{equation}\label{e:onm5}
D_2^2=I_m+D_4^2,
\end{equation}
\begin{equation}\label{e:onm6}
A^TD_2D_3=D_1D_4B.
\end{equation}
Taking the transpose of this last equation and multiplying back on the right we find that
\begin{equation}\label{e:onm7}
A^TD_2^2D_3^2A=D_1^2D_4^2.
\end{equation}
However, if $A\in O(m)$ sends a diagonal matrix to a diagonal matrix, then $A$ must be a transposition. Similarly
\[
B^TD_1^2D_4^2B=D_2^2D_3^2,
\]
and so $A=\pm B^T$.

Denote the  diagonal elements of $D_1$, $D_2$, $D_3$ and $D_4$ by $\lambda_i$, $\mu_i$, $a_i$ and $b_i$, respectively, where $i=n-m+1,...n$.
Then equations (\ref{e:onm4}), (\ref{e:onm5}) and (\ref{e:onm7}) read
\[
\lambda_i^2=1+a_i^2
\qquad
\mu_i^2=1+b_i^2
\qquad
\mu_i^2a_i^2=\lambda_{p(i)}^2b_{p(i)}^2,
\]
where $p$ is the permutation of $(n-m+1,...n)$ determined by the transposition A. Combining these three equations we get
\[
a_i^2+a_i^2b_i^2=a_{p(i)}^2+a_{p(i)}^2b_{p(i)}^2,
\]
which when summed yields
\[
\sum_ia_i^2=\sum_ib_i^2
\qquad\quad
{\mbox{and}}
\qquad\quad
\sum_i\lambda_i^2=\sum_i\mu_i^2.
\]
Thus $|D_1|^2=|D_2|^2$ as claimed.

\end{proof}

\vspace{0.1in}

\begin{Def}
The function $v$ is defined to be
\[
v^2=V^{\alpha\beta}V_{\alpha\beta}.
\]
This is a generalization of the {\it tilt function} in the codimension one case \cite{Bart}.
\end{Def}

\vspace{0.1in}

We now use the normal form to construct estimates of the norms of the adapted frames:

\vspace{0.1in}

\begin{Prop}\label{p:frameest}
For an adapted frame $\{\tau_i,\nu_\alpha\}$ we have
\[
\|\tau_i\|^2\leq [n^2+m(m-1)^2] v^2 \qquad \|\nu_\alpha\|^2\leq m^3v^2,
\]
for all $i=1,2,...n$ and $\alpha=1,2,...m$.
\end{Prop}
\begin{proof}
First consider an adapted frame $\{\mathring{\tau}_i,\mathring{\nu}_\alpha\}$ for which, with respect to an orthonormal background  basis 
$\{\mathring{e}_i,\mathring{T}_\alpha\}$, the matrix $M$ has the form given in
Proposition \ref{p:gauge}. For a general adapted frame $\{\tau_i,\nu_\alpha\}$
\[
\tau_i=A_i^j\mathring{\tau}_j
\qquad
\nu_\alpha=B_\alpha^\beta\mathring{\nu}_\beta,
\]
where $A\in O(n)$ and $B\in O(m)$. Then
\begin{align}
\|\tau_i\|^2&=\sum_j({\mathbb G}(\tau_i,\mathring{e}_j))^2+\sum_\alpha({\mathbb G}(\tau_i,\mathring{T}_\alpha))^2\nonumber\\
&=\sum_j\left[\sum_kA_i^k{\mathbb G}(\mathring{\tau}_k,\mathring{e}_j)\right]^2
    +\sum_\alpha\left[\sum_kA_i^k{\mathbb G}(\mathring{\tau}_k,\mathring{T}_\alpha)\right]^2 \nonumber\\
&\leq\sum_j\left[\sum_k|A_i^k|\;|{\mathbb G}(\mathring{\tau}_k,\mathring{e}_j)|\right]^2
    +\sum_\alpha\left[\sum_k|A_i^k|\;|{\mathbb G}(\mathring{\tau}_k,\mathring{T}_\alpha)|\right]^2\nonumber\\
&\leq\sum_j\left[\sum_k\;|{\mathbb G}(\mathring{\tau}_k,\mathring{e}_j)|\right]^2
    +\sum_\alpha\left[\sum_k\;|{\mathbb G}(\mathring{\tau}_i,\mathring{T}_\alpha)|\right]^2\nonumber\\
&\leq\sum_j\left[\sum_k\;|X_{kj}|\right]^2+\sum_\alpha\left[\sum_k\;|W_{k\alpha}|\right]^2\nonumber\\
&\leq\left[\sum_k\;|X_{kk}|\right]^2+\sum_\alpha\left[\sum_{k=n-m+1}^n\;|W_{k\alpha}|\right]^2\nonumber\\
&\leq n^2 v^2+m(m-1)^2v^2\nonumber\\
&\leq [n^2+m(m-1)^2] v^2\nonumber.
\end{align}

Similarly for $\nu_\alpha$.

\end{proof}
\vspace{0.1in}

\subsection{The height functions}
Let $u_\alpha:\Sigma\rightarrow {\mathbb R}$ be the {\it height function} $u_\alpha=t_\alpha\circ f$. Then
\begin{Prop}\label{p:slice}
For all $\alpha=1...m$
\[
\nabla u_\alpha=\overline{\nabla}t_\alpha+\psi^{-1}_\alpha\sum_\beta V_{\beta\alpha}\nu_\beta,
\]
\[
\nabla u_\alpha\cdot\nabla u_\beta=\psi^{-1}_\alpha\psi^{-1}_\beta\left(\sum_\gamma V_{\gamma\alpha}V_{\gamma\beta}-\delta_{\alpha\beta}\right).
\]
\end{Prop}
\begin{proof}
From the definition of $u_\alpha$ and $T_\alpha$ we have 
\[
\nabla u_\alpha=\overline{\nabla}t_\alpha+\psi^{-1}_\alpha\sum_\beta V_{\beta\alpha}\nu_\beta
=\psi^{-1}_\alpha\left(\sum_\beta V_{\beta\alpha}\nu_\beta-T_\alpha\right),
\]
and so
\begin{align}
\nabla u_\alpha\cdot\nabla u_\beta=&\psi^{-1}_\alpha\psi^{-1}_\beta\;{\mathbb G}\left(\sum_\gamma V_{\gamma\alpha}\nu_\gamma-T_\alpha,
                      \sum_\delta V_{\delta\beta}\nu_\delta-T_\beta\right)\nonumber\\
&=\psi^{-1}_\alpha\psi^{-1}_\beta\left(\sum_\gamma V_{\gamma\alpha}V_{\gamma\beta}-\delta_{\alpha\beta}\right)\nonumber.
\end{align}
as claimed.
\end{proof}

\begin{Prop}\label{p:laps}
\[
\triangle u_\gamma=-\psi_\gamma^{-1}V_{\alpha\gamma}H^\alpha+g^{ij}\overline{\nabla}_i\overline{\nabla}_jt_\gamma.
\]
\begin{align}
\triangle V_{\alpha\beta}=& V_{\gamma\beta}(A_{ij\gamma}A^{ij}_\alpha-<\overline{\mbox {R}}(\tau_i,\nu_\gamma)\tau_i,\nu_\alpha>)
      -\tilde{\nabla}_{T_\beta}H_\alpha-A^{ij}_\alpha T_\beta(g_{ij})\nonumber\\
&\qquad +{\textstyle{\frac{1}{2}}}(\overline{\nabla}{\mathcal{L}}_{T_\beta}{\mathbb G})(\nu_\alpha,\tau_i,\tau_i)
    -(\overline{\nabla}{\mathcal{L}}_{T_\beta}{\mathbb G})(\tau_i,\nu_\alpha,\tau_i)-(\overline{\nabla}T_\beta)(H,\nu_\alpha)\nonumber\\
&\qquad -2C_{i\alpha}^{\;\;\;\;\gamma}<\nu_\gamma,\overline{\nabla}_{T_\beta}\tau_i>+(\nabla_i C_{i\alpha}^{\;\;\;\;\gamma}
    +C_{i\alpha}^{\;\;\;\;\delta}C_{i\delta}^{\;\;\;\;\gamma})V_{\gamma\beta}\nonumber,
\end{align}
where $\triangle$ is the Laplacian of the induced metric $\triangle =g^{ij}\nabla_i\nabla_j$.
\end{Prop}
\begin{proof}
The first statement follows from a straightforward generalization of the codimension one case \cite{EaH}.

For the second statement we follow Bartnik \cite{Bart} and fix a point $p\in\Sigma$ and choose an orthonormal frame $\{\tau_i\}$ on $\Sigma$ 
such that $(\nabla_i\tau_j)(p)=0$.
Extend this frame in a neighbourhood of $\Sigma$ by ${\mathcal L}_{T_\beta}\tau_i=0$ for fixed $\beta$. Then
\begin{align}
-\triangle V_{\alpha\beta}&= \triangle <\nu_\alpha,T_\beta>\nonumber\\
&=\tau_i\tau_i<\nu_\alpha,T_\beta>\nonumber\\
&= \tau_i(<\overline{\nabla}_{\tau_i}\nu_\alpha,T_\beta>+<\nu_\alpha,\overline{\nabla}_{\tau_i}T_\beta>)\nonumber\\
&= \tau_i(A_{i\alpha}^j<\tau_j,T_\beta>+C_{i\alpha}^\gamma<\nu_\gamma,T_\beta>+<\nu_\alpha,\overline{\nabla}_{\tau_i}T_\beta>)\nonumber\\
&= <\overline{R}(\tau_i,T_\beta)\tau_i,\nu_\alpha>+<\nu_\alpha,\overline{\nabla}_{T_\beta}\overline{\nabla}_{\tau_i}\tau_i>
  +<\overline{\nabla}_{\tau_i}\nu_\alpha,\overline{\nabla}_{T_\beta}\tau_i>)\nonumber\\
&\qquad\quad +(\overline{\nabla}_{\tau_i}H_\alpha+<\overline{R}(\tau_i,\tau_j)\nu_\alpha,\tau_i>-A_{ij}^\gamma C_{i\gamma}^\alpha
  +H^\gamma C_{j\gamma}^\alpha)<\tau_j,T_\beta>\nonumber\\
&\qquad\quad +C_{i\alpha}^\gamma(<\overline{\nabla}_{\tau_i}\nu_\gamma,T_\beta>+<\nu_\gamma,\overline{\nabla}_{\tau_i}T_\beta>)
  +<\nu_\gamma,T_\beta>\overline{\nabla}_{\tau_i}C_{i\alpha}^\gamma \nonumber\\
&\qquad\quad +A_{i\alpha}^j(<\overline{\nabla}_{\tau_i}\tau_j,T_\beta>+<\tau_j,\overline{\nabla}_{\tau_i}T_\beta>)\nonumber\\
&= <\overline{R}(\tau_i,\nu_\gamma)\tau_i,\nu_\alpha><\nu_\gamma,T_\beta>+<\nu_\alpha,\overline{\nabla}_{T_\beta}\overline{\nabla}_{\tau_i}\tau_i>
  +2A_{i\alpha}^j<\tau_j,\overline{\nabla}_{T_\beta}\tau_i>\nonumber\\
&\qquad\quad +2C_{i\alpha}^\gamma<\nu_\gamma,\overline{\nabla}_{\tau_i}T_\beta>+<\tau_i,T_\beta>\overline{\nabla}_iH_\alpha
  +H^\gamma C_{i\gamma\alpha}<\tau_i,T_\beta>\nonumber\\
&\qquad\quad +A_{\alpha}^{ij} A_{ij}^\gamma<\nu_\gamma,T_\beta>+C_{i\alpha}^{\gamma} C_{i\gamma}^\delta<\nu_\delta,T_\beta>
+<\nu_\gamma,T_\beta>\overline{\nabla}_{\tau_i}C_{i\alpha}^\gamma \nonumber\\
&=-V_{\gamma\beta}(A_{ij\gamma}A^{ij}_\alpha+<\overline{\mbox {R}}(\tau_i,\nu_\gamma)\tau_i,\nu_\alpha>)
      +\tilde{\nabla}_{T_\beta}H_\alpha\nonumber\\
&\qquad <\nu_\alpha,\overline{\nabla}_{T_\beta}\overline{\nabla}_i\tau_i>+A_\alpha^{ij}T_\beta<\tau_i,\tau_j>\nonumber\\
&\qquad +2C_{i\alpha}^{\;\;\;\;\gamma}<\nu_\gamma,\overline{\nabla}_{T_\beta}\tau_i>+(\nabla_i C_{i\alpha}^{\;\;\;\;\gamma}
    -C_{i\alpha}^{\;\;\;\;\delta}C_{i\delta}^{\;\;\;\;\gamma})V_{\gamma\beta}\label{e:lap1}.
\end{align}

We now use the following:

\begin{Lem}\label{l:lie}
\begin{align}
T_\beta<\tau_i,\overline{\nabla}_i\nu_\alpha>&=-<\overline{\nabla}_i\tau_i,\overline{\nabla}_{T_\beta}\nu_\alpha>+
   {\textstyle{\frac{1}{2}}}(\overline{\nabla}{\mathcal{L}}_{T_\beta}{\mathbb G})(\nu_\alpha,\tau_i,\tau_i)\nonumber\\
  &\qquad  -(\overline{\nabla}{\mathcal{L}}_{T_\beta}{\mathbb G})(\tau_i,\nu_\alpha,\tau_i)-<\overline{\nabla}_HT_\beta,\nu_\alpha>\nonumber.
\end{align}
\end{Lem}
\begin{proof}
The proof of this follows the codimension one case (Proposition 2.1 of \cite{Bart}).
\end{proof}

To complete the proof of the proposition we note that
\begin{align}
<\nu_\alpha,\overline{\nabla}_{T_\beta}\overline{\nabla}_i\tau_i>&=T_\beta<\nu_\alpha,\overline{\nabla}_i\tau_i>
     -<\overline{\nabla}_{T_\beta}\nu_\alpha,\overline{\nabla}_i\tau_i>\nonumber\\
&=-T_\beta<\overline{\nabla}_i\nu_\alpha,\tau_i>-<\overline{\nabla}_{T_\beta}\nu_\alpha,\overline{\nabla}_i\tau_i>\nonumber\\
&=-{\textstyle{\frac{1}{2}}}(\overline{\nabla}{\mathcal{L}}_{T_\beta}{\mathbb G})(\nu_\alpha,\tau_i,\tau_i)
    +(\overline{\nabla}{\mathcal{L}}_{T_\beta}{\mathbb G})(\tau_i,\nu_\alpha,\tau_i)\nonumber\\
  &\qquad  +<\overline{\nabla}_HT_\beta,\nu_\alpha>\nonumber,
\end{align}
where in the last equality we have used Lemma \ref{l:lie}.
Substituting this in equation (\ref{e:lap1}) then yields the
result.

\end{proof}

\vspace{0.1in}

\subsection{The I.V.P.}

Let $f_s:\Sigma\rightarrow{{\mathbb M}}$ be a family of compact $n$-dimensional spacelike immersed submanifold in an $n+m$-dimensional manifold 
${\mathbb M}$ with a metric
${\mathbb G}$ of signature $(n,m)$. In addition, we assume that $n\geq m$, the case $n< m$ follows by similar arguments.

Then $f_s$ moves by parameterized mean curvature flow if it satisfies the following initial value problem:

\vspace{0.2in}
\begin{center}\fbox{\parbox{4.8in}{
\begin{center}{\Large{ \bf I.V.P.}}\end{center}
\vspace{0.1in}
{\it
Let $f_s:\Sigma\rightarrow{{\mathbb M}}$ be a family of spacelike immersed submanifolds satisfying
\[
\frac{d f}{ds}=H,
\]
with initial conditions 
\[
f_0(\Sigma)=\Sigma_0
\] 
where $H$ is the mean curvature vector associated with the immersion $f_s$ in $({\mathbb M},{\mathbb G})$, and $\Sigma_0$ is some given
initial compact $n$-dimensional spacelike immersed submanifold.
}
\vspace{0.1in}
}}
\end{center}
\vspace{0.2in}

The flow of the functions $u_\gamma$ and $v$ are given by:

\begin{Prop}\label{p:flowuv}
\begin{equation}\label{e:uflow}
\left(\frac{d}{ds}-\triangle\right)u_\gamma=-g^{ij}\overline{\nabla}_i\overline{\nabla}_jt_\gamma,
\end{equation}
\begin{align}
v\left(\frac{d}{ds}-\triangle\right) v\leq& -V^{\alpha\beta}V_{\gamma\beta}
      (A_{ij\gamma}A^{ij}_\alpha-<\overline{\mbox {R}}(\tau_i,\nu_\gamma)\tau_i,\nu_\alpha>)+A^{ij}_\alpha {\mathcal L}_{T_\beta}g_{ij}V^{\alpha\beta}\nonumber\\
&\qquad -{\textstyle{\frac{1}{2}}}(\overline{\nabla}{\mathcal{L}}_{T_\beta}{\mathbb G})(\nu_\alpha,\tau_i,\tau_i)V^{\alpha\beta}
    +(\overline{\nabla}{\mathcal{L}}_{T_\beta}{\mathbb G})(\tau_i,\nu_\alpha,\tau_i)V^{\alpha\beta}\nonumber\\
&\qquad +2C_{i\alpha}^{\;\;\;\;\gamma}<\nu_\gamma,\overline{\nabla}_{T_\beta}\tau_i>V^{\alpha\beta}
    -C_{i\alpha}^{\;\;\;\;\delta}C_{i\delta}^{\;\;\;\;\gamma}V_{\gamma\beta}V^{\alpha\beta}\nonumber\label{e:muflow}.
\end{align}
\end{Prop}
\begin{proof}
Generalizing Proposition 3.1 of \cite{EaH}, note the time derivatives are
\[
\frac{du_\gamma}{ds}=-\psi_\gamma^{-1}V_{\alpha\gamma}H^\alpha,
\]
\[
\frac{dV_{\alpha\beta}}{ds}=-\overline{\nabla}_{T_\beta}H_\alpha-H^\gamma<\overline{\nabla}_{\nu_\gamma}T_\beta,\nu_\alpha>.
\]
The flow of $u_\gamma$ then follows immediately from Proposition \ref{p:laps}.

To find the flow of $v$ note that
\[
v\left(\frac{d}{ds}-\triangle\right)v=V^{\alpha\beta}\left(\frac{d}{ds}-\triangle\right)V_{\alpha\beta}
  +\frac{1}{v^2}\left[(v_\alpha\nabla V_\alpha)\cdot (v_\beta\nabla V_\beta)-V_\alpha^2|\nabla V_\beta|^2 \right],
\]
where we sum over $\alpha$ and $\beta$ and diagonalised $V_{\alpha\beta}={\mbox{diag}}(V_1,...V_m)$. By the Cauchy-Schwarz inequality we have
\[
v\left(\frac{d}{ds}-\triangle\right)v\leq V^{\alpha\beta}\left(\frac{d}{ds}-\triangle\right)V_{\alpha\beta}.
\]
Now contracting the second equation of Proposition \ref{p:laps} with $V^{\alpha\beta}$ yields the result.
\end{proof}

\vspace{0.1in}

\begin{Prop}\label{p:gradest}
Assume that ${\mathbb M}$ satisfies the timelike curvature condition (\ref{e:tcc}). Let $\Sigma_s$ be a smooth solution of {\bf I.V.P.} on the
interval $0\leq s<s_0$ such that $\Sigma_s$ is contained in a
compact subset of ${\mathbb M}$ for all $0\leq s<s_0$. Then
the function $ v$ satisfies the a priori estimate
\[
 v(p,s)\leq(m+\sup_{\Sigma\times 0} v)\sup_{(q,s)\in \Sigma\times[0,s_0]}\exp[K(u(q,s)-u(p,s))],
\]
for some positive constant K$(n,m,\|t\|_3,|\psi|,\|\overline{R}\|,|H|,k)$, where $u=\sum_\alpha u_\alpha$.
\end{Prop}
\begin{proof}
The argument is an extension of Bartnik's estimate in the stationary case \cite{Bart} to the parabolic case with higher codimension.

Let K$>$0 be a constant to be determined later and set
\[
C_K=(m+\sup_{\Sigma\times 0} v)\sup_{\Sigma\times[0,s_0]}\exp(Ku).
\]
Consider the test function $h= v \exp(Ku)$. Suppose, for the sake of
contradiction, that the function $h$ reaches $C_K$ for the first
time at $(p_1,s_1)\in \Sigma\times(0,s_0]$. Then at this point $v\geq
m+1$ and by the maximum principle
\[
\left(\frac{d}{ds}-\triangle\right) h \stackrel{\cdot}{\geq}0
\qquad\qquad \nabla h \stackrel{\cdot}{=}0.
\]
Here and throughout a dot over an inequality or equality will refer to evaluation at the point $(p_1,s_1)$. Working out these two equations
we have
\begin{equation}\label{e:atmax1}
\left(\frac{d}{ds}-\triangle\right)  v +K v
\left(\frac{d}{ds}-\triangle\right) u-2K\nabla u\cdot\nabla v -K^2
v |\nabla u|^2\stackrel{\cdot}{\geq}0,
\end{equation}
\begin{equation}\label{e:atmax2}
\nabla  v +K v \nabla u\stackrel{\cdot}{=}0.
\end{equation}
Substituting the second of these in the first we obtain
\begin{equation}\label{e:atmax3}
K v \left(\frac{d}{ds}-\triangle\right)
u\stackrel{\cdot}{\geq}-\left(\frac{d}{ds}-\triangle\right) v -K^2
v |\nabla u|^2.
\end{equation}
Now, from Proposition \ref{p:flowuv} and the estimates in
Proposition \ref{p:frameest}
\begin{equation}\label{e:boxuest}
\left(\frac{d}{ds}-\triangle\right)
u=-g^{ij}\overline{\nabla}_i\overline{\nabla}_j t\leq
\|\overline{\nabla}_i\overline{\nabla}_j
t\|.\|\tau_i\|.\|\tau_j\|\leq C_1 v^2,
\end{equation}
where $C_1=C_1(n,m,\|t\|_2)$.

At $p_1$ we can set $C_{i\alpha}^{\;\;\;\beta}=0$ and then, Proposition \ref{p:flowuv} and the timelike curvature condition (\ref{e:tcc}) imply that
\begin{align}
v\left(\frac{d}{ds}-\triangle\right)v\stackrel{\cdot}{\leq}&
-\sum_{\alpha}V_\alpha^2|A_\alpha|^2+C_2(\|T\|_1)|A_\alpha|V_\alpha+C_3(n,m,\|T\|_2)v^4\nonumber\\
\leq&-(1-\epsilon)\sum_{\alpha}V_\alpha^2|A_\alpha|^2+C_4(\epsilon,n,m,\|T\|_2)v^4\label{e:eq11},
\end{align}
for any $\epsilon>0$. Here we have utilised the gauge choice $V_{\alpha\beta}=V_\alpha\delta_{\alpha\beta}$, summation
is over $\alpha$ and the last inequality follows from Young's:
\[
ab\leq\frac{\epsilon a^2}{2}+\frac{b^2}{2\epsilon}
\]

Now, from the Schwartz and arithmetic-geometric mean inequalities
\begin{equation}\label{e:ineq1}
\sum_{\alpha}V_\alpha^2|A_\alpha|^2\geq\sum_{\alpha}\left(1+\frac{1}{n}\right)\lambda_\alpha^2V_\alpha^2-H_\alpha^2V_\alpha^2,
\end{equation}
where $\lambda_\alpha$ is the eigenvalue of $A_{ij\alpha}$ with
the maximum absolute value, so that in an eigenframe $A_{ij\alpha}\leq|\lambda_\alpha|\delta_{ij}$.

On the other hand we compute
\[
\nabla_iV_{\alpha\beta}=-A_{i\alpha}^j<\tau_j,T_\beta>-<\nu_\alpha,\overline{\nabla}_iT_\beta>,
\]
and so
\[
v\nabla_iv=V^{\alpha\beta}\nabla_iV_{\alpha\beta}=-A_{i\alpha}^jW_{j\beta}V^{\alpha\beta}-<\nu_\alpha,\overline{\nabla}_iT_\beta>V^{\alpha\beta}.
\]
The square norm is
\begin{align}
v^2|\nabla v|^2&=v^2\nabla_iv\nabla^iv\nonumber\\
&=\left(A_{i\alpha}^jW_{j\beta}+<\nu_\alpha,\overline{\nabla}_iT_\beta>\right)
\left(A_{\gamma}^{ik}W_{k\delta}+<\nu_\gamma,\overline{\nabla}^iT_\delta>\right)V^{\alpha\beta}V^{\gamma\delta}\nonumber
\\
&=\left(A_{i\alpha}^jA_{\gamma}^{ik}W_{j\beta}W_{k\delta}+2A_{i\alpha}^jW_{j\beta}<\nu_\gamma,\overline{\nabla}^iT_\delta>
+<\nu_\alpha,\overline{\nabla}_iT_\beta><\nu_\gamma,\overline{\nabla}^iT_\delta>\right)V^{\alpha\beta}V^{\gamma\delta}\nonumber
\end{align}
Take these three summands separately, computing in eigenframes (so that $V^{\alpha\beta}=V_\alpha\delta^{\alpha\beta}$ and 
$A_{ij\alpha}\leq|\lambda_\alpha|\delta_{ij}$). The first term is
\begin{align}
A_{i\alpha}^jA_{\gamma}^{ik}W_{j\beta}W_{k\delta}V^{\alpha\beta}V^{\gamma\delta}&\leq |\lambda_\alpha\lambda_\gamma| .|W^k_{\beta}W_{k\delta}
V^{\alpha\beta}V^{\gamma\delta}|\nonumber\\
&=|\lambda_\alpha\lambda_\gamma| .|\left(V^\rho_{\beta}V_{\rho\delta}-\delta_{\beta\gamma}\right)V^{\alpha\beta}V^{\gamma\delta}|\nonumber\\
&=\sum_{\alpha}\lambda^2_\alpha\left(V_{\alpha}^2-1\right)V^2_{\alpha}\nonumber\\
&=v^2\sum_{\alpha}\lambda^2_\alpha V^2_{\alpha}\nonumber
\end{align}
where we have used the relationship between the matrices $W$ and $V$ given in the middle of equations (\ref{e:onm1a}). Note that this equation implies
$\|W_\beta\|^2=V_\beta^2-1\leq v^2-1\leq v^2$.

For the second term
\begin{align}
2A_{i\alpha}^jW_{j\beta}<\nu_\gamma,\overline{\nabla}^iT_\delta>V^{\alpha\beta}V^{\gamma\delta}
&\leq 2|\lambda_{\alpha}|.|W_{i\beta}<\nu_\gamma,\overline{\nabla}^iT_\delta>V^{\alpha\beta}V^{\gamma\delta}|\nonumber\\
&=2\sum_{\alpha,\gamma}|\lambda_{\alpha}|.|W_{i\alpha}<\nu_\gamma,\overline{\nabla}^iT_\gamma>|.|V_{\alpha}V_{\gamma}|\nonumber\\
&\leq 2\sum_{\alpha,\gamma}|\lambda_{\alpha}|\|W_{\alpha}\|.\|\nu_\gamma\|.\|\overline{\nabla}T_\gamma\|.|V_{\alpha}V_{\gamma}|\nonumber\\
&\leq 2m^{\scriptstyle{\frac{3}{2}}}v^2\|T\|_1\sum_{\alpha,\gamma}|\lambda_{\alpha}|.|V_{\alpha}V_{\gamma}|\nonumber
\end{align}
where we use $\|W_\beta\|^2\leq v^2$ and  $\|\nu_\gamma\|^2\leq m^3v^2$ from Proposition \ref{p:frameest}.

For each $\alpha$ we use Young's inequality with $a=v\lambda_{\alpha}|V_{\alpha}|$ and $b=m^{\scriptstyle{\frac{3}{2}}}v\|T\|_1\sum_\gamma |V_{\gamma}|$ to conclude the second estimate
\[
2A_{i\alpha}^jW_{j\beta}<\nu_\gamma,\overline{\nabla}^iT_\delta>V^{\alpha\beta}V^{\gamma\delta}
\leq\epsilon\sum_\alpha v^2\lambda^2_{\alpha}V^2_{\alpha}+m^5\epsilon^{-1}\|T\|_1^2v^4
\]
The final term is easily estimated in a similar manner
\[
<\nu_\alpha,\overline{\nabla}_iT_\beta><\nu_\gamma,\overline{\nabla}^iT_\delta>V^{\alpha\beta}V^{\gamma\delta}\leq C_5(m,\|T\|_1)v^4
\]
Putting these last three estimates together and canceling the $v^2$ factor we bound the square norm:
\[
|\nabla v|^2\leq(1+\epsilon)\sum_{\alpha}V_\alpha^2\lambda_\alpha^2+C_6(\epsilon,m,\|T\|_1)v^2.
\] 
or, rearranging
\begin{equation}\label{e:ineq2a}
\sum_{\alpha}V_\alpha^2\lambda_\alpha^2\geq \frac{1}{1+\epsilon}|\nabla v|^2-C_6v^2.
\end{equation}
Combining inequalities (\ref{e:ineq1}) and (\ref{e:ineq2a}) we get
\[
\sum_{\alpha}V_\alpha^2|A_\alpha|^2\geq\left(1+\frac{1}{n}\right)\left[\frac{1}{1+\epsilon}|\nabla
v|^2-C_6v^2\right]-\sum_{\alpha}H_\alpha
^2V_\alpha^2,
\]
which, when substituted in inequality (\ref{e:eq11}) yields
\[
v\left(\frac{d}{ds}-\triangle\right)v\stackrel{\cdot}{\leq}
-\left(1+\frac{1}{n}\right)\frac{1-\epsilon}{1+\epsilon}|\nabla v|^2+C_7(\epsilon,n,m,|H|,\|T\|_1)v^2+C_4v^4,
\]
and, by virtue of equation (\ref{e:atmax2}),
\[
|\nabla v|^2\stackrel{\cdot}{=}K^2v^2|\nabla u|^2,
\]
yielding
\begin{equation}\label{e:boxvest}
\left(\frac{d}{ds}-\triangle\right)v\stackrel{\cdot}{\leq}
-\left(1+\frac{1}{n}\right)\frac{1-\epsilon}{1+\epsilon}K^2v|\nabla u|^2+C_7v+C_4v^3.
\end{equation}
Substituting inequalities (\ref{e:boxuest}) and (\ref{e:boxvest}) in (\ref{e:atmax3}) we get
\[
KC_1v^2\stackrel{\cdot}{\geq} \left[\left(1+\frac{1}{n}\right)\frac{1-\epsilon}{(1+\epsilon)}-1\right]K^2|\nabla u|^2-C_7-C_4v^2,
\]
for any $\epsilon>0$.

Now for $0<\epsilon<1/(1+2n)$
\[
\left(1+\frac{1}{n}\right)\frac{1-\epsilon}{1+\epsilon}-1>0,
\]
and so using Proposition \ref{p:slice}
\[
|\nabla u|^2=\sum_{\alpha,\beta}\nabla u_\alpha\cdot\nabla u_\beta \geq{\mbox{min}}_\alpha\psi_\alpha^{-2}(v^2-m),
\]
we have
\[
KC_1v^2\stackrel{\cdot}{\geq} C_8(\epsilon,n,|\psi|)K^2(v^2-m)-C_7-C_4v^2,
\]
which can be rearranged to
\[
v^2\stackrel{\cdot}{\leq}\frac{mC_8K^2+C_7}{C_8K^2-C_1K-C_4},
\]
where, in summary, $C_1(n,m,\|t\|_2)$, $C_4(\epsilon,n,m,\|T\|_1)$, $C_7(\epsilon,n,m,\|t\|_2,\|T\|_1)$ and $C_8(\epsilon,n,|\psi|)$.

For large $K$ this inequality violates $ v\geq m+1$ and we have a contradiction.
\end{proof}

\vspace{0.2in}
For tensors $H_\alpha$ and $A_{ij\alpha}$ we define a positive norm by
\[
|H|_+^2=-H_\alpha H^\alpha \qquad\qquad |A|_+^2=-A_{ij\alpha}A^{ij\alpha},
\]
and similarly for their gradients.

\vspace{0.1in}

\begin{Prop}
Under the mean curvature flow, the norms of the mean curvature vector and the second fundamental form of a positive m-dimensional submanifold in an 
indefinite m+n-dimensional manifold evolve according to:
\[
\left(\frac{d}{ds}-\triangle\right)
|H|_+^2=-2|\tilde{\nabla}H|_+^2-2|H\cdot A|_+^2-2H^\alpha H^\beta\bar{R}_{i\alpha i\beta},
\]
\[
\left(\frac{d}{ds}-\triangle\right)
|A|_+^2=-2|\tilde{\nabla}A|_+^2-2|A|_+^4+A*A*\overline{R}+A*\overline{\nabla}\;\overline{R},
\]
where $\tilde{\nabla}$ is the covariant derivative in both the tangent and normal bundles and $*$ represents linear combinations of contractions
of the tensors involved.
\end{Prop}
\begin{proof}
These are proven in Proposition 4.1 of \cite{LaS}, generalizing the expressions in Proposition 3.3 of \cite{EaH}.
\end{proof}
\vspace{0.1in}

\begin{Prop}\label{p:bdsff}
Under the mean curvature flow
\[
|H|_+^2\leq C_1(1+s^{-1}),
\]
\[
|A|_+^2\leq C_2(1+s^{-1}),
\]
where $C_1=C_1(n,k)$ and $C_2=C_2(n,\|\overline{R}\|_1)$, $k$ being the constant in the timelike curvature condition (\ref{e:tcc}).
\end{Prop}
\begin{proof}
From the previous proposition and the timelike curvature condition
we conclude that
\[
\left(\frac{d}{ds}-\triangle\right) |H|_+^2\leq-2n^{-1}|H|_+^4+2k|H|_+^2,
\]
while
\[
\left(\frac{d}{ds}-\triangle\right) |A|_+^2\leq-2|A|_+^4+C_3|A|_+^2+C_4|A|_+\leq-|A|_+^4+C_5.
\]
The result then follows by a suitable modification of Lemma 4.5 of Ecker and Huisken \cite{EaH}.
\end{proof}

\vspace{0.1in}

\begin{Thm}\label{t:indflow}
Let $\Sigma_0$ be a smooth compact $n$-dimensional spacelike submanifold of an $n+m$ dimensional manifold ${\mathbb M}$ with indefinite metric 
${\mathbb G}$ satisfying the timelike curvature 
condition. Then there exists a unique family $f_s(\Sigma)$ of smooth compact $n$-dimensional spacelike submanifolds satisfying the initial value 
problem {\bf I.V.P.} on an interval $0\leq s< s_0$. Moreover, if $f_s(\Sigma)$ remains in
a smooth compact region of ${\mathbb M}$ as $s\rightarrow s_0$, the solution can be extended beyond $s_0$.
\end{Thm}
\begin{proof}
The flow is a quasilinear parabolic system and therefore short time existence follows from linear Schauder estimates and the implicit function theorem.
In the case where $n=m=2$, ${\mathbb M}=TS^2$ and we impose boundary conditions, short-time existence is established in detail
in Theorem \ref{t:ste}.

Having bounded the gradient and the second fundamental form in Propositions \ref{p:gradest} and \ref{p:bdsff}, bounds on the higher derivatives and 
long-time existence follow from standard parabolic bootstrapping arguments, as in \cite{EaH}.
\end{proof}

\vspace{0.2in}


\section{{\bf Neutral K\"ahler Geometry of $TS^2$}}\label{s:neutral}

In this section we give details of the neutral K\"ahler geometry of the space of oriented lines, paying particular attention to those
parts that play a role later.

\subsection{The neutral K\"ahler surface and submanifolds}\label{s:nkm}

\begin{Def}
A {\it neutral K\"ahler surface} is a 4-manifold ${\mathbb M}$ endowed with a complex structure ${\mathbb J}$, a symplectic structure
$\Omega$ and a metric ${\mathbb G}$ of signature $(++--)$ which are compatible in the sense that
\[
{\mathbb G}({\mathbb J}\cdot,{\mathbb J}\cdot)={\mathbb G}(\cdot,\cdot)
\qquad\qquad
{\mathbb G}({\mathbb J}\cdot,\cdot)=\Omega(\cdot,\cdot).
\]
\end{Def}
For such a structure we have the following identity:

\vspace{0.1in}

\begin{Prop}\label{p:callib}\cite{gak9}
Let $({\Bbb{M}},{\Bbb{J}},\Omega,{\Bbb{G}})$  be a neutral K\"ahler surface and let $p\in{\Bbb{M}}$ and $v_1,v_2\in T_p{\Bbb{M}}$ span a plane. Then
\[
\Omega(v_1,v_2)^2-det[v_1,v_2,{\mathbb J}v_1,{\mathbb J}v_2]={\mbox{det }}{\Bbb{G}}(v_i,v_j).
\]

\end{Prop}
\vspace{0.1in}

We turn now to the special case of $TS^2$. In order to compute geometric quantities, we introduce local coordinates. These are readily supplied 
by lifting the standard complex coordinate $\xi$ (obtained by stereographic projection from the south pole on $S^2$) to complex coordinates 
($\xi,\eta$) on $TS^2$. In particular, identify $(\xi,\eta)\in{\mathbb {C}}^2$ with the vector
\[
\eta\frac{\partial}{\partial \xi}+\bar{\eta}\frac{\partial}{\partial \bar{\xi}}\in T_\xi S^2.
\]
These coordinates are holomorphic with respect to the complex structure ${\mathbb J}$:
\[
{\mathbb J}\left(\frac{\partial}{\partial \xi}\right)=i\frac{\partial}{\partial \xi}
\qquad\qquad
{\mathbb J}\left(\frac{\partial}{\partial \eta}\right)=i\frac{\partial}{\partial \eta},
\]
and the symplectic 2-form and neutral metric have the following local expressions:
\[
\Omega=4(1+\xi\bar{\xi})^{-2}{\Bbb{R}}\mbox{e}\left(d\eta\wedge d\bar{\xi}-\frac{2\bar{\xi}\eta}{1+\xi\bar{\xi}}d\xi\wedge d\bar{\xi}\right),
\]
\begin{equation}\label{e:metric}
{\Bbb{G}}=4(1+\xi\bar{\xi})^{-2}{\Bbb{I}}\mbox{m}\left(d\bar{\eta} d\xi+\frac{2\bar{\xi}\eta}{1+\xi\bar{\xi}}d\xi d\bar{\xi}\right).
\end{equation}

\begin{Def}
The canonical coordinates $(\xi,\bar{\xi})$ are called {\it Gauss coordinates} and $R=|\xi|$ is the {\it Gauss radius}. See the comments before 
Proposition \ref{p:Gauss}.
\end{Def}

This metric on $TS^2$ is of neutral signature $(++--)$ and is conformally and scalar flat, but not K\"ahler-Einstein. It sits within a larger class of
natural scalar flat neutral metrics on $TN$, where $(N,g)$ is a Riemannian 2-manifold \cite{gak4}.   

We turn now to immersed surfaces in $TS^2$. Such 2-parameter
families of oriented lines are often referred to as line congruences. 
Consider surfaces which are graphs of local sections of the bundle $\pi:TS^2\rightarrow S^2$. 
Such local sections are given by $\xi\mapsto (\xi,\eta=F(\xi,\bar{\xi}))$, for some function $F:U\rightarrow{\mathbb C}$, for 
$U\subset{\mathbb C}$.

\vspace{0.1in}
\begin{Def}\label{d:spinco}
For a section $\eta=F(\xi,\bar{\xi})$, introduce the weighted complex slopes of $F$:
\[
\sigma=-\partial \bar{F} \qquad\qquad
\rho=\vartheta+i\lambda=(1+\xi\bar{\xi})^2\partial [F(1+\xi\bar{\xi})^{-2}].
\]
Here, and throughout, $\partial$ represents differentiation with respect to $\xi$. The functions $\lambda$ and $\sigma$ are commonly referred
to as the {\it twist} and {\it shear} of the underlying family $\Sigma$ of oriented lines in ${\mathbb E}^3$ \cite{PaR}.
\end{Def}

For later use we also introduce the notation
\[
\Delta=\lambda^2-|\sigma|^2 \qquad\qquad \mu=\frac{|\sigma|}{|\lambda|}.
\]
\vspace{0.1in}

Note the two identities, which follow from these definitions:
\begin{equation}\label{e:id1}
-(1+\xi\bar{\xi})^2\partial\left[\frac{\bar{\sigma}}{(1+\xi\bar{\xi})^2}\right]=\bar{\partial}\rho+\frac{2F}{(1+\xi\bar{\xi})^2},
\end{equation}
\begin{equation}\label{e:id2}
{\mathbb I}{\mbox{m}}\;\partial\left\{(1+\xi\bar{\xi})^2\partial\left[\frac{\bar{\sigma}}{(1+\xi\bar{\xi})^2}\right]\right\}
    =\partial\bar{\partial}\lambda+\frac{2\lambda}{(1+\xi\bar{\xi})^2}.
\end{equation}

The geometric significance of $\lambda$ and $\sigma$ are:

\vspace{0.1in}

\begin{Prop}\cite{gak2}
A surface $\Sigma$ given by a local section $\eta=F(\xi,\bar{\xi})$ is Lagrangian iff $\lambda=0$ and is holomorphic
iff $\sigma=0$.
\end{Prop}

\vspace{0.1in}

For the normal congruence
$\Sigma$ of a surface $S$ in ${\mathbb E}^3$ the following relationship between the quantities $\sigma$ and $\rho$ and the principal 
curvatures and directions of $S$ hold:

\begin{Prop}
Let $S$ be a convex surface in ${\mathbb E}^3$  and $\Sigma\subset TS^2$ be the surface formed by the 2-parameter family of oriented normals to $S$.
Then $\Sigma$ is the graph of a section and $\lambda=0$.

Moreover
\[
|\sigma|={\textstyle{\frac{1}{2}}}|\kappa_1^{-1}-\kappa_2^{-1}|
\qquad
r+\rho={\textstyle{\frac{1}{2}}}(\kappa_1^{-1}+\kappa_2^{-1}),
\]
where $\kappa_1,\kappa_2$ are the principal curvatures of $S$ and $r$ is the support function of $S$ (see below). The argument of $\sigma$
gives the principal directions of $S$.
\end{Prop}
\vspace{0.1in}

To construct the 1-parameter
family of parallel surfaces in ${\mathbb E}^3$ from a Lagrangian section consider the following.

\vspace{0.1in}

\begin{Prop}\cite{gak2}
Let $\Sigma$ be a local Lagrangian section given by $\xi\mapsto (\xi,\eta=F(\xi,\bar{\xi}))$, for some function $F:{\mathbb C}\rightarrow{\mathbb C}$. 
Then there exists a real function $\xi\mapsto r(\xi,\bar{\xi})$, satisfying
\[
\bar{\partial}r=\frac{2F}{(1+\xi\bar{\xi})^2}.
\]
Such a function is defined up to an additive real constant $C$. In terms of Euclidean coordinates the surfaces 
\begin{equation}\label{e:mt}
x^1+ix^2=\frac{2(F-\bar{F}\xi^2)+2\xi(1+\xi\bar{\xi})r}{(1+\xi\bar{\xi})^2},
\qquad
x^3=\frac{-2(F\bar{\xi}+\bar{F}\xi)+(1-\xi^2\bar{\xi}^2)r}{(1+\xi\bar{\xi})^2},
\end{equation}
are orthogonal to the oriented lines of $\Sigma$.
\end{Prop}

\vspace{0.1in}

\begin{Def}
The function $r:S^2\rightarrow{\mathbb R}$ in the previous Proposition is called the {\it support function}. Given a point $p$ on a closed convex
surface $S$, $r$ is the distance between $p$ and the point closest to the origin on the normal line through $p$.

The surfaces obtained by replacing $r$ by $r+C$ are called {\it parallel surfaces}.
\end{Def}

\vspace{0.1in}

Returning to Example \ref{ex1} we illustrate the preceding.

\begin{Ex}
Consider the section with $F=(1+\xi\bar{\xi})^2\bar{\xi}$. This is Lagrangian since
\[
\lambda={\mathbb I}m\;(1+\xi\bar{\xi})^2\partial\left(\frac{F}{(1+\xi\bar{\xi})^2}\right)=0,
\]
and the support function is easily found to be $r=\xi^2+\bar{\xi}^2+C$. 

To see that this surface has no umbilics, compute
\[
\sigma=-\partial\bar{F}=-(1+4\xi\bar{\xi}+3\xi^2\bar{\xi}^2)<0.
\]

To explicitly construct the surfaces, substitute the expressions for $F$ and $r$ in equations (\ref{e:mt}). The results are as given
in Example \ref{ex1}.

\end{Ex}

\vspace{0.1in}

For the induced metric we have:

\vspace{0.1in}

\begin{Prop}\label{p:ts2indmet}
The metric induced on the graph of a section by the K\"ahler metric is given in coordinates ($\xi,\bar{\xi}$) by;
\[
g=\frac{2}{(1+\xi\bar{\xi})^2}\left[\begin{matrix}
i\sigma & -\lambda\\
-\lambda & -i\bar{\sigma}\\
\end{matrix}
\right],
\]
with inverse
\[
g^{-1}=\frac{(1+\xi\bar{\xi})^2}{2(\lambda^2-\sigma\bar{\sigma})}\left[\begin{matrix}
i\bar{\sigma} & -\lambda\\
-\lambda & -i\sigma\\
\end{matrix}
\right].
\]
\end{Prop}
\begin{proof}
This follows from pulling back the neutral metric (\ref{e:metric}) along a local section $\eta=F(\xi,\bar{\xi})$.
\end{proof}

\vspace{0.1in}

\begin{Def}
A surface $\Sigma\subset TS^2$ is {\it positive} if the induced metric on $\Sigma$ is positive definite.
\end{Def}

\vspace{0.1in}

\begin{Prop}\label{p:indmet}
The induced metric on a Lagrangian surface is Lorentz, except at complex points, where it is degenerate. The induced metric on a holomorphic surface is 
positive, except at complex points, where it is degenerate. 
\end{Prop}
\begin{proof}
By the previous Proposition we see that the determinant of the induced metric is $2(1+\xi\bar{\xi})^{-4}(\lambda^2-\sigma\bar{\sigma})$,
and the result follows. This is, in fact, a special case of Proposition \ref{p:callib}.
\end{proof}

\vspace{0.1in}

\begin{Def}\label{d:perpdist}
Let $\chi:TS^2\rightarrow {\mathbb R}$ be the map that takes an oriented line to the square of the perpendicular distance from the line to the origin. 
\end{Def}
\vspace{0.1in}

\begin{Prop}
The coordinate expression for $\chi$ is:
\[
\chi^2(\xi,\bar{\xi},\eta,\bar{\eta})=\frac{4\eta\bar{\eta}}{(1+\xi\bar{\xi})^2}.
\]
\end{Prop}
\begin{proof}
The point on an oriented line $(\xi,\eta)$ which lies closest to the origin has Euclidean coordinates (see equation (\ref{e:mt}))
\[
x^1_0+ix^2_0=-\frac{2(\eta-\bar{\eta}\xi^2)}{(1+\xi\bar{\xi})^2},
\qquad
x^3_0=-\frac{2(\eta\bar{\xi}+\bar{\eta}\xi)}{(1+\xi\bar{\xi})^2},
\]
and so the perpendicular distance to the origin is
\[
\chi^2= (x^1_0)^2+(x^2_0)^2+(x^3_0)^2 =\frac{4\eta\bar{\eta}}{(1+\xi\bar{\xi})^2}.
\] 
\end{proof}

\vspace{0.1in}

\begin{Note}
In Section \ref{s:reform} we have seen that umbilic points on surfaces in ${\mathbb
E}^3$ give rise to complex points on Lagrangian surfaces in
$TS^2$ and now we see that these correspond to degeneracies
in the induced Lorentz metric. 
Thus the Carath\'eodory conjecture bounds the number of degenerate points on certain Lorentz surfaces, and the hyperbolic nature 
(and hence difficulty) of the problem becomes evident.
\end{Note}

\vspace{0.1in}

In order to continue, we introduce geometric tools which will prove useful later.

\vspace{0.1in}

\subsection{The second fundamental form of a positive surface}

Let $\Sigma\rightarrow$$TS^2$ be an immersed surface and assume that the induced metric on $\Sigma$ is positive,
so that for $\gamma\in\Sigma$ we have the orthogonal splitting $T_\gamma TS^2=T_\gamma \Sigma\oplus N_\gamma\Sigma$. In what follows
we omit the subscript $\gamma$.

\vspace{0.1in}

\begin{Prop}\label{p:frames}
If $\Sigma$ is a positive surface given by the graph $\xi\rightarrow(\xi,\eta=F(\xi,\bar{\xi}))$, then the following vector fields form
an orthonormal basis for $TTS^2$ along $\Sigma$:
\[
E_{(1)}=2{\mathbb R}{ e}\left[\alpha_1\left(\frac{\partial}{\partial \xi}+\partial F\frac{\partial}{\partial \eta}
          +\partial \bar{F}\frac{\partial}{\partial \bar{\eta}}    \right) \right],
\]
\[
E_{(2)}=2{\mathbb R}{ e}\left[\alpha_2\left(\frac{\partial}{\partial \xi}+\partial F\frac{\partial}{\partial \eta}
          +\partial \bar{F}\frac{\partial}{\partial \bar{\eta}}    \right) \right],
\]
\[
E_{(3)}=2{\mathbb R}{ e}\left[\alpha_2\left(\frac{\partial}{\partial \xi}
       +(\bar{\partial} \bar{F}-2(F\partial u-\bar{F}\bar{\partial} u))\frac{\partial}{\partial \eta}
          -\partial \bar{F}\frac{\partial}{\partial \bar{\eta}}    \right) \right],
\]
\[
E_{(4)}=2{\mathbb R}{ e}\left[\alpha_1\left(\frac{\partial}{\partial \xi}
       +(\bar{\partial} \bar{F}-2(F\partial u-\bar{F}\bar{\partial} u))\frac{\partial}{\partial \eta}
          -\partial \bar{F}\frac{\partial}{\partial \bar{\eta}}    \right) \right],
\]
for
\[
\alpha_1=\frac{e^{-u-{\scriptstyle \frac{1}{2}}\phi i+{\scriptstyle \frac{1}{4}}\pi i}}
      {\sqrt{2}[-\lambda-|\sigma|]^{\scriptstyle \frac{1}{2}}}
\qquad\qquad
\alpha_2=\frac{e^{-u-{\scriptstyle \frac{1}{2}}\phi i-{\scriptstyle \frac{1}{4}}\pi i}}
      {\sqrt{2}[-\lambda+|\sigma|]^{\scriptstyle \frac{1}{2}}},
\]
where $\bar{\partial}F=-|\sigma|e^{-i\phi}$ and we have introduced $e^{2u}=4(1+\xi\bar{\xi})^{-2}$. Note that when 
$|\sigma|=0$, then $\phi$ is just a gauge freedom for the frame.

Moreover, $\{E_{(1)},E_{(2)}\}$ span $T\Sigma$ and $\{E_{(3)},E_{(4)}\}$ span $N\Sigma$. 
\end{Prop}

\vspace{0.1in}

Using the same notation as above:

\vspace{0.1in}

\begin{Prop}\label{p:dualbasis}
The dual basis of 1-forms is:
\[
\theta^{(1)}={\mathbb I}m\left[(\alpha_1\partial\bar{F}+\bar{\alpha}_1(\bar{\partial}\bar{F}
    -2(F\partial u-\bar{F}\bar{\partial} u)))d\xi-\bar{\alpha}_1d\eta\right]e^{2u},
\]
\[
\theta^{(2)}=\;{\mathbb I}m\left[(\alpha_2\partial\bar{F}+\bar{\alpha}_2(\bar{\partial}\bar{F}
    -2(F\partial u-\bar{F}\bar{\partial} u)))d\xi-\bar{\alpha}_2d\eta\right]e^{2u},
\]
\[
\theta^{(3)}=\;{\mathbb I}m\left[(\alpha_2\partial\bar{F}-\bar{\alpha}_2\partial F)d\xi
       +\bar{\alpha}_2d\eta\right]e^{2u},
\]
\[
\theta^{(4)}={\mathbb I}m\left[(\alpha_1\partial\bar{F}-\bar{\alpha}_1\partial F)d\xi
       +\bar{\alpha}_1d\eta\right]e^{2u}.
\]
\end{Prop}

\vspace{0.1in}

\begin{Def}
We refer to the above frame as the {\it canonical frame} associated with the positive surface $\Sigma$. Any other frame that respects the
tangent and normal splitting $TTS^2=T\Sigma\oplus N\Sigma$ is of the form
\[
E'_{(1)}=\cos\theta {E}_{(1)}-\sin\theta {E}_{(2)},
\]
\[
E'_{(2)}=\sin\theta {E}_{(1)}+\cos\theta {E}_{(2)},
\]
\[
E'_{(3)}=\cos\psi {E}_{(3)}+\sin\psi {E}_{(4)},
\]
\[
E'_{(4)}=-\sin\psi {E}_{(3)}+\cos\psi {E}_{(4)},
\]
for some $\theta,\psi\in S^1$.
\end{Def}
\vspace{0.1in}

Now consider the Levi-Civita connection $\overline{\nabla}$ associated with ${\mathbb G}$ and for $X,Y\in T\Sigma$ we have
the orthogonal splitting
\[
\overline{\nabla}_X Y= \nabla^\parallel_X Y+A(X,Y),
\]
where $A:T\Sigma\times T\Sigma\rightarrow N\Sigma$ is the second fundamental form of the immersed surface $\Sigma$.

Let $\Delta=\lambda^2-|\sigma|^2$, not to be confused with the Laplacian $\triangle$ of the last section.

\vspace{0.1in}

\begin{Prop}\label{p:2ndff}
The second fundamental form is:
\[
A(e_{(a)},e_{(b)})=2{\mathbb R}{ e}\left[\beta_{ab}\left(\frac{\partial}{\partial \xi}
       +(\bar{\partial} \bar{F}-2(F\partial u-\bar{F}\bar{\partial} u))\frac{\partial}{\partial \eta}
          -\partial \bar{F}\frac{\partial}{\partial \bar{\eta}}    \right) \right],
\]
for $a,b=1,2$, where
\[
\beta_{11}=\left[i\lambda\partial |\sigma|-\sigma\bar{\partial}|\sigma|+i\lambda\partial\lambda-\sigma\bar{\partial}\lambda
    +|\sigma|(|\sigma|+\lambda)(\partial\phi-ie^{i\phi}\bar{\partial}\phi+2i\partial u-2e^{i\phi}\bar{\partial}u)\right]
\]
\[
\left/\left[2e^{2u+i\phi}(|\sigma|+\lambda)^2(-|\sigma|+\lambda)\right]\right.,
\]
\[
\beta_{22}=\left[-i\lambda\partial |\sigma|+\sigma\bar{\partial}|\sigma|+i\lambda\partial\lambda-\sigma\bar{\partial}\lambda
    +|\sigma|(|\sigma|-\lambda)(\partial\phi+ie^{i\phi}\bar{\partial}\phi+2i\partial u+2e^{i\phi}\bar{\partial}u)\right]
\]
\[
\left/\left[2 e^{2u+i\phi}(|\sigma|-\lambda)^2(-|\sigma|-\lambda)\right]\right.,
\]
\[
\beta_{12}=\left(-|\sigma|\partial |\sigma|+i\lambda e^{i\phi}\bar{\partial}|\sigma|+\lambda\partial\lambda-i\sigma\bar{\partial}\lambda
    \right)
\]
\[
\left/\left[2e^{2u+i\phi}(|\sigma|^2-\lambda^2)\sqrt{|\Delta|}\right]\right. .
\]
\end{Prop}
\begin{proof}
Consider the parallel and perpendicular projection operators $^\parallel P:TTS^2\rightarrow T\Sigma$ and
$^\perp P:TTS^2\rightarrow N\Sigma$. These are given in terms of an adapted frame by
\[
^\parallel P_j^k=\delta_j^k-E_{(3)}^k\theta_j^{(3)}-E_{(4)}^k\theta_j^{(4)}
\qquad\qquad
^\perp P_j^k=\delta_j^k-E_{(1)}^k\theta_j^{(1)}-E_{(2)}^k\theta_j^{(2)}.
\]
The parallel projection operator has the following coordinate description:
\begin{align}
^\parallel P_{\bar{\eta}}^\xi=-{\textstyle{\frac{1}{2\Delta}}}\bar{\sigma}
&\qquad^\parallel P_{\bar{\xi}}^\xi=-{\textstyle{\frac{1}{2\Delta}}} (\bar{\partial}\bar{F}+\lambda i)\bar{\sigma},\nonumber\\
^\parallel P_\eta^\xi= -{\textstyle{\frac{1}{2\Delta}}} \lambda i
&\qquad^\parallel P_\xi^\xi= {\textstyle{\frac{1}{2\Delta}}} [(\partial F -2\lambda i)\lambda i -|\sigma|^2], \nonumber\\
^\parallel P_{\bar{\eta}}^\eta={\textstyle{\frac{1}{2\Delta}}} \bar{\sigma}(\partial F-\lambda i)
&\qquad^\parallel P_{\bar{\xi}}^\eta={\textstyle{\frac{1}{2\Delta}}} [-\bar{\sigma}[\partial F\bar{\partial}\bar{F}-|\sigma|^2
                                 -\lambda i(\bar{\partial}\bar{F}-\partial F)]+2\lambda^2],\nonumber\\
^\parallel P_\eta^\eta= -{\textstyle{\frac{1}{2\Delta}}}[\lambda i\partial F +|\sigma|^2]
&\qquad^\parallel P_\xi^\eta= {\textstyle{\frac{1}{2\Delta}}} \lambda i[(\partial F-2\lambda i)\partial F-|\sigma|^2]\nonumber,
\end{align}
while the perpendicular projection operator is
\[
^\perp P_\xi^\xi=\;^\parallel P_\eta^\eta
\qquad\qquad
^\perp P_{\bar{\xi}}^\xi=-\;^\parallel P_{\bar{\xi}}^\xi
\qquad\qquad
^\perp P_\eta^\xi=-\;^\parallel P_\eta^\xi
\qquad\qquad
^\perp P_{\bar{\eta}}^\xi=-\;^\parallel P_{\bar{\eta}}^\xi,
\]
\[
^\perp P_\xi^\eta=-\;^\parallel P_\xi^\eta
\qquad\qquad
^\perp P_{\bar{\xi}}^\eta=-\;^\parallel P_{\bar{\xi}}^\eta
\qquad\qquad
^\perp P_\eta^\eta=\;^\parallel P_\xi^\xi
\qquad\qquad
^\perp P_{\bar{\eta}}^\eta=-\;^\parallel P_{\bar{\eta}}^\eta.
\]
In terms of a frame in which $\{E_{(1)},E_{(2)}\}$ span the tangent space of $\Sigma$, the second fundamental form
\[
A_{(ab)}^{\;\;\;\;\;j}=\;^\perp P_k^j\;E_{(a)}^l\overline{\nabla}_l\;E_{(b)}^k.
\]
The result follows by direct computation of these quantities.
\end{proof}

\vspace{0.1in}

\begin{Prop}
The mean curvature vector of the surface $\Sigma$ is:
\[
H=2{\mathbb R}{ e}\left[\gamma\left(\frac{\partial}{\partial \xi}
       +(\bar{\partial} \bar{F}-2(F\partial u-\bar{F}\bar{\partial} u))\frac{\partial}{\partial \eta}
          -\partial \bar{F}\frac{\partial}{\partial \bar{\eta}}    \right) \right],
\]
where
\[
\gamma=\left[-\lambda(-i\lambda\partial |\sigma|+\sigma\bar{\partial}|\sigma|)
     -|\sigma|(i\lambda\partial\lambda-\sigma\bar{\partial}\lambda)-|\sigma|(|\sigma|^2-\lambda^2)(\partial\phi+2i\partial u)\right]
\]
\[
\left/\left[e^{2u+i\phi}(|\sigma|^2-\lambda^2)^2\right]\right. .
\]
\end{Prop}
\begin{proof}
The mean curvature vector of the surface $\Sigma$ is the trace of the second fundamental form, which is
\[
H^j=H_{(11)}^{\;\;\;\;\;j}+ H_{(22)}^{\;\;\;\;\;j}.
\]
The result follows from computing this with the aid of the previous Proposition.
\end{proof}
\vspace{0.1in}

\begin{Note}
We can also write the mean curvature vector component (see \cite{gak8} for a variational derivation of this formula)
\begin{equation}\label{e:meanc}
H^\xi=\frac{2e^{-2u}}{\sqrt{|\lambda^2-|\sigma|^2|}}\left[
              ie^{-2u}\partial\left(\frac{\bar{\sigma}e^{2u}}{\sqrt{|\lambda^2-|\sigma|^2|}}\right)
       -\bar{\partial}\left(\frac{\lambda}{\sqrt{|\lambda^2-|\sigma|^2|}}\right)\right].
\end{equation}
\end{Note}

\begin{Cor}
A holomorphic graph has vanishing mean curvature.
\end{Cor}
\begin{proof}
This follows from inserting $\sigma=0$ in equation (\ref{e:meanc}).
\end{proof}
\vspace{0.1in}

\vspace{0.1in}

\subsection{Angles between positive surfaces}

In section \ref{s:anglesgen} a reduction of the action of the action of the group $O(n,m)$ by the maximal compact subgroup $O(n)\times O(m)$ 
led to a matrix of angles between 
pairs of positive $n$-planes in ${\mathbb E}^{n+m}$. This reduction carries over to $n+m$-dimensional manifolds and we now consider in more detail
the case of intersecting positive surfaces in $(TS^2,{\mathbb G})$. 

The positive surfaces we have in mind are the flowing disc $f_s(D)$ and the boundary surface $\tilde{\Sigma}$ intersecting along $f_s(\partial D)$. 
While we are working pointwise along this intersection, for ease of notation we drop any mention of the point of intersection.

Assume that $f_s(D)$ and $\tilde{\Sigma}$ are positive and intersect at a point. Choose an orthonormal frame 
$E_{(\mu)}=\{e_{(1)},e_{(2)},f_{(1)},f_{(2)}\}$ and coframe $E^{(\mu)}=\{e^{(1)},e^{(2)},f^{(1)},f^{(2)}\}$ so that 
$\{e_{(a)}\}$ span the tangent plane $Tf_s(D)$, while 
$\{f_{(a)}\}$ span the normal plane $Nf_s(D)$. Similarly, let $\tilde{E}_{(\mu)}=\{\tilde{e}_{(1)},\tilde{e}_{(2)},\tilde{f}_{(1)},\tilde{f}_{(2)}\}$ and 
$\tilde{E}^{(\mu)}=\{\tilde{e}^{(1)},\tilde{e}^{(2)},\tilde{f}^{(1)},\tilde{f}^{(2)}\}$ be similar frames and coframes for $\tilde{\Sigma}$. 

\begin{Def}
For frames as above define the $4\times 4$ matrix
\[
M_{(\mu)}^{\;\;(\nu)}=<E_{(\mu)},\tilde{E}^{(\nu)}>.
\]
\end{Def}
Then
\vspace{0.1in}

\begin{Prop}
There exist adapted frames such that the matrix $M$ has the form:
\[
M=\left(
     \begin{array}{cccc}
      \cosh A & 0 & 0 & \sinh A\\
       0 &  \cosh B & \sinh B  & 0 \\
       0 &  \sinh B & \cosh B & 0 \\
       \sinh A & 0 & 0 & \cosh A
     \end{array}\right),
\]
for hyperbolic angles $A,B\in{\mathbb R}$.
\end{Prop}
\begin{proof}
This is a special case of Proposition \ref{p:gauge}.
\end{proof}

\vspace{0.1in}
 
\begin{Cor}\label{c:multi}
In the special case where $f_s(D)$ and $\tilde{\Sigma}$ intersect along a curve the angle matrix reduces to
\[
M=\left(
     \begin{array}{cccc}
      1 & 0 & 0 & 0\\
       0 &  \cosh B & \sinh B  & 0 \\
       0 &  \sinh B & \cosh B & 0 \\
       0 & 0 & 0 & 1
     \end{array}\right).
\]
\end{Cor}

Here we have established the basic fact of neutral geometry that if two positive planes intersect on a line, then their normal planes must also 
intersect each other in a line. To pursue this further we need to relate the various angles between the tangent and normal planes of the two surfaces and 
their intersections, together with their complex slopes.

\begin{Prop} \label{p:hypang1}
Define $\mu=|\sigma|/|\lambda|$ and $\tilde{\mu}=|\tilde{\sigma}|/|\tilde{\lambda}|$. By positivity of the surfaces $\mu<1$ and $\tilde{\mu}<1$. Then
\begin{align}\label{e:hypangle}
\cosh B&=\frac{(1-\tilde{\mu}^2)(1+2\cos 2\theta\;\mu+\mu^2)+(1-\mu^2)(1+2\cos 2\tilde{\theta}\;\tilde{\mu}+\tilde{\mu}^2)}
{2(1-\tilde{\mu}^2)^{\scriptstyle{\frac{1}{2}}}(1-\mu^2)^{\scriptstyle{\frac{1}{2}}}(1+\cos 2\theta\;\mu)(1+\cos 2\tilde{\theta}\;\tilde{\mu})}\nonumber\\
&\qquad\qquad\qquad\qquad -\frac{\sin 2\theta\sin 2 \tilde{\theta}\;\mu\tilde{\mu}}{(1+\cos 2\theta\;\mu)(1+\cos 2\tilde{\theta}\;\tilde{\mu})}\nonumber\\
&=\frac{(1-\tilde{\mu}^2)(1+2\cos 2\psi\;\mu+\mu^2)+(1-\mu^2)(1+2\cos 2\tilde{\psi}\;\tilde{\mu}+\tilde{\mu}^2)}
  {2(1-\tilde{\mu}^2)^{\scriptstyle{\frac{1}{2}}}(1-\mu^2)^{\scriptstyle{\frac{1}{2}}}(1+\cos 2\psi\;\mu)(1+\cos 2\tilde{\psi}\;\tilde{\mu})}\nonumber\\
&\qquad\qquad\qquad\qquad -\frac{\sin 2\psi\sin 2 \tilde{\psi}\;\mu\tilde{\mu} }
{(1+\cos 2\psi\;\mu)(1+\cos 2\tilde{\psi}\;\tilde{\mu})},
\end{align}
where $\theta$,$\tilde{\theta}$,$\psi$ and $\tilde{\psi}$ determine the angles that the lines of intersection make with the canonical 
frames on $TD$,$T\tilde{\Sigma}$,$ND$ and $N\tilde{\Sigma}$, respectively.
\end{Prop}

\begin{proof}
Consider the tangent space to a point $\gamma\in D\cap\tilde{\Sigma}$ 
(assumed to be non-empty). This can be split two distinct ways
\[
T\gamma TS^2=T_\gamma D\oplus N_\gamma D=T_\gamma \tilde{\Sigma}\oplus N_\gamma \tilde{\Sigma},
\]
and adapted orthonormal bases $\{e_{(a)},f_{(b)}\}_{a,b=1}^2$ and $\{\tilde{e}_{(a)},\tilde{f}_{(b)}\}_{a,b=1}^2$ chosen for the splittings (respectively).

Recall the canonical frames from Proposition \ref{p:frames}:
\[
e_{(1)}=2{\mathbb R}{ e}\left[\alpha_1\left(\frac{\partial}{\partial \xi}+\partial F\frac{\partial}{\partial \eta}
          +\partial \bar{F}\frac{\partial}{\partial \bar{\eta}}    \right) \right],
\]
\[
e_{(2)}=2{\mathbb R}{ e}\left[\alpha_2\left(\frac{\partial}{\partial \xi}+\partial F\frac{\partial}{\partial \eta}
          +\partial \bar{F}\frac{\partial}{\partial \bar{\eta}}    \right) \right],
\]
\[
f_{(1)}=2{\mathbb R}{ e}\left[\alpha_2\left(\frac{\partial}{\partial \xi}
       +(\bar{\partial} \bar{F}-2(F\partial u-\bar{F}\bar{\partial} u))\frac{\partial}{\partial \eta}
          -\partial \bar{F}\frac{\partial}{\partial \bar{\eta}}    \right) \right],
\]
\[
f_{(2)}=2{\mathbb R}{ e}\left[\alpha_1\left(\frac{\partial}{\partial \xi}
       +(\bar{\partial} \bar{F}-2(F\partial u-\bar{F}\bar{\partial} u))\frac{\partial}{\partial \eta}
          -\partial \bar{F}\frac{\partial}{\partial \bar{\eta}}    \right) \right],
\]
for
\begin{equation}\label{e:alphas}
\alpha_1=\frac{e^{-u-{\scriptstyle \frac{1}{2}}\phi i+{\scriptstyle \frac{1}{4}}\pi i}}
      {\sqrt{2}[-\lambda-|\sigma|]^{\scriptstyle \frac{1}{2}}},
\qquad\qquad
\alpha_2=\frac{e^{-u-{\scriptstyle \frac{1}{2}}\phi i-{\scriptstyle \frac{1}{4}}\pi i}}
      {\sqrt{2}[(-\lambda+|\sigma|)]^{\scriptstyle \frac{1}{2}}},
\end{equation}
where $\bar{\partial}F=-|\sigma|e^{-i\phi}$ and $e^{2u}=4(1+\xi\bar{\xi})^{-2}$.

Analogous expressions hold for the bases of $T\tilde{\Sigma}$ and $N\tilde{\Sigma}$, with a tilde on appropriate quantities. 

Assume that $D$ and $\tilde{\Sigma}$ intersect along a curve and rotate the frames in the tangent bundle so that
the first vector of each basis coincide and lie along the tangent to the intersection. Thus, referring to the above frames, there exists  
rotated frames $\{\mathring{e}_{(a)},\mathring{f}_{(b)}\}$ and $\{\mathring{\tilde{e}}_{(a)},\mathring{\tilde{f}}_{(b)}\}$
with
\[
\mathring{e}_{(1)}=\cos\theta {e}_{(1)}-\sin\theta {e}_{(2)},
\qquad\qquad
\mathring{e}_{(2)}=\sin\theta {e}_{(1)}+\cos\theta {e}_{(2)},
\] 
\[
\mathring{\tilde{e}}_{(1)}=\cos\tilde{\theta} {\tilde{e}}_{(1)}-\sin\tilde{\theta} {\tilde{e}}_{(2)},
\qquad\qquad
\mathring{\tilde{e}}_{(2)}=\sin\tilde{\theta} {\tilde{e}}_{(1)}+\cos\tilde{\theta} {\tilde{e}}_{(2)},
\]
and $\mathring{e}_{(1)}=\mathring{\tilde{e}}_{(1)}$, for some $\theta,\tilde{\theta}\in [0,2\pi)$.

\newpage
\includegraphics{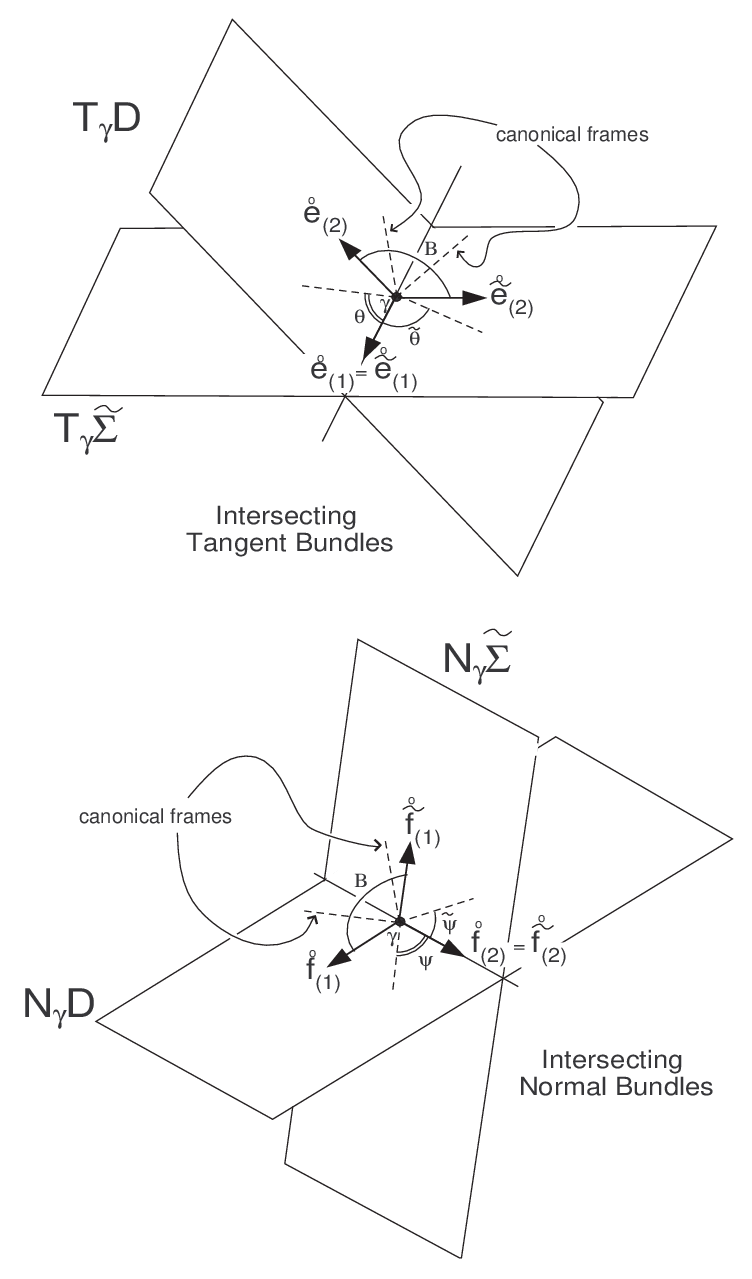}
\newpage

Now compute $\cosh B={\mathbb G}(\mathring{e}_{(2)},\mathring{\tilde{e}}_{(2)})$ with the aid of these expressions.
First note that $\mathring{e}_{(1)}=\mathring{\tilde{e}}_{(1)}$ implies that
\begin{equation}\label{e:intsect1}
\cos\theta {\alpha}_{1}-\sin\theta {\alpha}_{2}=\cos\tilde{\theta} {\tilde{\alpha}}_{1}-\sin\tilde{\theta} {\tilde{\alpha}}_{2},
\end{equation}
\begin{equation}\label{e:intsect2}
(\cos\theta {\alpha}_{1}-\sin\theta {\alpha}_{2})\partial(F-\tilde{F})+(\cos\theta \bar{\alpha}_{1}
-\sin\theta \bar{\alpha}_{2})\bar{\partial}(F-\tilde{F})=0.
\end{equation}

Substituting the expressions for $\alpha_1$ and $\alpha_2$ in equation (\ref{e:intsect1}), we find
\begin{equation}\label{e:intsect2a}
e^{{\scriptstyle \frac{1}{2}}{\phi}}\left(\frac{\cos{\theta}}{[-{\lambda}-|{\sigma}|]^{\scriptstyle \frac{1}{2}}}
    -\frac{i\sin{\theta}}{[-{\lambda}+|{\sigma}|]^{\scriptstyle \frac{1}{2}}}\right)
  =e^{{\scriptstyle \frac{1}{2}}\tilde{\phi}}\left(\frac{\cos\tilde{\theta}}{[-\tilde{\lambda}-|\tilde{\sigma}|]^{\scriptstyle \frac{1}{2}}}
    -\frac{i\sin\tilde{\theta}}{[-\tilde{\lambda}+|\tilde{\sigma}|]^{\scriptstyle \frac{1}{2}}}\right).
\end{equation}
The norm and argument of this equation give
\begin{equation}\label{e:hypangcon1}
(\tilde{\lambda}^2-|\tilde{\sigma}|^2)(-\lambda+|\sigma|\cos 2\theta)=(\lambda^2-|\sigma|^2)(-\tilde{\lambda}+|\tilde{\sigma}|\cos 2\tilde{\theta}),
\end{equation}
and
\begin{equation}\label{e:hypangcon2}
e^{(\phi-\tilde{\phi})i}=\frac{(\lambda^2-|\sigma|^2)(|\tilde{\sigma}|-\tilde{\lambda}\cos 2\tilde{\theta}
   -i(\tilde{\lambda}^2-|\tilde{\sigma}|^2)^{\scriptstyle \frac{1}{2}} \sin 2\tilde{\theta})}
   {(\tilde{\lambda}^2-|\tilde{\sigma}|^2)(|{\sigma}|-{\lambda}\cos 2{\theta}
   -i({\lambda}^2-|{\sigma}|^2)^{\scriptstyle \frac{1}{2}} \sin 2{\theta})}.
\end{equation}

Moreover, introducing the complex slopes $\sigma$ and $\rho=\vartheta+i\lambda$ as before, equation (\ref{e:intsect2}) 
can be split into real and imaginary
parts. The result, with the aid of equation (\ref{e:hypangcon2}), is
\begin{equation}\label{e:hypangcon3}
\vartheta-\tilde{\vartheta}=\frac{\sin 2\tilde{\theta}(\tilde{\lambda}^2-|\tilde{\sigma}|^2)^{\scriptstyle \frac{1}{2}}}
  {-\tilde{\lambda}+|\tilde{\sigma}|\cos 2\tilde{\theta}}|\tilde{\sigma}|
   -\frac{\sin 2{\theta}({\lambda}^2-|{\sigma}|^2)^{\scriptstyle \frac{1}{2}}}
  {-{\lambda}+|{\sigma}|\cos 2{\theta}}|{\sigma}|,
\end{equation}
and
\begin{equation}\label{e:hypangcon4}
\lambda-\tilde{\lambda}=-\frac{\tilde{\lambda}\cos 2\tilde{\theta}-|\tilde{\sigma}|}
{-\tilde{\lambda}+|\tilde{\sigma}|\cos 2\tilde{\theta}}|\tilde{\sigma}|
  +\frac{{\lambda}\cos 2{\theta}-|{\sigma}|}{-{\lambda}+|{\sigma}|\cos 2{\theta}}|{\sigma}|.
\end{equation}
Note that equations (\ref{e:hypangcon1}) and (\ref{e:hypangcon4}) are, in fact, the same equation.

We now compute, with the aid of the metric expression, that
\begin{align}
{\mathbb G}(\mathring{e}_{(2)},\mathring{\tilde{e}}_{(2)})&=\frac{4}{(1+\xi\bar{\xi})^2}{\mathbb I}{\mbox {m}}\left[  
     (\sin\theta \bar{\alpha}_{1}+\cos\theta \bar{\alpha}_{2})(\sin\tilde{\theta} \tilde{\alpha}_{1}+\cos\tilde{\theta} \tilde{\alpha}_{2})
     (\bar{\rho}-\tilde{\rho})\right. \nonumber\\
&\qquad \qquad\qquad \qquad\left. -(\sin\theta {\alpha}_{1}+\cos\theta {\alpha}_{2})(\sin\tilde{\theta} \tilde{\alpha}_{1}+\cos\tilde{\theta} \tilde{\alpha}_{2})(\tilde{\sigma}+\sigma)\right].\nonumber
\end{align}

The expression for the angle given in the Proposition follows from a lengthy computation in which equations (\ref{e:hypangcon2}) 
and (\ref{e:hypangcon3}), recalling that $\rho=\vartheta+i\lambda$,
along with the definitions of $\alpha_1$ and $\alpha_2$ given in equations (\ref{e:alphas}), are substituted in the above expression. These result in

\begin{align}\label{e:hypangle1}
\cosh B=&-\frac{(\lambda^2-|\sigma|^2)^{\scriptstyle{\frac{1}{2}}}}
{(\tilde{\lambda}^2-|\tilde{\sigma}|^2)^{\scriptstyle{\frac{1}{2}}}(-\lambda+|\sigma|\cos 2\theta)}
\left(\lambda+\frac{\tilde{\lambda}\cos 2\tilde{\theta}-|\tilde{\sigma}|}{|\tilde{\sigma}|\cos 2\tilde{\theta}-\tilde{\lambda}}|\tilde{\sigma}|\right)
\nonumber\\
&\qquad\qquad-\frac{\sin 2\theta\sin 2\tilde{\theta}|\sigma||\tilde{\sigma}|}
{(-\lambda+|\sigma|\cos 2\theta)(-\tilde{\lambda}+|\tilde{\sigma}|\cos 2\tilde{\theta})},
\end{align}
where the intersection constraint is
\begin{equation}\label{e:intersectionconstraint1}
(\tilde{\lambda}^2-|\tilde{\sigma}|^2)(-\lambda+|\sigma|\cos 2\theta)=(\lambda^2-|\sigma|^2)(-\tilde{\lambda}+|\tilde{\sigma}|\cos 2\tilde{\theta}).
\end{equation}

The first equation of the Proposition follows from using the definitions $\mu=|\sigma|/|\lambda|$ and $\tilde{\mu}=|\tilde{\sigma}|/|\tilde{\lambda}|$.
Note that, the intersection equation can be taken as defining the ratio of the twists of the intersecting surfaces:

\begin{equation}\label{e:intersectionconstraint}
\frac{\tilde{\lambda}}{\lambda}=\frac{(1-\mu^2)(1+\cos 2\tilde{\theta}\;\tilde{\mu})}{(1-\tilde{\mu}^2)(1+\cos 2\theta\;\mu)}.
\end{equation}

The preceding arguments, which involve the angles $\theta$ and $\tilde{\theta}$, can also be expressed in terms of the angles $\psi$ and $\tilde{\psi}$ 
tracking the intersection of the {\it normal} bundles $ND$ and $N\tilde{\Sigma}$, respectively.  

More specifically suppose that the adapted frames in the normal bundles are
\[
\mathring{f}_{(1)}=\cos\psi {f}_{(1)}+\sin\psi {f}_{(2)},
\qquad\qquad
\mathring{f}_{(2)}=-\sin\psi {f}_{(1)}+\cos\psi {f}_{(2)},
\] 
\[
\mathring{\tilde{f}}_{(1)}=\cos\tilde{\psi} {\tilde{f}}_{(1)}+\sin\tilde{\psi} {\tilde{f}}_{(2)},
\qquad\qquad
\mathring{\tilde{f}}_{(2)}=-\sin\tilde{\psi} {\tilde{f}}_{(1)}+\cos\tilde{\psi} {\tilde{f}}_{(2)},
\]
and $\mathring{f}_{(2)}=\mathring{\tilde{f}}_{(2)}$, for some $\psi,\tilde{\psi}\in [0,2\pi)$.

Then the angle can be computed using $\cosh B=-{\mathbb{G}}(\mathring{f}_{(1)},\mathring{\tilde{f}}_{(1)})$ which yields the second part of the 
Proposition.

We also have the relations ({\it c.f.} equations (\ref{e:hypangcon1}), (\ref{e:hypangcon2}) and (\ref{e:hypangcon3})):
\[
(\tilde{\lambda}^2-|\tilde{\sigma}|^2)(-\lambda+|\sigma|\cos 2\psi)=(\lambda^2-|\sigma|^2)(-\tilde{\lambda}+|\tilde{\sigma}|\cos 2\tilde{\psi}),
\]
\[
e^{(\phi-\tilde{\phi})i}=\frac{(\lambda^2-|\sigma|^2)(|\tilde{\sigma}|-\tilde{\lambda}\cos 2\tilde{\psi}
   -i(\tilde{\lambda}^2-|\tilde{\sigma}|^2)^{\scriptstyle \frac{1}{2}} \sin 2\tilde{\psi})}
   {(\tilde{\lambda}^2-|\tilde{\sigma}|^2)(|{\sigma}|-{\lambda}\cos 2{\psi}
   -i({\lambda}^2-|{\sigma}|^2)^{\scriptstyle \frac{1}{2}} \sin 2{\psi})},
\]
\[
\vartheta-\tilde{\vartheta}=-\frac{\sin 2\tilde{\psi}(\tilde{\lambda}^2-|\tilde{\sigma}|^2)^{\scriptstyle \frac{1}{2}}}
  {-\tilde{\lambda}+|\tilde{\sigma}|\cos 2\tilde{\psi}}|\tilde{\sigma}|
   +\frac{\sin 2{\psi}({\lambda}^2-|{\sigma}|^2)^{\scriptstyle \frac{1}{2}}}
  {-{\lambda}+|{\sigma}|\cos 2{\psi}}|{\sigma}|.
\]
\end{proof}
\vspace{0.1in}

We now state a series of corollaries to these relations that play a key role in controlling the flow of the edge.

\begin{Cor}
If $f_s(D)$ is holomorphic along $f_s(\partial D)$, then 
\[
\cosh B=\frac{{\mbox{Area}}_\Omega(T\tilde{\Sigma})}{{\mbox{Area}}_{\mathbb{G}}(T\tilde{\Sigma})}.
\]
\end{Cor}
\vspace{0.1in}

\begin{Cor}
The hyperbolic angle can also be written
\begin{align}
\cosh B&=\frac{|\lambda|^{\textstyle{\frac{1}{2}}}(1+2\mu\cos 2\theta+\mu^2)}
{2|\tilde{\lambda}|^{\textstyle{\frac{1}{2}}}(1+\mu\cos 2\theta)^{\textstyle{\frac{3}{2}}}(1+\tilde{\mu}\cos 2\tilde{\theta})^{\textstyle{\frac{1}{2}}}}
+\frac{|\tilde{\lambda}|^{\textstyle{\frac{1}{2}}}(1+2\tilde{\mu}\cos 2\tilde{\theta}+\tilde{\mu}^2)}
{2|\lambda|^{\textstyle{\frac{1}{2}}}(1+\mu\cos 2\theta)^{\textstyle{\frac{1}{2}}}(1+\tilde{\mu}\cos 2\tilde{\theta})^{\textstyle{\frac{3}{2}}}}\nonumber\\
&\qquad\qquad\qquad\qquad-\frac{\sin 2\theta\sin 2\tilde{\theta}}{(1+\mu\cos 2\theta)(1+\tilde{\mu}\cos 2\tilde{\theta})},\nonumber\\
&=\frac{|\lambda|^{\textstyle{\frac{1}{2}}}(1+2\mu\cos 2\psi+\mu^2)}
{2|\tilde{\lambda}|^{\textstyle{\frac{1}{2}}}(1+\mu\cos 2\psi)^{\textstyle{\frac{3}{2}}}(1+\tilde{\mu}\cos 2\tilde{\psi})^{\textstyle{\frac{1}{2}}}}
+\frac{|\tilde{\lambda}|^{\textstyle{\frac{1}{2}}}(1+2\tilde{\mu}\cos 2\tilde{\psi}+\tilde{\mu}^2)}
{2|\lambda|^{\textstyle{\frac{1}{2}}}(1+\mu\cos 2\psi)^{\textstyle{\frac{1}{2}}}(1+\tilde{\mu}\cos 2\tilde{\psi})^{\textstyle{\frac{3}{2}}}}\nonumber\\
&\qquad\qquad\qquad\qquad-\frac{\sin 2\psi\sin 2\tilde{\psi}}{(1+\mu\cos 2\psi)(1+\tilde{\mu}\cos 2\tilde{\psi})},\label{e:hypang2}
\end{align}
\end{Cor}
\begin{proof}
The two equality follows from Proposition \ref{p:hypang1} and equation (\ref{e:intersectionconstraint}) and the equivalent $\psi$ version.
\end{proof}\vspace{0.1in}

\begin{Cor}
The hyperbolic angle between two intersecting positive planes satisfies the inequalities
\[
\frac{1-\mu\tilde{\mu}}{(1-\mu^2)^{\scriptstyle{\frac{1}{2}}}(1-\tilde{\mu}^2)^{\scriptstyle{\frac{1}{2}}}}\leq\cosh B
  \leq\frac{1+\mu\tilde{\mu}}{(1-\mu^2)^{\scriptstyle{\frac{1}{2}}}(1-\tilde{\mu}^2)^{\scriptstyle{\frac{1}{2}}}}.
\]
\end{Cor}
\begin{proof}
These follow from maximizing and minimizing $\cosh B$ over $\theta$ and $\tilde{\theta}$.
\end{proof}

\vspace{0.1in}
\begin{Cor}\label{c:nulltogether}
Consider the set $U_B$ of pairs of intersecting positive planes with fixed hyperbolic angle $B$ in ${\mathbb R}^{2,2}$. Let $(P_s,\tilde{P}_s)\in U_B$ be
a curve of pairs of planes for
$s\in[0,1)$. Then $\lim_{s\rightarrow 1}P_s$ is degenerate iff $\lim_{s\rightarrow 1}\tilde{P}_s$ is degenerate.
\end{Cor}
\begin{proof}
From the left-hand inequality of the previous Corollary it is clear that, for fixed $B$, $\lim_{s\rightarrow 1}\mu=1$ iff 
$\lim_{s\rightarrow 1}\tilde{\mu}=1$.
\end{proof}

\vspace{0.1in}

\begin{Cor}\label{c:ineq2}
The following inequality holds for intersecting positive planes forming an angle $B>0$:
\begin{equation}\label{e:ineq2}
\frac{\tilde{\mu}-\cosh B\sinh B\;(1-\tilde{\mu}^2)}{\tilde{\mu}^2+\cosh^2 B\;(1-\tilde{\mu}^2)}\leq \mu
  \leq\frac{\tilde{\mu}+\cosh B\sinh B\;(1-\tilde{\mu}^2)}{\tilde{\mu}^2+\cosh^2 B\;(1-\tilde{\mu}^2)},
\end{equation}
which can also be written
\begin{equation}\label{e:ineq3}
\frac{\cosh B-\tilde{\mu}\sinh B}{(1-\tilde{\mu}^2)^{\scriptstyle{\frac{1}{2}}}}\leq\frac{1}{(1-\mu^2)^{\scriptstyle{\frac{1}{2}}}}\leq \frac{\cosh B+\tilde{\mu}\sinh B}{(1-\tilde{\mu}^2)^{\scriptstyle{\frac{1}{2}}}}.
\end{equation}
\end{Cor}
\vspace{0.2in}


\section{{\bf Mean Curvature Flow With Boundary in $TS^2$}}\label{s:mcfts2}

Throughout we use the term positive surface to mean spacelike surface: the induced metric is positive definite.

\vspace{0.1in}

\subsection{The I.B.V.P.}

We now investigate the initial boundary value problem of section \ref{s:bvp}, namely unparameterised mean curvature flow with boundary
conditions:

\vspace{0.2in}
\begin{center}\fbox{\parbox{4.8in}{
\begin{center}{\Large{\bf I.B.V.P.}}\end{center}
\vspace{0.1in}
{\it
Consider a family of positive sections $f_s:D\rightarrow TS^2$ such that
\[
\frac{d f}{ds}^\bot=H,
\]
with initial and boundary conditions:
\vspace{0.1in}
\begin{enumerate}
\item[(i)] $f_0(D)=\Sigma_0,$
\item[(ii)]$f_s(\partial D)\subset \tilde{\Sigma},$
\item[(iii)] the hyperbolic angle $B$ between $Tf_s(D)$ and $T\tilde{\Sigma}$ is constant along $f_s(\partial D)$,
\item[(iv)] $f_s(\partial D)$ is asymptotically holomorphic: $|\bar{\partial}f_s|=C/(1+s)$,
\end{enumerate} 
\vspace{0.1in}
where $H$ is the mean curvature vector of $f_s(D)$, and $\Sigma_0$ and
$\tilde{\Sigma}$ are some given positive sections.
}
\vspace{0.1in}
}}
\end{center}
\vspace{0.2in}

We consider this flow when the flowing and boundary surfaces are graphs of sections of $TS^2\rightarrow S^2$. In this
case it is most convenient to use the base to parameterize the surfaces. That is, we consider a flowing surface given by
$\xi\mapsto(\xi,\eta=F_s(\xi,\bar{\xi}))$ and boundary surface given by $\xi\mapsto(\xi,\eta=\tilde{F}(\xi,\bar{\xi}))$. We show that
under the flow, the surface remains a graph. 

For the moment, we compute the explicit expressions for the flow of the function $F_s$.

\begin{Prop}
For a positive graph in $TS^2$, the mean curvature flow is
\begin{align}
\frac{\partial F}{\partial s}=&g^{jk}\partial_j\partial_k F+\frac{i\bar{\sigma}}{\Delta}\left((\sigma\xi-\bar{\rho}\bar{\xi})
    (1+\xi\bar{\xi})+\bar{F}-\bar{\xi}^2F\right)\nonumber\\
&=\frac{(1+\xi\bar{\xi})^2}{2(\lambda^2-\sigma\bar{\sigma})}\left(-2\bar{\sigma}\partial\lambda-i\bar{\sigma}\bar{\partial}\sigma+2\lambda\partial\bar{\sigma}+i\sigma\bar{\partial}\bar{\sigma}
   +\frac{4i\bar{\sigma}(\sigma\xi+\lambda i\bar{\xi})}{1+\xi\bar{\xi}}\right)\label{e:floweq}.
\end{align}
\end{Prop}
\begin{proof}
Consider a positive surface $f:\Sigma\times[0,s_0)\rightarrow{\mathbb M}$ such that 
$f_s(\xi,\bar{\xi})=(\xi,\bar{\xi},F_s(\xi,\bar{\xi}),\bar{F}_s(\xi,\bar{\xi}))$. Then
\[
\frac{\partial f}{\partial s}=\frac{\partial F}{\partial s}\frac{\partial }{\partial \eta}+\frac{\partial \bar{F}}{\partial s}\frac{\partial }{\partial \bar{\eta}}.
\]
Projecting onto the normal of $\Sigma$
\begin{align}
\frac{\partial f}{\partial s}^\bot=&(^\bot P^\xi_\eta \dot{F}+\;^\bot P^\xi_{\bar{\eta}} \dot{\bar{F}})\frac{\partial }{\partial \xi}
    +(^\bot P^\eta_\eta \dot{F}+\;^\bot P^\eta_{\bar{\eta}} \dot{\bar{F}})\frac{\partial }{\partial \eta}\nonumber\\
&\qquad+(^\bot P^{\bar{\xi}}_{\bar{\eta}} \dot{\bar{F}}+\;^\bot P^{\bar{\xi}}_{\eta} \dot{F})\frac{\partial }{\partial \bar{\xi}}
    +(^\bot P^{\bar{\eta}}_{\bar{\eta}} \dot{\bar{F}}+\;^\bot P^{\bar{\eta}}_{\eta} \dot{F})\frac{\partial }{\partial \bar{\eta}}\nonumber,
\end{align}
and so the mean curvature flow is
\[
^\bot P^\xi_\eta \dot{F}+\;^\bot P^\xi_{\bar{\eta}} \dot{\bar{F}}=H^\xi,
\]
or from the expressions of the projection operators given in the proof of Proposition \ref{p:2ndff}
\[
\frac{\lambda i}{2\Delta}\dot{F}-\frac{\bar{\sigma}}{2\Delta} \dot{\bar{F}}=H^\xi.
\]
Combining this with its complex conjugate we have
\begin{equation}\label{e:fdot1}
\dot{F}=-2\lambda iH^\xi+2\bar{\sigma}H^{\bar{\xi}}.
\end{equation}
Using the expression (\ref{e:meanc}) for the mean curvature we get that
\begin{align}
H^\xi=&\frac{(1+\xi\bar{\xi})^2}{4\Delta^2}\Big[2\left(i\partial\bar{\sigma}-\bar{\partial}\lambda
    -\frac{2i\bar{\xi}\bar{\sigma}}{1+\xi\bar{\xi}}\right)\Delta\nonumber\\
  &\qquad\qquad\qquad-2i\lambda\bar{\sigma}\partial\lambda+i\sigma\bar{\sigma}\partial\bar{\sigma}+i\bar{\sigma}^2\partial \sigma+2\lambda^2\bar{\partial}\lambda
      -\lambda\sigma\bar{\partial}\bar{\sigma}-\lambda\bar{\sigma}\bar{\partial}\sigma\Big]\nonumber,
\end{align}
and the second equality stated in the Proposition follows from inserting this in equation (\ref{e:fdot1}).

To see that the first equality in the Proposition holds, we compute
\begin{align}
g^{jk}\partial_j\partial_k F&=\frac{(1+\xi\bar{\xi})^2}{2\Delta}\left(i\bar{\sigma}\partial^2F-2\lambda\partial\bar{\partial}F
     -i\sigma\bar{\partial}^2F \right)\nonumber\\
&=\frac{(1+\xi\bar{\xi})^2}{2\Delta}\left[i\bar{\sigma}\partial\left(\theta+i\lambda+\frac{2\bar{\xi}F}{1+\xi\bar{\xi}}\right)
    +2\lambda\partial\bar{\sigma}+i\sigma\bar{\partial}\bar{\sigma} \right]\nonumber\\
&=\frac{(1+\xi\bar{\xi})^2}{2\Delta}\left[-2\bar{\sigma}\partial\lambda-i\bar{\sigma}\bar{\partial}\sigma+i\sigma\bar{\partial}\bar{\sigma} 
   +2\lambda\partial\bar{\sigma}\right.\nonumber\\
&\qquad\qquad\qquad\qquad\left.+i\bar{\sigma}\left(\frac{2(\sigma\xi+\rho\bar{\xi})}{1+\xi\bar{\xi}}
   -\frac{2(\bar{F}-\bar{\xi}^2F)}{(1+\xi\bar{\xi})^2} \right)\right]\nonumber,
\end{align}
where we have used identity (\ref{e:id2}) in the more convenient form
\[
\partial\theta=i\partial\lambda-(1+\xi\bar{\xi})^2\partial\left(\frac{\bar{\sigma}}{(1+\xi\bar{\xi})^2}\right)-\frac{2F}{(1+\xi\bar{\xi})^2}.
\]
Thus 
\[
g^{jk}\partial_j\partial_k F+\frac{i\bar{\sigma}}{\Delta}\left((\sigma\xi-\bar{\rho}\bar{\xi})
    (1+\xi\bar{\xi})+\bar{F}-\bar{\xi}^2F\right)\qquad\qquad\qquad\qquad\qquad\qquad\qquad\qquad
\]
\[
=\frac{(1+\xi\bar{\xi})^2}{2(\lambda^2-\sigma\bar{\sigma})}\left(-2\bar{\sigma}\partial\lambda-i\bar{\sigma}\bar{\partial}\sigma
    +i\sigma\bar{\partial}\bar{\sigma}+2\lambda\partial\bar{\sigma}+\frac{4i\bar{\sigma}(\sigma\xi+\lambda i\bar{\xi})}{1+\xi\bar{\xi}}\right),
\]
as claimed.
\end{proof}

\subsection{Evolution equations for a positive surface}

\begin{Prop}\label{p:shearflow}
Under the mean curvature flow the shear evolves by:
\[
\left(\frac{\partial }{\partial s}-{\mathbb G}^{jk}\partial_j\partial_k\right)\sigma=\frac{H_1(1+\xi\bar{\xi})^2+2H_2(1+\xi\bar{\xi})+2H_3}{2\Delta^2},
\]
where
\begin{align}
H_1=&-4\lambda \sigma \partial \lambda \bar{\partial} \lambda
-2i \lambda \bar{\sigma} \partial \lambda \partial \sigma
+2(\lambda^{2}+|\sigma|^2) \partial \lambda \bar{\partial} \sigma
+2i \lambda \sigma \partial \lambda \partial \bar{\sigma}
+2\lambda^{2} \bar{\partial} \lambda \partial \sigma \nonumber\\
&\qquad+2\sigma^{2} \bar{\partial} \lambda \partial \bar{\sigma}
+i \bar{\sigma}^{2}\left(\partial \sigma\right)^{2}
-2\lambda \bar{\sigma} \partial \sigma \bar{\partial} \sigma
-2\lambda \sigma \bar{\partial} \sigma \partial \bar{\sigma}
-i \sigma^{2}\left(\partial \bar{\sigma}\right)^{2} \nonumber,
\end{align}
\begin{align}
H_2&=-2 \sigma \partial \lambda\left(2 i \lambda\bar{\sigma} \bar{\xi}+ (\lambda^{2}+\sigma \bar{\sigma}) \xi \right)
+2 \sigma \bar{\partial} \lambda \left( \lambda^{2}-\sigma \bar{\sigma}\right) \bar{\xi}
+ \partial\sigma \left(i\bar{\sigma} \bar{\xi}(3 \lambda^{2} - \sigma \bar{\sigma}) +2 \lambda^{3} \xi \right)\nonumber\\
& \qquad\qquad
+ \bar{\partial}\sigma\left( i\sigma \xi-2 \lambda \bar{\xi} \right)(\lambda^{2} -\sigma \bar{\sigma})
+2 \partial \bar{\sigma} \left( i\sigma^{2} \bar{\sigma} \bar{\xi}+ \lambda\sigma^{2} \xi \right) \nonumber,
\end{align}
and
\[
H_3=-\sigma \left( i\sigma \xi^{2}-3 i\bar{\sigma} \bar{\xi}^{2}-4 \lambda\xi \bar{\xi} \right)(\lambda^{2} -\sigma \bar{\sigma}).
\]
In addition,
\[
\left(\frac{\partial }{\partial s}-{\mathbb G}^{jk}\partial_j\partial_k\right)\rho=\frac{H_4(1+\xi\bar{\xi})^3+H_5(1+\xi\bar{\xi})^2
     +H_6(1+\xi\bar{\xi})+H_7}{2(1+\xi\bar{\xi})\Delta^2},
\]
where
\begin{align}
H_4&=4\lambda \bar{\sigma} \left(\partial \lambda\right)^{2}
-2\bar{\sigma}^{2} \partial \lambda \partial \sigma
+2i \lambda \bar{\sigma} \partial \lambda \bar{\partial} \sigma
-4\lambda^{2} \partial \lambda \partial \bar{\sigma}
-2\sigma \bar{\sigma} \partial \lambda \partial \bar{\sigma}
-2i \lambda \sigma \partial \lambda \bar{\partial} \bar{\sigma} \nonumber\\
&\qquad -i \bar{\sigma}^{2}\partial \sigma \bar{\partial} \sigma
+2\lambda \bar{\sigma} \partial \sigma \partial \bar{\sigma}
+i \lambda^{2}\partial \sigma \bar{\partial} \bar{\sigma}
-i \lambda^{2}\bar{\partial} \sigma \partial \bar{\sigma}
+2\lambda \sigma \left(\partial \bar{\sigma}\right)^{2}
+i \sigma^{2}\partial \bar{\sigma} \bar{\partial} \bar{\sigma} \nonumber,
\end{align}
\begin{align}
H_5&=-4 \partial \lambda \left(2 i \lambda\sigma \bar{\sigma} \xi
- \lambda^{2} \bar{\sigma} \bar{\xi}
-\sigma \bar{\sigma}^{2} \bar{\xi} \right) +2\partial\sigma\left( i \lambda^{2} \bar{\sigma} \xi
+ i\sigma \bar{\sigma}^{2} \xi
-2 \lambda \bar{\sigma}^{2} \bar{\xi} \right) \nonumber\\
&\qquad +2\left(2 \partial \bar{\sigma} i \lambda^{2}\sigma \xi
-2 \partial \bar{\sigma} \lambda\sigma \bar{\sigma} \bar{\xi}
+ \bar{\partial} \bar{\sigma} i \lambda^{2}\sigma \bar{\xi}
- \bar{\partial} \bar{\sigma} i\sigma^{2} \bar{\sigma} \bar{\xi} \right)\nonumber,
\end{align}
\[
H_6=-4( i (\lambda^{2}-\sigma \bar{\sigma})+ \lambda \vartheta)(\lambda^{2}-\sigma \bar{\sigma}),
\]
and
\[
H_7=4i (F \sigma \xi-\bar{F} \bar{\sigma} \bar{\xi})(\lambda^{2}-\sigma \bar{\sigma}).
\]
\end{Prop}
\begin{proof}
The proofs of these statements follow from differentiation of the flow equation (\ref{e:floweq}). We illustrate this for the flow of
$\sigma$, leaving the flow of $\rho$ to the reader.

We start by splitting the expression into convenient terms:
\[
-\dot{\bar{\sigma}}=\bar{\partial}\dot{F}=E_1+E_2+E_3+E_4,
\]
where $E_1$ is second order in the derivatives of $\lambda$ and $\sigma$, $E_2$ and $E_3$ are the quadratic and linear first order terms, and $E_4$ is the zeroth 
order terms. We now compute each of these terms in turn.

So, differentiating equation (\ref{e:floweq}) we have
\[
E_1=\frac{(1+\xi\bar{\xi})^2}{2\Delta}\left(-2\bar{\sigma}\bar{\partial}\partial\lambda-i\bar{\sigma}\bar{\partial}\bar{\partial}\sigma+2\lambda\bar{\partial}\partial\bar{\sigma}+i\sigma\bar{\partial}\bar{\partial}\bar{\sigma}\right).
\]

At this point we exploit the 3-jet identity (\ref{e:id2}) which we write in the more favorable form
\[
\partial\bar{\partial}\lambda={\mathbb I}{\mbox m}\;\left[\bar{\partial}\bar{\partial}\sigma-\frac{2\xi}{1+\xi\bar{\xi}}\bar{\partial}\sigma
+\frac{2\xi^2}{(1+\xi\bar{\xi})^2}\sigma\right]-\frac{2\lambda}{(1+\xi\bar{\xi})^2}.
\]
Inserting this in the expression for $E_1$ yields
\[
E_1=\frac{(1+\xi\bar{\xi})^2}{2\Delta}\left(-i\bar{\sigma}\partial\partial\bar{\sigma}+2\lambda\partial\bar{\partial}\bar{\sigma}+i\sigma\bar{\partial}\bar{\partial}\bar{\sigma}
-\frac{2i\bar{\sigma}(\xi\bar{\partial}\sigma-\bar{\xi}\partial\bar{\sigma})}{1+\xi\bar{\xi}}
+\frac{2i\bar{\sigma}(\xi^2\sigma-\bar{\xi}^2\bar{\sigma}-2i\lambda)}{(1+\xi\bar{\xi})^2}
\right).
\]
The first three terms of this, noting the expression for $g^{-1}$ in Proposition \ref{p:ts2indmet}, are easily seen to be the 
rough Laplacian of  $-\bar{\sigma}$:  
\[
E_1=-g^{jk}\partial_j\partial_k\bar{\sigma}+\frac{i\bar{\sigma}}{\Delta}\left(
(\bar{\xi}\partial\bar{\sigma}-\xi\bar{\partial}\sigma)(1+\xi\bar{\xi})
+\xi^2\sigma-\bar{\xi}^2\bar{\sigma}-2i\lambda\right).
\]
We note that in the final sets of expressions, the lower order terms introduced into $E_1$ by the 3-jet identity will have to be added to $E_3$ 
and $E_4$.

Moving to the quadratic first order term, differentiating equation (\ref{e:floweq}) we compute that
\begin{align}
E_2=&\frac{(1+\xi\bar{\xi})^2}{2\Delta}\Big(-2\bar{\partial}\bar{\sigma}\partial\lambda-i\bar{\partial}\bar{\sigma}\bar{\partial}\sigma+2\bar{\partial}\lambda\partial\bar{\sigma}+i\bar{\partial}\sigma\bar{\partial}\bar{\sigma}\nonumber\\
&\qquad\qquad\qquad\qquad
-\frac{1}{\Delta}\bar{\partial}\Delta(-2\bar{\sigma}\partial\lambda-i\bar{\sigma}\bar{\partial}\sigma+2\lambda\partial\bar{\sigma}+i\sigma\bar{\partial}\bar{\sigma})\Big)\nonumber\\
&=\frac{(1+\xi\bar{\xi})^2}{2\Delta^2}\Big(
4\lambda\bar{\sigma}\bar{\partial}\lambda\partial\lambda
+2i\lambda\bar{\sigma}\bar{\partial}\lambda\bar{\partial}\sigma
-2(\lambda^2+\sigma\bar{\sigma})\bar{\partial}\lambda\partial\bar{\sigma}
-2i\lambda\sigma\bar{\partial}\lambda\bar{\partial}\bar{\sigma}\nonumber\\
&\qquad\qquad\qquad\qquad
-2\lambda^2\bar{\partial}\bar{\sigma}\partial\lambda
+2\lambda\sigma\bar{\partial}\bar{\sigma}\partial\bar{\sigma}
+i\sigma^2(\bar{\partial}\bar{\sigma})^2-i\bar{\sigma}^2(\bar{\partial}\sigma)^2\nonumber\\
&\qquad\qquad\qquad\qquad\qquad\qquad\qquad\qquad
-2\bar{\sigma}^2\bar{\partial}\sigma\partial\lambda
+2\lambda\bar{\sigma}\bar{\partial}\sigma\partial\bar{\sigma}\Big)\nonumber\\
&=-\frac{(1+\xi\bar{\xi})^2}{2\Delta^2}\bar{H}_1\nonumber.
\end{align}
This establishes the quadratic first order term, once we recall that $\bar{\partial}\dot{F}=-\dot{\bar{\sigma}}$. 

Moving to the linear first order term
\begin{align}
E_3=&\frac{(1+\xi\bar{\xi})\xi}{\Delta}\left(-2\bar{\sigma}\partial\lambda-i\bar{\sigma}\bar{\partial}\sigma+2\lambda\partial\bar{\sigma}+i\sigma\bar{\partial}\bar{\sigma}\right)\nonumber\\
&\qquad\qquad
+\frac{1+\xi\bar{\xi}}{\Delta}\left(2i(\sigma\xi+i\lambda\bar{\xi})\bar{\partial}\bar{\sigma}+2i\bar{\sigma}(\xi\bar{\partial}\sigma+i\bar{\xi}\bar{\partial}\lambda)\right)\nonumber\\
&\qquad\qquad\qquad
-\frac{1+\xi\bar{\xi}}{\Delta^2}\left[2i\bar{\sigma}(\sigma\xi+i\lambda\bar{\xi})(2\lambda\bar{\partial}-\sigma\bar{\partial}\bar{\sigma}-\bar{\sigma}\bar{\partial}\sigma)\right]\nonumber\\
=&\frac{1+\xi\bar{\xi}}{\Delta^2}\Big\{-2\xi\bar{\sigma}(\lambda^2-\sigma\bar{\sigma})\partial\lambda+
[2\bar{\sigma}^2(i\sigma\xi-\lambda\bar{\xi})+i\xi\bar{\sigma}(\lambda^2-\sigma\bar{\sigma})]\bar{\partial}\sigma\nonumber\\
&\qquad\qquad\qquad
+2\xi\lambda(\lambda^2-\sigma\bar{\sigma})\partial\bar{\sigma}+[i\xi\sigma(3\lambda^2-\sigma\bar{\sigma})-2\bar{\xi}\lambda^3]\bar{\partial}\bar{\sigma}
\nonumber\\
&\qquad\qquad\qquad
+2[\bar{\xi}\bar{\sigma}(\lambda^2+\sigma\bar{\sigma})-2i\xi\lambda\sigma\bar{\sigma}]\bar{\partial}\lambda\Big\}\nonumber.
\end{align}
Adding the linear term from $E_1$ we compute that
\begin{align}
-\bar{H}_3=&\frac{\Delta^2}{1+\xi\bar{\xi}}E_3+i\bar{\xi}\bar{\sigma}(\lambda^2-\sigma\bar{\sigma})\partial\bar{\sigma}
  -i\xi\bar{\sigma}(\lambda^2-\sigma\bar{\sigma})\bar{\partial}\sigma\nonumber\\
=&\frac{1+\xi\bar{\xi}}{\Delta^2}\Big\{-2\xi\bar{\sigma}(\lambda^2-\sigma\bar{\sigma})\partial\lambda+
2\bar{\sigma}^2(i\sigma\xi-\lambda\bar{\xi})\bar{\partial}\sigma\nonumber\\
&\qquad\qquad\qquad
+(\lambda^2-\sigma\bar{\sigma})(2\xi\lambda+i\bar{\xi}\bar{\sigma})\partial\bar{\sigma}+[i\xi\sigma(3\lambda^2-\sigma\bar{\sigma})-2\bar{\xi}\lambda^3]\bar{\partial}\bar{\sigma}
\nonumber\\
&\qquad\qquad\qquad
+2[\bar{\xi}\bar{\sigma}(\lambda^2+\sigma\bar{\sigma})-2i\xi\lambda\sigma\bar{\sigma}]\bar{\partial}\lambda\Big\}\nonumber,
\end{align}
as claimed in the Proposition.

Finally, we work out the zero order term by looking again at the derivative of equation (\ref{e:floweq}) :
\[
E_4=\frac{2i\bar{\sigma}}{\Delta}\Big(\sigma\xi^2+i\lambda(1+2\xi\bar{\xi})\Big),
\]
and taking into account the zero order term of $E_1$, we have
\[
-\bar{H}_4=\Delta^2\;E_4+i\bar{\sigma}(\xi^2\sigma-\bar{\xi}^2\bar{\sigma}-2i\lambda)\Delta
   =\bar{\sigma}(3i\xi^2\sigma-i\bar{\xi}^2\bar{\sigma}-4\xi\bar{\xi}\lambda)\Delta,
\]
as claimed.
\end{proof}

\vspace{0.1in}

\begin{Cor}\label{c:flowdet}
Under the mean curvature flow the determinant of the induced metric evolves by:
\[
\left(\frac{\partial }{\partial s}-{\mathbb G}^{jk}\partial_j\partial_k\right)\Delta=\frac{H_8(1+\xi\bar{\xi})^2}
{2\Delta^2}+\frac{H_9(1+\xi\bar{\xi})+H_{10}}{\Delta}-4\lambda,
\]
where
\begin{align}
H_8&=2{\mathbb R}{\mbox{e}}[2i(\sigma \bar{\sigma}-3 \lambda^{2}) \bar{\sigma} \left(\partial \lambda\right)^{2} 
+2\lambda(\lambda^{2}+\sigma \bar{\sigma}) \partial \lambda \bar{\partial} \lambda 
+4i \lambda \bar{\sigma}^{2} \partial \lambda \partial \sigma  \nonumber\\
&\qquad-4\sigma \bar{\sigma}^{2} \partial \lambda \bar{\partial} \sigma
+4i \lambda^{3} \partial \lambda \partial \bar{\sigma} 
-4\lambda^{2} \sigma \partial \lambda \bar{\partial} \bar{\sigma}
-i \bar{\sigma}^{3}\left(\partial \sigma\right)^{2} 
+2\lambda \bar{\sigma}^{2} \partial \sigma \bar{\partial} \sigma  \nonumber\\
&\qquad -2i \sigma \bar{\sigma}^{2} \partial \sigma \partial \bar{\sigma} 
+\lambda \sigma \bar{\sigma} \partial \sigma \bar{\partial} \bar{\sigma} 
+i (2\lambda^{2}-\sigma \bar{\sigma}) \bar{\sigma} \left(\bar{\partial} \sigma\right)^{2}
-2\lambda^{3} \bar{\partial} \sigma \partial \bar{\sigma} 
+3\lambda \sigma \bar{\sigma} \bar{\partial} \sigma \partial \bar{\sigma}], \nonumber
\end{align}
\[
H_9={\mathbb R}{\mbox{e}}\left[(-2i\lambda\bar{\xi}-4\sigma\xi)\bar{\sigma}\partial\lambda
    -i\bar{\sigma}^2\bar{\xi}\partial\sigma+(-3i\sigma\xi+4\lambda\bar{\xi})\bar{\sigma}\bar{\partial}\sigma\right]/2,
\]
and
\[
H_{10}=2{\mathbb R}{\mbox{e}}\;\sigma\bar{\sigma}\left[i(\sigma\xi^2-\bar{\sigma}\bar{\xi}^2)-2\lambda\xi\bar{\xi}\right].
\]

\end{Cor}
\begin{proof}
This follows from the previous evolution equations and the definition $\Delta=\lambda^2-|\sigma|^2$.
\end{proof}
\vspace{0.1in}

\begin{Prop}\label{p:chiflow}
Under mean curvature flow the perpendicular distance function $\chi$ (see Definition \ref{d:perpdist}) evolves by
\begin{align}
\left(\frac{\partial }{\partial s}-{\mathbb G}^{jk}\partial_j\partial_k\right)\chi^2&
  =4\left[i\bar{F}^2\bar{\sigma}-2F\bar{F}\lambda-iF^2\sigma\right.\nonumber\\
&\qquad\qquad
   +[i\sigma\xi(\bar{F}\bar{\sigma}-F\bar{\rho})-i\bar{\sigma}\bar{\xi}(F\sigma-\bar{F}\rho)](1+\xi\bar{\xi})\nonumber\\
&\qquad\qquad\qquad+\left.\lambda(\rho\bar{\rho}-\sigma\bar{\sigma})(1+\xi\bar{\xi})^2\right]/[(1+\xi\bar{\xi})^2(\lambda^2-\sigma\bar{\sigma})].\nonumber
\end{align}
\end{Prop}
\begin{proof}
This follows from differentiating the expression for $\chi(\xi,\bar{\xi},F,\bar{F})$ and using the equations for the flow of $F$.
\end{proof}
\vspace{0.1in}

\begin{Prop}
The flow satisfies
\begin{align}
\left(\frac{\partial }{\partial s}-{\mathbb G}^{jk}\partial_j\partial_k\right)\left(\frac{|\sigma|^2}{\lambda^2-|\sigma|^2}\right)&=-2\frac{(\lambda^2+|\sigma|^2)}{(\lambda^2-|\sigma|^2)^3}
    \Big\|\lambda d|\sigma|-|\sigma|d\lambda\Big\|^2
     -2\frac{\lambda^2|\sigma|^2}{(\lambda^2-|\sigma|^2)^2}
    \Big\|d\phi\Big\|^2\nonumber\\
&\qquad +\frac{ 4\lambda|\sigma| }{(\lambda^{2}-|\sigma|^2)^3} \;  {\mathbb R}{\mbox{e}}\;[H_{11}]  \label{e:holconv},
\end{align}
where
\begin{align}
H_{11}=&(1+ \xi \bar{\xi})\left[i\lambda\bar{\sigma} \bar{\xi}\partial|\sigma|-i|\sigma|\bar{\sigma} \bar{\xi}\partial \lambda
+2\lambda|\sigma|( i \lambda \xi- \bar{\sigma} \bar{\xi}) \partial\phi\right]\nonumber\\
&\qquad -2 i\lambda \sigma|\sigma| \xi^{2}-2|\sigma|^{3}+2( 1+2 \xi \bar{\xi})\lambda^{2}|\sigma|  \nonumber,
\end{align}
and $\phi$ is the argument of $\sigma$. Here the norm $\|.\|$ is taken with respect to the induced metric given in Proposition \ref{p:ts2indmet}.
\end{Prop}
\begin{proof}
For the sake of brevity, introduce the heat operator ${\frak P}$:
\[
{\frak P}=\frac{\partial }{\partial s}-{\mathbb G}^{jk}\partial_j\partial_k
\]
Then the result follows from the fact that
\begin{align}
{\frak P}\left(\frac{|\sigma|^2}{\lambda^2-|\sigma|^2}\right)&=\frac{\lambda}{(\lambda^2-|\sigma|^2)^2}\left[\lambda\sigma{\frak P}(\bar{\sigma})
   +\lambda\bar{\sigma}{\frak P}(\sigma)
   -2|\sigma|^2{\frak P}(\lambda)\right]\nonumber\\
&\qquad -\frac{2|\sigma|^2(3\lambda^2+  |\sigma|^2)}{(\lambda-|\sigma|^2)^3}\|d\lambda\|^2
     -\frac{2\lambda^2(\lambda^2+  |\sigma|^2)}{(\lambda-|\sigma|^2)^3}<<d\sigma,d\bar{\sigma}>>\nonumber\\
&\qquad  +\frac{4\lambda(\lambda^2+|\sigma|^2)}{(\lambda-|\sigma|^2)^3}<<\sigma d\bar{\sigma}+\bar{\sigma}d\sigma,d\lambda>> 
   -\frac{2\lambda^2\sigma^2}{(\lambda-|\sigma|^2)^2}\|d\bar{\sigma}\|^2\nonumber\\
&\qquad -\frac{2\lambda^2\bar{\sigma}^2}{(\lambda-|\sigma|^2)^2}\|d\sigma\|^2\nonumber,
\end{align}
and the flow equations given in Proposition \ref{p:shearflow}, recalling that $\lambda={\mathbb I}{\mbox{m}}\;\rho$. 
\end{proof}

\vspace{0.1in}

\subsection{Boundary conditions}\label{s:bdryinit}

In our case we would like the boundary surface to be the totally real Lagrangian hemisphere
$\Sigma$, but, as the metric will
be Lorentz or degenerate on such a surface (see Proposition
\ref{p:indmet}), a Lagrangian surface can never be positive and so cannot be used as a
boundary condition. Instead we perturb the hemisphere to make it positive, and attach the initial disc to this perturbed surface.

More specifically, suppose $\Sigma$ is given by $\eta=F(\xi,\bar{\xi})$.
Define the perturbed surface $\tilde{\Sigma}_{C_0}$ by adding a {\it linear holomorphic twist}: 
\begin{equation}\label{e:addtwist}
\eta=\tilde{F}=F-iC_0\xi,
\end{equation}
where $C_0$ is a real positive constant.

\begin{Prop}\label{p:linholtwist}
For any closed $K\subset\{|\xi|<1\}$, there exists $C_0>0$ such that $\tilde{\Sigma}_{C}$ is positive on $K$ for all $C>C_0$ .
As we make $C_0$ large, the positive area containing $\xi=0$ becomes bigger, and tends to an open hemisphere as $C_0\rightarrow\infty$.

\end{Prop}
\begin{proof}
Since the deformation is holomorphic, the shear remains the same: $\tilde{\sigma}=\sigma$ and, computing the twist of $\tilde{\Sigma}$
\[
\tilde{\lambda}={\mathbb I}{\mbox m}\;(1+\xi\bar{\xi})^2\partial\left[\frac{\tilde{F}}{(1+\xi\bar{\xi})^2}\right]
={\mathbb I}{\mbox m}\;(1+\xi\bar{\xi})^2\partial\left[\frac{F-iC_0\xi}{(1+\xi\bar{\xi})^2}\right]=-C_0\frac{1-\xi\bar{\xi}}{1+\xi\bar{\xi}},
\]
where we have used the fact that $\Sigma$ is Lagrangian. For $C_0>|\tilde{\sigma}(0)|$ 
\[
\left.\left(\tilde{\lambda}^2-|\tilde{\sigma}|^2\right)\right|_0=C_0^2-|\tilde{\sigma}(0)|^2>0,
\]
i.e. the metric at $0$ is positive definite.
\end{proof}

\vspace{0.1in}

\begin{Def}
Fix $C_0>|\tilde{\sigma}(0)|$ and denote the set on which the induced metric is positive by
\[
\tilde{\Lambda}_{C_0}'=\left\{\;\gamma\in\tilde{\Sigma}\;\;\left|\;\; |\tilde{\sigma}(\gamma)|<|\tilde{\lambda}(\gamma)|\;\right.\right\}.
\] 
Clearly $\tilde{\Lambda}_{C_0}'$ is non-empty since it contains $\gamma_0$. Denote the connected component of $\gamma_0$ in $\tilde{\Lambda}_{C_0}'$ 
by $\tilde{\Lambda}_{C_0}$. 
\end{Def}

\vspace{0.1in}

\begin{Note}
In order for the induced metric to be positive (rather than
negative) definite we have arranged that $\tilde{\lambda}<0$ - see Proposition \ref{p:ts2indmet}.
\end{Note}
\vspace{0.1in}

\begin{Note}\label{n:H}
Along the edge, aside from the initial angle and $|\sigma(s=0)|^2$ having to be constant, we also have the following compatibility condition:
\begin{equation}\label{e:compat}
{\mbox{ Pr}}_{ND\cap N\tilde{\Sigma}}H=0,
\end{equation}
on $f_0(\partial D)$.  In terms of the components of an adapted frames along the edge, this is $<H,f^{(2)}>=0$. 

To see this, note that from Corollary \ref{c:multi}, if $B$ is the hyperbolic angle between the planes, the following relations hold:
\[
\mathring{\tilde{e}}_{(1)}=\mathring{e}_{(1)},
\qquad
\mathring{\tilde{e}}_{(2)}=\cosh B\;\mathring{e}_{(2)}+\sinh B \;\mathring{f}_{(1)},
\]
\[
\mathring{\tilde{f}}_{(1)}=\sinh B\;\mathring{e}_{(2)}+\cosh B \;\mathring{f}_{(1)},
\qquad
\mathring{\tilde{f}}_{(2)}=\mathring{f}_{(2)}.
\]

Now, suppose we write the flow equation as
\[
\frac{\partial}{\partial s}f=X,
\]
where $X^\perp=H$. Moreover, at the edge we must have $X\in T\tilde{\Sigma}$, and so
\[
X=X^{(1)}\mathring{\tilde{e}}_{(1)}+X^{(2)}\mathring{\tilde{e}}_{(2)}
=X^{(1)}\mathring{e}_{(1)}+X^{(2)}(\cosh B\;\mathring{e}_{(2)}+\sinh B \;\mathring{f}_{(1)}),
\]
thus
\[
X^\perp=X^{(2)}\sinh B \;\mathring{f}_{(1)}=H=H^{({1})}\mathring{f}_{(1)}+H^{({2})}\mathring{f}_{(2)}.
\]
So we have $H^{(2)}=0$ as claimed and $X^{(2)}=H^{({1})}/\sinh B$. Note also that we can parameterize the edge so that
$X^{(1)}=0$, that is, the deformation vector is perpendicular to the edge in the boundary surface.
\end{Note}

\vspace{0.1in}

We now write down the evolution of the edge in graph coordinates.

\begin{Prop}
In graph coordinates, the normal flow of the edge of the flowing disc is
\[
\frac{\partial \xi}{\partial s}=\frac{\lambda iH^\xi-\bar{\sigma}H^{\bar{\xi}}}{\rho-\tilde{\rho}}.
\]
\end{Prop}
\begin{proof}
Let the flowing edge be parameterized by $u\rightarrow \xi(u,s)$. Then the tangent to the edge is given by
\[
\frac{\partial}{\partial u}=\frac{\partial \xi}{\partial u}\frac{\partial}{\partial \xi}
                        +\frac{\partial \bar{\xi}}{\partial u}\frac{\partial}{\partial \bar{\xi}}
   =\xi'\frac{\partial}{\partial \xi}
                        +\bar{\xi}'\frac{\partial}{\partial \bar{\xi}},
\]
and the flow is
\[
\frac{\partial}{\partial s}=\frac{\partial \xi}{\partial s}\frac{\partial}{\partial \xi}
                        +\frac{\partial \bar{\xi}}{\partial s}\frac{\partial}{\partial \bar{\xi}}
   =\dot{\xi}\frac{\partial}{\partial \xi}
                        +\dot{\bar{\xi}}\frac{\partial}{\partial \bar{\xi}}.
\]
In order to factor out reparameterization of the edge, let us choose one such that the flow is orthogonal to the edge (see Note \ref{n:H}). 
This means that $\dot{\xi}=ai\xi'$ for some real function $a$.

Along the edge we have the intersection condition
\[
F_s(\xi(u,s),\bar{\xi}(u,s))=\tilde{F}(\xi(u,s),\bar{\xi}(u,s)).
\]
Differentiating this w.r.t. $u$ we find that
\[
\partial(F-\tilde{F})\xi'+\bar{\partial}(F-\tilde{F})\bar{\xi}'=0,
\]
or in terms of the shear, divergence and twist:
\[
(\rho-\tilde{\rho})\xi'-(\bar{\sigma}-\bar{\tilde{\sigma}})\bar{\xi}'=0.
\]
Substituting $\dot{\xi}=ai\xi'$ in this we have
\[
(\rho-\tilde{\rho})\dot{\xi}+(\bar{\sigma}-\bar{\tilde{\sigma}})\dot{\bar{\xi}}=0.
\]
On the other hand, differentiating the intersection condition in time we get
\[
\dot{F}+(\rho-\tilde{\rho})\dot{\xi}-(\bar{\sigma}-\bar{\tilde{\sigma}})\dot{\bar{\xi}}=0.
\]
Combining these last two equations yields
\[
\dot{\xi}=-\frac{\dot{F}}{2(\rho-\tilde{\rho})}.
\]
Note that by fixing the hyperbolic angle $\rho\neq\tilde{\rho}$, and so this is well-defined. 

Finally, since $\dot{F}=-2\lambda iH^\xi+2\bar{\sigma}H^{\bar{\xi}}$ we get the stated result.

\end{proof}

\vspace{0.2in}


\section{{\bf Existence Results for the I.B.V.P.}}

In this section we establish sufficient conditions, namely smallness of the initial angle and aholomorphicity along the edge, for which
long-time existence of the {\bf I.B.V.P.} holds. We then prove the existence of a holomorphic disc with edge lying on a totally real
Lagrangian hemisphere.

\subsection{Short-time existence}

Short-time existence of quasilinear parabolic equations of various types can be established by comparison with Schauder theory for linear equations.  
In what follows we prove this fact for the {\bf I.B.V.P.} 

Consider then a 1-parameter family of maps over a fixed domain $f_s:D\rightarrow TS^2$ which has local coordinate expression 
$\nu\mapsto(\xi(s,\nu,\bar{\nu}),\eta(s,\nu,\bar{\nu}))$. Let $\tilde{\Sigma}$ be a positive surface in $TS^2$ which is given as a graph 
$\eta=\tilde{F}(\xi,\bar{\xi})$ and $\Sigma_0$ some initial positive surface with edge lying in $\tilde{\Sigma}$.

We can now establish short-time existence for the {\bf I.B.V.P.}:

\begin{Thm}\label{t:ste}
Let $f_0:D\rightarrow TS^2$ be a smooth positive section whose edge lies in a fixed positive section $\tilde{\Sigma}$. Assume, in addition, that 
$\tilde{\Sigma}$ is totally real along the edge $f_0(\partial D)$.
 
Then there exists a unique family of positive sections $f_s(D)$ with
\[
f_s\in C^{2+\alpha}(\overline{D}\times[0,s_0)),
\]
satisfying the {\bf I.B.V.P.} on an interval $0\leq s< s_0$. 
\end{Thm}

\vspace{0.1in}

\begin{proof}
We are dealing with a parabolic system with mixed non-linear boundary conditions  (see \cite{Freire} for the codimension one case).  

Following section I.2.3 in \cite {Eid}, we first consider the linearization of the system at $f=(\xi,\eta)$, namely:
\[
\left(\frac{\partial }{\partial s}-a_{\alpha\beta}(\nabla f)D^\alpha D^\beta\right)\hat{f}=g,
\]
where the linearized initial and boundary conditions for $\hat{f}=(\hat{\xi},\hat{\eta})$ are
\begin{enumerate}
\item[(i)]  $\hat{f}(s=0)=f_0$,
\item[(ii)] $\hat{\eta}=\delta\tilde{F}$,
\item[(iii)] $4{\mathbb I}m(\bar{\beta}\partial_v\hat{\xi}+\alpha\partial_v\bar{\hat{\eta}})=0$,
\item[(iv)] ${\mathbb R}e [Ce^{-i\phi}(\partial_u\hat{\xi}\partial_v\eta-\partial_v\hat{\xi}\partial_u\eta
    +\partial_u\xi\partial_v\hat{\eta}-\partial_v\xi\partial_u\hat{\eta})]=0$,
\end{enumerate}
and, for brevity, we have introduced
\begin{equation}\label{e:def1}
Ce^{i\phi}=\partial_u\xi\partial_v\eta-\partial_v\xi\partial_u\eta.
\end{equation}

To show that this is a parabolic system we must check that the Lopatinskii-Shapiro conditions (equations (2.18) to (2.20) of \cite{Eid}) hold. 
We do this as follows.

Fix a point $p\in\partial D$ and, by an isometry of $TS^2$, set $\xi(f(p))=\eta(f(p))=0$. We retain the 
freedom to rotate $(\xi,\eta)\rightarrow (e^{i\theta}\xi,e^{-2i\theta}\eta)$, which we will implement shortly.

In addition, let us choose a parameterization $\nu=u+iv$ of the domain $D$ such that at the point $p$, the edge is given by $v=0$, and 
\[
\mathring{e}_{(1)}=f_*\left(\frac{\partial}{\partial u}\right),
\qquad\qquad
\mathring{e}_{(2)}=f_*\left(\frac{\partial}{\partial v}\right),
\]
form an orthonormal basis for $T_{f(p)}f(D)$ and
\[
\triangle_p=\left.\frac{\partial^2}{\partial u^2}\right|_p+\left.\frac{\partial^2}{\partial v^2}\right|_p.
\]
Note that orthonormality of the frame at $p$ implies that
\begin{equation}\label{e:onf1}
{\textstyle{\frac{1}{2i}}}(\partial_u\xi\partial_u\bar{\eta}-\partial_u\bar{\xi}\partial_u\eta)=1,
\end{equation}
\begin{equation}\label{e:onf2}
{\textstyle{\frac{1}{2i}}}(\partial_v\xi\partial_v\bar{\eta}-\partial_v\bar{\xi}\partial_v\eta)=1,
\end{equation}
\begin{equation}\label{e:onf3}
\partial_u\xi\partial_v\bar{\eta}-\partial_u\bar{\xi}\partial_v\eta-\partial_v\bar{\xi}\partial_u\eta+\partial_v\xi\partial_u\bar{\eta}=0.
\end{equation}
An orthonormal frame for $T_{f(p)}\tilde{\Sigma}$ is given by
\[
\tilde{e}_1=e_1,
\qquad\qquad
\tilde{e}_2=2{\mathbb{R}}e\left(\alpha\frac{\partial}{\partial \xi}+\beta\frac{\partial}{\partial \eta}\right),
\]
for some $\alpha,\beta\in{\mathbb C}$ where positivity of $\tilde{\Sigma}$ implies that $\beta\neq 0$. The fixed angle condition is 
\[
{\mathbb G}(e_2,\tilde{e}_2)=4{\mathbb I}m(\bar{\beta}\partial_v\xi+\alpha\partial_v\bar{\eta}) =\cosh B.
\]
On the other hand, the asymptotic holomorphicity boundary condition is
\[
|\partial_u\xi\partial_v\eta-\partial_v\xi\partial_u\eta|^2=C^2(1+s)^{-2}.
\]

Take the Fourier transform in the variable $u$ and the Laplace transform in the variable $s$. If $w$ and $t$ are the transformed variables,
respectively, we get the transformed ODE's
\[
\frac{\partial^2 \hat{\xi}}{\partial v^2}-(t+w^2)\hat{\xi}=0,
\qquad\qquad
\frac{\partial^2 \hat{\eta}}{\partial v^2}-(t+w^2)\hat{\eta}=0.
\]
We must now verify that these equations have a solution that decay for $v\rightarrow-\infty$ and satisfy the transformed boundary conditions
\begin{enumerate}
\item[(ii)]  $\hat{\eta}(0)=\delta\tilde{F}$,
\item[(iii)] $4{\mathbb I}m(\bar{\beta}\partial_v\hat{\xi}+\alpha\partial_v\bar{\hat{\eta}})(0)=0$,
\item[(iv)]  
\begin{align}
&Ce^{-i\phi}(iw\hat{\xi}\partial_v\eta-\partial_v\hat{\xi}\partial_u\eta+\partial_u\xi\partial_v\hat{\eta}-iw\partial_v\xi\hat{\eta})(0)\nonumber\\
&\qquad+Ce^{i\phi}(iw\bar{\hat{\xi}}\partial_v\bar{\eta}-\partial_v\bar{\hat{\xi}}\partial_u\bar{\eta}
           +\partial_u\bar{\xi}\partial_v\bar{\hat{\eta}}-iw\partial_v\bar{\xi}\bar{\hat{\eta}})(0)=0.\nonumber
\end{align}
\end{enumerate}
To do this, first solve the ODE's
\[
\hat{\xi}=\hat{\xi}(0)\exp(v\sqrt{t+w^2}),
\qquad\qquad
\hat{\eta}=\hat{\eta}(0)\exp(v\sqrt{t+w^2}),
\]
and then substituting this in the boundary conditions (i) to (iii), we are led to the following linear system for 
$\hat{U}=[\hat{\xi}(0),\bar{\hat{\xi}}(0),\hat{\eta}(0),\bar{\hat{\eta}}(0)]^T$:
\[
\hat{M}\hat{U}=\hat{V},
\]
where
\[
\hat{M}=\left[
\begin{array}{cccc}
0&0&1&0 \\
0&0&0&1 \\
\bar{\beta}&-\beta&-\bar{\alpha}&\alpha \\
Ce^{-i\phi}(iw\partial_v\eta-\sqrt{t+w^2}\partial_u\eta)&Ce^{i\phi}(iw\partial_v\bar{\eta}-\sqrt{t+w^2}\partial_u\bar{\eta})& * & * \\
\end{array}
\right].
\]
Thus the linearized system has a unique solution iff the determinant of this matrix is non-zero. Now
\[
|\hat{M}|=\bar{\beta}Ce^{i\phi}(iw\partial_v\bar{\eta}-\sqrt{t+w^2}\partial_u\bar{\eta})
      +\beta Ce^{-i\phi}(iw\partial_v\eta-\sqrt{t+w^2}\partial_u\eta).
\]
Clearly a necessary condition for this expression to be non-zero is that $C\neq0$. However, this is true by our asymptotic holomorphicity
condition, so long as it is true initially. Thus the edge of the initial disc must be totally real - which we now assume. 

Motivated by this expression, let us now exhaust our coordinate freedom by using the rotation $(\xi,\eta)\rightarrow (e^{i\theta}\xi,e^{-2i\theta}\eta)$
to set (recalling definition (\ref{e:def1}))
\[
\bar{\beta}Ce^{i\phi}\partial_v\bar{\eta}=-\beta Ce^{-i\phi}\partial_v\eta.
\]
Then
\begin{align}
|\hat{M}|=&-\sqrt{t+w^2}(\bar{\beta}Ce^{i\phi}\partial_u\bar{\eta} +\beta Ce^{-i\phi}\partial_u\eta) \nonumber\\
&=\frac{\sqrt{t+w^2}}{\partial_v\bar{\eta}}\beta Ce^{-i\phi}(\partial_v\eta\partial_u\bar{\eta}-\partial_u\eta\partial_v\bar{\eta}).\nonumber
\end{align}
This is well-defined since, by equation (\ref{e:onf2}) we have $\partial_v\eta\neq 0$. 

Now, for the sake of contradiction, suppose that $|\hat{M}|=0$. Then from the above expression and the fact that $\beta\neq 0$ and $C\neq 0$, 
we see that this implies that
$\partial_u\eta=\lambda\partial_v\eta$ for some non-zero $\lambda\in{\mathbb R}$. Substituting this in equations (\ref{e:onf1}) and
(\ref{e:onf3}) we find that
\begin{equation}\label{e:onf4}
{\textstyle{\frac{1}{2i}}}(\partial_u\xi\partial_v\bar{\eta}-\partial_u\bar{\xi}\partial_v\eta)={\textstyle{\frac{1}{\lambda}}},
\end{equation}
\begin{equation}\label{e:onf5}
(\partial_u\xi+\lambda\partial_v\xi)\partial_v\bar{\eta}-(\partial_u\bar{\xi}+\lambda\partial_v\bar{\xi})\partial_v\eta=0.
\end{equation}
Now adding equation (\ref{e:onf4}) to $\lambda$ times equation (\ref{e:onf5}) yields
\[
{\textstyle{\frac{1}{2i}}}[(\partial_u\xi+\lambda\partial_v\xi)\partial_v\bar{\eta}-(\partial_u\bar{\xi}+\lambda\partial_v\bar{\xi})\partial_v\eta]=
   \lambda+{\textstyle{\frac{1}{\lambda}}}.
\]
Comparing this with (\ref{e:onf5}) we find that $\lambda+\lambda^{-1}=0$ which is impossible. We conclude that $|\hat{M}|\neq0$, and hence the 
Lopatinskii-Shapiro conditions hold. Thus the  boundary conditions (ii) to (iv)
of the system {\bf I.B.V.P} satisfy the complementarity condition, and therefore,
coupled with Cauchy initial data (i), is strongly parabolic \cite{Eid}. 

The proof of short-time existence now proceeds as follows.
Consider the set $Q_T=\overline{D}\times[0,T]$ and the space 
\[
B_R^T=\{f\in C^{2+\alpha}(Q_T, {\mathbb R}^4)\;|\; f(s=0)=f_0,\; \|f-f_0\|_{1+\alpha}<R  \}.
\]

Given $f\in B_R^T$, we can solve the linearized initial boundary value problem in the beginning
of the proof to obtain $\hat{f}$, 
which by strict parabolicity and Theorem VI.21 of \cite{Eid} satisfies 
the following a priori estimate:

\[
\|\hat{f}\|_{C^{2+\alpha}(Q_T)}\leq c\left(\|g\|_{C^{\alpha}(Q_T)}
+\|f_0\|_{C^{\alpha}(D)}+\sum_q \|h^q\|_{C^{2+\alpha-r_q}(\partial D\times[0,T])}   \right),
\]
\noindent
where $g(x,t)$ depends on the induced metric and ambient Christoffel symbols at $f$
literally as in (5.1) of \cite{LaS}, and $h$ on the boundary data at $f$ as they appear in (ii) - (iv)
above. Note that if $\hat{f}=f$, then $f$ is a solution of the quasilinear system, and therefore we seek 
fixed points  of the map ${\mathcal G}: f\rightarrow\hat{f}$.

\begin{Lem}
There exists an $R$ and $T$ depending on $f_0$ such that ${\mathcal G}: B_R^T\rightarrow B_R^T$.
\end{Lem}
\begin{proof} 
By the above Schauder estimate, we have a uniform $C^{2+\alpha}$ - bound for every $ {\mathcal G}(f)$ 
for all $f \in B^T_R $.  As a result, we may bound the $C^{1+\alpha}$ - norm of $\hat f$
by a $c_1 T^\epsilon$. Therefore, choosing $T$ sufficiently small, the claim follows.
\end{proof}
Short-time existence now follows from an application of the Schauder fixed point theorem
since when considered in the $C^{1+\alpha} $ - norm, $B_R^T$ is compact and convex.
\end{proof}

\vspace{0.1in}

\subsection{Long-time existence}

Let $\Sigma$ be a totally real Lagrangian section of $TS^2$ that projects to a hemisphere and $\gamma_0\in\Sigma$ be the pole. 
By an isometry we can make $\gamma_0$ 
have Gauss coordinate $\xi=0$. Assume $\Sigma$ is totally real, that is, it is nowhere holomorphic: $|\sigma|\neq0$.

Suppose we add a holomorphic twist (as in Proposition \ref{p:linholtwist}):
\[
\tilde{F}=F_0-iC_0\xi,
\]
to form the section $\tilde{\Sigma}$. While the aholomorphicity remains unchanged ($|\tilde{\sigma}|=|\sigma|$), $\tilde{\Sigma}$ is 
not Lagrangian in the open unit hemisphere, since
\[
\tilde{\lambda}=-C_0\frac{1-R^2}{1+R^2},
\]
where $R=|\xi|$. Fix $C_0>|\tilde{\sigma}(0)|$ and, as before, denote the set on which the induced metric is positive by
\[
\tilde{\Lambda}_{C_0}'=\left\{\;\gamma\in\tilde{\Sigma}\;\;\left|\;\; |\tilde{\sigma}(\gamma)|<|\tilde{\lambda}(\gamma)|\;\right.\right\}.
\] 
Clearly $\tilde{\Lambda}_{C_0}'$ is non-empty since it contains $0$. Denote the connected component of $0$ in $\tilde{\Lambda}_{C_0}'$ 
by $\tilde{\Lambda}_{C_0}$. Thus, $\tilde{\Lambda}_{C_0}$ is a positive section over an open subset in the unit hemisphere. 

In what follows we take the sup and inf of $|\tilde{\sigma}|$ over the hemisphere, while for other quantities, such as $|\tilde{\lambda}|$, sup and inf
will be over the positive region $\tilde{\Lambda}_{C_0}$.

Denote the maximum
Gauss radius of the positive region by
\[
R_0=\sup_{\tilde{\Lambda}_{C_0}} |\xi|.
\]
The following picture illustrates these definitions.

\vspace{0.1in}
\includegraphics{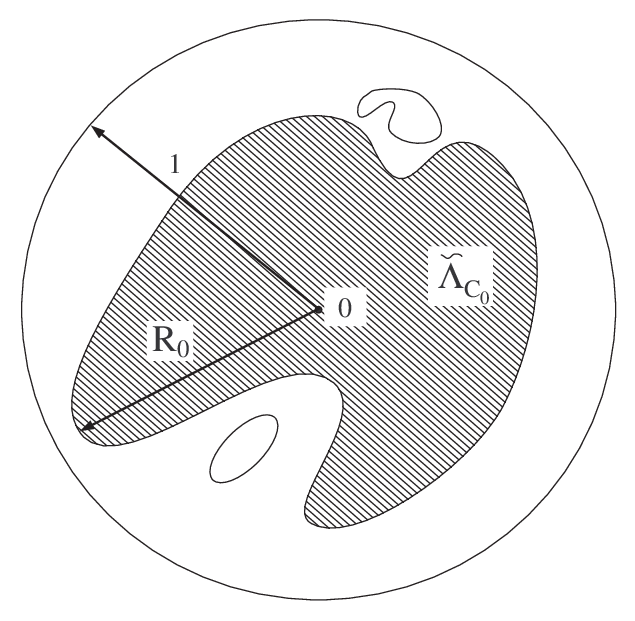}
\vspace{0.1in}

Note that $0<R_0<1$ and 
\[
\sup_{\tilde{\Lambda}_{C_0}}|\tilde{\lambda}|=C_0,
\qquad\qquad
\inf_{\tilde{\Lambda}_{C_0}}|\tilde{\lambda}|=C_0\frac{1-R_0^2}{1+R_0^2}.
\]

In this section we prove the following:

\begin{Thm}\label{t:lte}
Let $\tilde{\Sigma}$, $C_0$ and $R_0$ be as above. Then there exist constants $B_1$ and $C_1$ s.t. $\forall B,C$ satisfying
\[
0<B<B_1(C_0,R_0,\sup|\tilde{\sigma}|,\inf|\tilde{\sigma}|)
\qquad\qquad
0<C<C_1(C_0,R_0,\sup|\tilde{\sigma}|,\inf|\tilde{\sigma}|),
\]
the {\bf I.B.V.P.} with initial constants $B$ and $C$ has a solution for all time.
\end{Thm}
\begin{proof}
In order to prove long-time existence for the {\bf I.B.V.P.} we show that it is uniformly parabolic and, since it is quasilinear,
we need a global gradient estimate. 

In section \ref{s:compactmcf} we proved long-time existence for compact flowing submanifolds, so long as the
ambient space satisfies the timelike curvature condition and the submanifold stays in a compact region of the ambient space. The proof was based on the 
maximum principle to bound the gradient and the mean curvature. These arguments extend to the case where these functions have an interior maximum 
and so, to prove this theorem, we extend the estimates to the edge.

More specifically, we ensure that the timelike curvature condition holds for the flow and that the surface stays in a compact region of $TS^2$.
This we do in Propositions \ref{p:tcc} and \ref{p:compact} below, respectively.

We also find estimates for the gradient and mean curvature vector at the edge. The first of these is established by ensuring that the edge
stays in a positive region of the boundary surface (Proposition \ref{p:posregion}) and the first derivative remains bounded 
(Proposition \ref{p:bdryest1}). To establish the latter, we first control the angles which the edge makes with the canonical orthonormal frame
(Proposition \ref{p:anglecontrol}) and then extract bounds on the 2-jet of the flowing surface in terms of the 2-jet of the boundary surface 
(Theorem \ref{t:bdryest3}). Thus we bound $|H|$ and have established the gradient estimate (see Proposition \ref{p:gradest}).

Once we have the global gradient estimate, the proof proceeds as follows. By the short-time existence Theorem \ref{t:ste} we know that there exists 
an $s_0$ such that there is a solution to initial boundary value problem with
\[
f\in C^{2+\alpha}(\overline{D}\times[0,s_0)). 
\]
By interior and boundary regularity as in \cite{thorpe} \cite{white}, 
the gradient estimate implies that this solution can be extended to
\[
f\in C^{2+\alpha}(\overline D\times[0,s_0]). 
\]
Let $s_{max}$ be the maximum time of existence for the initial boundary problem:
\[ 
s_{max}={\mbox{inf}}_{U\subset\subset D}\{ {\mbox{sup}}\;s \in{\mathbb R}_{\geq0}\;|\;  f\in C^{2+\alpha}(U\times[0,s)) \}.
\]
Assume for the sake of contradiction that $s_{max}<\infty$. That is, there exists $U\subset\subset D$ and a sequence of times 
$\{s_j\}\rightarrow s_{max}$ such that 
\[
\lim_{j\rightarrow\infty}\|f_{s_j}\|_{C^{2+\alpha}(U)}\rightarrow\infty.
\]
Given the a priori gradient estimate, this contradicts the regularity result of \cite{thorpe} and so we have that $s_{max}$ is 
infinite.

\end{proof}

\subsubsection{Interior gradient estimate}

We start by showing that the timelike curvature condition holds along the flow and that the flow stays in a compact region. First,

\begin{Prop}\label{p:tcc}
The timelike curvature condition (\ref{e:tcc}) holds for as long as the flow exists.
\end{Prop}
\begin{proof}
Given a timelike plane $P$, chose an adapted orthonormal frame $E_{(\mu)}=\{e_{1)},e_{(2)},f_{(1)},f_{(2)}\}$ and for a timelike vector 
$X=X^{(1)}f_{(1)}+X^{(2)}f_{(2)}$ compute
\[
<\overline{R}(X,e_{(a)})X,e_{(a)}>=\frac{2|\sigma|}{\lambda^2-|\sigma|^2}(-(X^{(1)})^2+(X^{(2)})^2)\geq\frac{2|\sigma|}{\lambda^2-|\sigma|^2}{\mathbb G}(X,X).
\]
From Proposition \ref{p:shearest} we have that
\[
\left(\frac{\partial }{\partial s}-{\mathbb G}^{jk}\partial_j\partial_k\right)\left(\frac{|\sigma|^2}{\lambda^2-|\sigma|^2}\right)\leq0,
\]
and so we have the a priori bound
\[
\frac{|\sigma|^2}{\lambda^2-|\sigma|^2}\leq C_1.
\]
In addition, we claim that under the flow
\[
\Delta=\lambda^2-|\sigma|^2\geq C_2.
\]
To see this, note that at the edge this follows from Corollary \ref{c:bdryests}. On the other hand at an interior minimum of $\Delta$ we have that 
$\partial\Delta=2\lambda\partial\lambda-\sigma\partial\bar{\sigma}-\bar{\sigma}\partial\sigma=0$.
Substituting this in the flow for $\Delta$, as expressed in Corollary \ref{c:flowdet}, we find that
\begin{align}
\left(\frac{\partial }{\partial s}-{\mathbb G}^{jk}\partial_j\partial_k\right)\Delta&=\frac{2\Delta}{\lambda^2}\|da\|^2
    +\frac{2ai(1+\xi\bar{\xi})^2}{\lambda}[\bar{\partial}a(\partial\phi-{\textstyle{\frac{2i\bar{\xi}}{1+\xi\bar{\xi}}}})
          -\partial a(\bar{\partial}\phi+{\textstyle{\frac{2i\xi}{1+\xi\bar{\xi}}}})]\nonumber\\
&\qquad\qquad\qquad\qquad+a^2\|d\phi-2j\;d\ln(1+\xi\bar{\xi})\|^2-4\lambda,
\end{align}
where we have introduced $\sigma=ae^{i\phi}$ and $\|da\|^2={\Bbb{G}}^{jk}\partial_ja\;\partial_ka$ is the square norm 
of the gradient.

Now, we have that
\[
i[\bar{\partial}a(\partial\phi-{\textstyle{\frac{2i\bar{\xi}}{1+\xi\bar{\xi}}}})-\partial a(\bar{\partial}\phi+{\textstyle{\frac{2i\xi}{1+\xi\bar{\xi}}}})]d\xi\wedge d\bar{\xi}=(d\phi-2j\;d\ln(1+\xi\bar{\xi}))\wedge da,
\]
and taking the Hodge star operator with respect to the  metric ${\Bbb{G}}$ we get
\[
\frac{i(1+\xi\bar{\xi})^2}{2\sqrt{\Delta}}[\bar{\partial}a(\partial\phi-{\textstyle{\frac{2i\bar{\xi}}{1+\xi\bar{\xi}}}})-\partial a(\bar{\partial}\phi+{\textstyle{\frac{2i\xi}{1+\xi\bar{\xi}}}})]=\star[(d\phi-2j\;d\ln(1+\xi\bar{\xi}))\wedge da].
\]
By elementary geometry we have that 
\[
\star[(d\phi-2j\;d\ln(1+\xi\bar{\xi}))\wedge da]\geq-\|d\phi-2j\;d\ln(1+\xi\bar{\xi})\|.\|da\|,
\]
and substituting this, we find
that
\begin{align}
\left(\frac{\partial }{\partial s}-{\mathbb G}^{jk}\partial_j\partial_k\right)\Delta&\geq\frac{\Delta}{\lambda^2}\|da\|^2
-\frac{2a\sqrt{\Delta}}{\lambda}\|d\phi-2j\;d\ln(1+\xi\bar{\xi})\|.\|da\|\nonumber\\
&\qquad\qquad\qquad\qquad   +a^2\|d\phi-2j\;d\ln(1+\xi\bar{\xi})\|^2-4\lambda\nonumber\\
&=\left(\frac{\sqrt{\Delta}}{\lambda}\|da\|-a\|d\phi-2j\;d\ln(1+\xi\bar{\xi})\|\right)^2-4\lambda\geq0.\nonumber
\end{align}
The interior bound on $\Delta$ follows.

Therefore
\[
\left(\frac{|\sigma|}{\lambda^2-|\sigma|^2}\right)^2=\frac{|\sigma|^2}{\lambda^2-|\sigma|^2}\frac{1}{\lambda^2-|\sigma|^2}\leq \frac{C_1}{C_2}.
\]
Thus
\[
{\mathbb G}(\overline{R}(X,e_{(a)})X,e_{(a)}))\;\geq k|X|^2,
\]
as claimed.
\end{proof}

\vspace{0.1in}

We now turn to showing that the flow stays in a compact region. We do so by showing that the graph function is bounded in the fibre directions of
$TS^2$ by bounding the perpendicular distance function $\chi$ (see Definition \ref{d:perpdist}).

\begin{Prop}\label{p:compact}
Under mean curvature flow the perpendicular distance function satisfies
\[
\chi\leq C_1(\chi_0,\tilde{\chi}),
\]
where $\chi_0$ and $\tilde{\chi}$ are the initial and boundary perpendicular distances to the origin.
\end{Prop}

\begin{proof}
At the edge $f(\partial D)$ the Dirichlet condition means that $F_s=\tilde{F}$ so that
\[
\chi(\partial D)=\tilde{\chi}(\partial D)\leq C_1.
\]
Thus we need only consider the interior.

First note that
\[
\partial\chi^2=4\bar{F}\partial\left(\frac{F}{(1+\xi\bar{\xi})^2}\right)+\frac{4F}{(1+\xi\bar{\xi})^2}\partial\bar{F}
    =4\frac{\bar{F}\rho-F\sigma}{(1+\xi\bar{\xi})^2},
\]
and so the expression in Proposition \ref{p:chiflow} can be rewritten
\[
\left(\frac{\partial }{\partial s}-{\mathbb G}^{jk}\partial_j\partial_k\right)\chi^2=\frac{i(F\partial\chi^2-\bar{F}\bar{\partial}\chi^2)
   -i(\sigma\xi\bar{\partial}\chi^2-\bar{\sigma}\bar{\xi}\partial\chi^2)(1+\xi\bar{\xi})+4\lambda(\rho\bar{\rho}-\sigma\bar{\sigma})}
  {\lambda^2-\sigma\bar{\sigma}}.
\]
Thus, at an interior maximum of $\chi$, $\partial\chi=\bar{\partial}\chi=0$ and so
\[
\left(\frac{\partial }{\partial s}-{\mathbb G}^{jk}\partial_j\partial_k\right)\chi^2
=\frac{4\lambda(\rho\bar{\rho}-\sigma\bar{\sigma})}{\lambda^2-\sigma\bar{\sigma}}\leq0.
\]
By the maximum principle, the result follows.

\end{proof}

\vspace{0.1in}

Having established these two propositions, we move on to the edge estimates.

\vspace{0.1in}

\subsubsection{Boundary estimates}

In this section we establish two things: a gradient estimate and an estimate on the norm of the mean curvature vector at the edge of the flowing disc.

To start, we show that the flowing disc cannot become degenerate at the edge by becoming tangent to a fibre of $TS^2\rightarrow S^2$.

\begin{Prop}\label{p:bdryest1}
At the edge, $|dF|^2<C_1(|d\tilde{F}|^2)$.
\end{Prop}
\begin{proof}
To bound the gradient of the graph function $F(\xi,\bar{\xi})$ we need to bound the slopes $\rho=\vartheta+i\lambda$ and $\sigma$. 
The asymptotic holomorphicity condition at the edge says that
\[
|\sigma|=\frac{C}{1+s}\leq C,
\]
and so we have a bound on $|\sigma|$. On the other hand the intersection condition implies 
that
\[
\rho=\tilde{\rho}+(\bar{\sigma}-\bar{\tilde{\sigma}})e^{i\beta},
\]
for some $\beta$, and so $|\rho|\leq |\tilde{\rho}|+|\sigma|+|\tilde{\sigma}|<C$, since the gradient of the graph function $\tilde{F}(\xi,\bar{\xi})$
of the boundary surface is bounded.
\end{proof}

\vspace{0.1in}

To complete the gradient estimate at the edge we must now ensure that the flowing surface does not pick up a degenerate direction which is not
tangent to the fibre of $TS^2$. That is, we must ensure that $\lambda^2-|\sigma|^2>0$ at the edge. 

To ensure this we note that, by Corollary \ref{c:nulltogether}, $\lambda^2-|\sigma|^2=0$ can only happen if the edge of the flowing disc hits a 
degenerate point on the boundary surface $\tilde{\Sigma}$,  i.e. if $\tilde{\lambda}^2-|\tilde{\sigma}|^2=0$. Thus, if we can constrain
the flowing edge to lie in a positive disc $\tilde{\Lambda}\subset \tilde{\Sigma}$, we will have the required gradient estimate.

\begin{Prop}\label{p:posregion}
Suppose 
\begin{equation}\label{e:smallness1}
C<\min\left\{\frac{\inf|\tilde{\sigma}|}{(1+4\cosh B)^2},\frac{\inf|\tilde{\sigma}|}{\sup|\tilde{\sigma}|}\right\}.
\end{equation}
If $f_s(\partial D)\subset\tilde{\Lambda}_{C_0}$ for $s=0$, then $f_s(\partial D)\subset\tilde{\Lambda}_{C_0}$ for all $s$.
\end{Prop}
\begin{proof}
Assume for the sake of contradiction that there exists a point and time $(\gamma_0,s_0)$ such that $\tilde{\mu}(\gamma_0,s_0)=1$. Then, by 
Corollary \ref{c:nulltogether}, $\mu(\gamma_0,s_0)=1$. Substituting these into equation (\ref{e:hypang2}) we find that at that point
\begin{align}
\cosh B&=\frac{\left|\frac{\sigma}{\tilde{\sigma}}\right|^{\scriptstyle{\frac{1}{2}}}
   +\left|\frac{\tilde{\sigma}}{\sigma}\right|^{\scriptstyle{\frac{1}{2}}}-\sin 2\theta\sin 2\tilde{\theta}}
{(1+\cos 2\theta)(1+\cos 2\tilde{\theta})}\geq {\textstyle{\frac{1}{4}}}\left[\left|\frac{\sigma}{\tilde{\sigma}}\right|^{\scriptstyle{\frac{1}{2}}}
   +\left|\frac{\tilde{\sigma}}{\sigma}\right|^{\scriptstyle{\frac{1}{2}}}-1\right]\nonumber\\
&\geq {\textstyle{\frac{1}{4}}}\left[\left(\frac{|\sigma|}{\sup|\tilde{\sigma}|}\right)^{\scriptstyle{\frac{1}{2}}}
   +\left(\frac{\inf|\tilde{\sigma}|}{|\sigma|}\right)^{\scriptstyle{\frac{1}{2}}}-1\right]> {\textstyle{\frac{1}{4}}}\left[
\left(\frac{C}{\sup|\tilde{\sigma}|}\right)^{\scriptstyle{\frac{1}{2}}}
   +\left(\frac{\inf|\tilde{\sigma}|}{C}\right)^{\scriptstyle{\frac{1}{2}}}-1\right],\nonumber
\end{align}
where the final inequality follows from the assumption
\[
|\sigma|\leq C<\frac{\inf|\tilde{\sigma}|}{\sup|\tilde{\sigma}|}.
\]
Rearranging the inequality, we find that at $(\gamma_0,s_0)$ 
\begin{equation}\label{e:eqq1}
1+4\cosh B-\frac{C^{\scriptstyle{\frac{1}{2}}}}{\sup|\tilde{\sigma}|^{\scriptstyle{\frac{1}{2}}}}
   -\frac{\inf|\tilde{\sigma}|^{\scriptstyle{\frac{1}{2}}}}{C^{\scriptstyle{\frac{1}{2}}}}>0.
\end{equation}
But, by assumption
\[
C<\frac{\inf|\tilde{\sigma}|}{(1+4\cosh B)^2},
\]
and so
\[
1+4\cosh B-\frac{C^{\scriptstyle{\frac{1}{2}}}}{\sup|\tilde{\sigma}|^{\scriptstyle{\frac{1}{2}}}}
   -\frac{\inf|\tilde{\sigma}|^{\scriptstyle{\frac{1}{2}}}}{C^{\scriptstyle{\frac{1}{2}}}}
<-\frac{C^{\scriptstyle{\frac{1}{2}}}}{\sup|\tilde{\sigma}|^{\scriptstyle{\frac{1}{2}}}}<0,
\]
which contradicts (\ref{e:eqq1}). Thus the flow never reaches such a point $(\gamma_0,s_0)$.
\end{proof}
\vspace{0.1in}

\begin{Cor}\label{c:bdryests}
Propositions \ref{p:bdryest1} and \ref{p:posregion} imply that if the initial aholomorphicity 
is chosen to satisfy
inequality (\ref{e:smallness1}), there exist positive constants $C_1,...,C_7$ depending only on boundary 
and initial quantities such that, for as long as the flow exists, we have at the edge
\[
C_1<\lambda^2-|\sigma|^2<C_2, \qquad 0\leq\mu<C_3<1,
\]
\[
C_4<\tilde{\lambda}^2-|\tilde{\sigma}|^2<C_5, \qquad C_6<\tilde{\mu}<C_7<1.
\]
\end{Cor}
\vspace{0.1in}

We finally turn to establishing a bound on the norm of the mean curvature at the edge. This is necessary because the interior gradient estimate 
from section \ref{s:compactmcf} depends on a bound on $|H|$ (see Proposition \ref{p:gradest}), and while we have an interior estimate for $|H|$
thanks to Proposition \ref{p:bdsff}, we still need to show that $|H|$ does not blow-up at the edge. In fact, we will bound all of the second 
derivatives.

To this end, working at a point on $f_s(\partial D)\subset\tilde{\Lambda}_{C_0}$, the tangent
space of the ambient manifold splits in two different ways:
\[
TTS^2=TD\oplus ND=T\tilde{\Lambda}_{C_0}\oplus N\tilde{\Lambda}_{C_0},
\]
where we drop the $f_s$ on the flowing disc for ease of notation here and in what follows. Let $\{\mathring{e}_{(a)},\mathring{f}_{(b)}\}$ 
be an adapted orthonormal frame for the first splitting with co-frame $\{\mathring{e}^{(a)},\mathring{f}^{(b)}\}$ and similarly define 
$\{\mathring{\tilde{e}}_{(a)},\mathring{\tilde{f}}_{(b)}\}$ and  
$\{\mathring{\tilde{e}}^{(a)},\mathring{\tilde{f}}^{(b)}\}$ for the second splitting. 

Chose these frames so that the first
tangent vectors align with the common intersection of the tangent planes, while the second normal vectors align with the common intersection of the 
normal planes. From Corollary \ref{c:multi} this means that, if $B$ is the hyperbolic angle between the planes, the following relations hold:
\[
\mathring{\tilde{e}}^{(1)}=\mathring{e}^{(1)},
\qquad
\mathring{\tilde{e}}^{(2)}=\cosh B\;\mathring{e}^{(2)}+\sinh B \;\mathring{f}^{(1)},
\]
\[
\mathring{\tilde{f}}^{(1)}=\sinh B\;\mathring{e}^{(2)}+\cosh B \;\mathring{f}^{(1)},
\qquad
\mathring{\tilde{f}}^{(2)}=\mathring{f}^{(2)},
\]
and
\[
\mathring{e}_{(1)}=\mathring{\tilde{e}}_{(1)},
\qquad
\mathring{e}_{(2)}=\cosh B\;\mathring{\tilde{e}}_{(2)}+\sinh B \;\mathring{\tilde{f}}_{(1)},
\]
\[
\mathring{f}_{(1)}=\sinh B\;\mathring{\tilde{e}}_{(2)}+\cosh B \;\mathring{\tilde{f}}_{(1)},
\qquad
\mathring{f}_{(2)}=\mathring{\tilde{f}}_{(2)}.
\]

Suppose that these aligned frames are related to the canonical frames in Proposition \ref{p:frames} by angles  
$\theta,\tilde{\theta},\psi,\tilde{\psi} $: 

\[
\mathring{e}_{(1)}=\cos\theta {e}_{(1)}-\sin\theta {e}_{(2)},
\qquad\qquad
\mathring{e}_{(2)}=\sin\theta {e}_{(1)}+\cos\theta {e}_{(2)},
\] 
\[
\mathring{f}_{(1)}=\cos\psi {f}_{(1)}+\sin\psi {f}_{(2)},
\qquad\qquad
\mathring{f}_{(2)}=-\sin\psi {f}_{(1)}+\cos\psi {f}_{(2)},
\] 
\[
\mathring{\tilde{e}}_{(1)}=\cos\tilde{\theta} {\tilde{e}}_{(1)}-\sin\tilde{\theta} {\tilde{e}}_{(2)},
\qquad\qquad
\mathring{\tilde{e}}_{(2)}=\sin\tilde{\theta} {\tilde{e}}_{(1)}+\cos\tilde{\theta} {\tilde{e}}_{(2)},
\]
\[
\mathring{\tilde{f}}_{(1)}=\cos\tilde{\psi} {\tilde{f}}_{(1)}+\sin\tilde{\psi} {\tilde{f}}_{(2)},
\qquad\qquad
\mathring{\tilde{f}}_{(2)}=-\sin\tilde{\psi} {\tilde{f}}_{(1)}+\cos\tilde{\psi} {\tilde{f}}_{(2)}.
\]

We now set out to control the values of these angles under the flow. First
\begin{Prop}
If we chose $B$ satisfying
\begin{equation}\label{e:smallness2}
B<\tanh^{-1}\left(\frac{\inf|\tilde{\sigma}|}{C_0}\right),
\end{equation}
then $\tilde{\mu}>\tanh B$ at $f_s(\partial D)$ for all $s$.
\end{Prop}
\begin{proof}
This follows simply from
\[
\tilde{\mu}=\frac{|\tilde{\sigma}|}{|\tilde{\lambda}|}\geq\frac{\inf|\tilde{\sigma}|}{C_0}>\tanh B
\]
\end{proof}

\begin{Prop}\label{p:anglecontrol}
If, in addition to the inequalities (\ref{e:smallness1}) and (\ref{e:smallness2}), the following hold:
\begin{equation}\label{e:smallness3}
C<C_0\frac{1-R_0^2}{1+R_0^2}\min\left\{1-\cosh B \left(1-{\textstyle{\frac{\inf|\tilde{\sigma}|^2}{C_0^2}}}\right)^{\scriptstyle{\frac{1}{2}}},
\cosh B \left(1-{\textstyle{\frac{\sup|\tilde{\sigma}|^2(1+R_0^2)^2}{C_0^2(1-R_0^2)^2}}}\right)^{\scriptstyle{\frac{1}{2}}}\right\},
\end{equation}
and if $0<\theta<{\textstyle{\frac{\pi}{2}}}$ and  $0<\psi<{\textstyle{\frac{\pi}{2}}}$ initially, then these bounds on the angles 
hold for as long as the flow exists.
\end{Prop}
\begin{proof}
Suppose for the sake of contradiction that at some point and time $(\gamma_0,s_0)$ we have $\theta={\textstyle{\frac{\pi}{2}}}$. Note that the
following argument holds even if $s_0=\infty$. The expression of
the hyperbolic angle (\ref{e:hypangle}) and equation (\ref{e:intersectionconstraint}) means that at $(\gamma_0,s_0)$,
\[
\cosh B=\frac{(1+\mu)\left(1-{\textstyle{\frac{|\sigma|}{|\tilde{\lambda}|}}}\right)}
{(1-\mu^2)^{\scriptstyle{\frac{1}{2}}}(1-\tilde{\mu}^2)^{\scriptstyle{\frac{1}{2}}}},
\]
or, after rearrangement
\[
\mu=\frac{\cosh^2B(1-\tilde{\mu}^2)-\left(1-{\textstyle{\frac{|\sigma|}{|\tilde{\lambda}|}}}\right)^2}
{\left(1-{\textstyle{\frac{|\sigma|}{|\tilde{\lambda}|}}}\right)^2+\cosh^2B(1-\tilde{\mu}^2)}.
\]
But
\begin{align}
\cosh B(1-\tilde{\mu}^2)^{\scriptstyle{\frac{1}{2}}}-1+{\textstyle{\frac{|\sigma|}{|\tilde{\lambda}|}}}
&\leq\cosh B(1-\tilde{\mu}^2)^{\scriptstyle{\frac{1}{2}}}-1+{\textstyle{\frac{C}{\inf|\tilde{\lambda}|}}}\nonumber\\
&\leq\cosh B\left(1-{\textstyle{\frac{\inf|\tilde{\sigma}|^2}{C_0^2}}}\right)^{\scriptstyle{\frac{1}{2}}}-1
+{\textstyle{\frac{C(1+R_0^2)}{C_0(1-R_0^2)}}}<0,\nonumber
\end{align}
by our assumption (\ref{e:smallness3}) and so $\mu<0$. But this contradicts $\mu\geq0$ and so $\theta$ is never equal to ${\textstyle{\frac{\pi}{2}}}$.

On the other hand, suppose for the sake of contradiction that at some point and time $(\gamma_0,s_0)$ we have $\theta=0$. In this case the hyperbolic
angle turns out to be
\[
\cosh B=\frac{(1-\mu)\left(1+{\textstyle{\frac{|\sigma|}{|\tilde{\lambda}|}}}\right)}{(1-\mu^2)^{\scriptstyle{\frac{1}{2}}}(1-\tilde{\mu}^2)^{\scriptstyle{\frac{1}{2}}}},
\]
which means that at $(\gamma_0,s_0)$ 
\[
\mu=\frac{\left(1-{\textstyle{\frac{|\sigma|}{|\tilde{\lambda}|}}}\right)^2-\cosh^2B(1-\tilde{\mu}^2)}{\left(1-{\textstyle{\frac{|\sigma|}{|\tilde{\lambda}|}}}\right)^2+\cosh^2B(1-\tilde{\mu}^2)}.
\]
Putting this into the left-hand inequality of (\ref{e:ineq2}) we find
\begin{align}
0&\leq \mu-\frac{\tilde{\mu}-\cosh B\sinh B\;(1-\tilde{\mu}^2)}{\tilde{\mu}^2+\cosh^2 B\;(1-\tilde{\mu}^2)}\nonumber\\
&=\frac{\left(1-{\textstyle{\frac{|\sigma|}{|\tilde{\lambda}|}}}\right)^2-\cosh^2B(1-\tilde{\mu}^2)}{\left(1-{\textstyle{\frac{|\sigma|}{|\tilde{\lambda}|}}}\right)^2+\cosh^2B(1-\tilde{\mu}^2)}-\frac{\tilde{\mu}-\cosh B\sinh B\;(1-\tilde{\mu}^2)}{\tilde{\mu}^2+\cosh^2 B\;(1-\tilde{\mu}^2)}\nonumber\\
&\leq\frac{\left[\left(1-{\textstyle{\frac{|\sigma|}{|\tilde{\lambda}|}}}\right)^2-\cosh^2B(1-\tilde{\mu}^2)\right](\tilde{\mu}^2+\cosh^2 B\;(1-\tilde{\mu}^2))}{(1+\cosh^2B(1-\tilde{\mu}^2))^2}\nonumber\\
&\qquad\qquad\qquad-\frac{(\tilde{\mu}-\cosh B\sinh B\;(1-\tilde{\mu}^2))\left[\left(1-{\textstyle{\frac{|\sigma|}{|\tilde{\lambda}|}}}\right)^2+\cosh^2B(1-\tilde{\mu}^2)\right]}{(1+\cosh^2B(1-\tilde{\mu}^2))^2}\nonumber\\
&=\frac{\left(1-{\textstyle{\frac{|\sigma|}{|\tilde{\lambda}|}}}\right)^2(\tilde{\mu}^2+\cosh^2 B\;(1-\tilde{\mu}^2)-\tilde{\mu}+\cosh B\sinh B\;(1-\tilde{\mu}^2))}{(1+\cosh^2B(1-\tilde{\mu}^2))^2}\nonumber\\
&\qquad\qquad\qquad-\frac{\cosh^2B(1-\tilde{\mu}^2)(\tilde{\mu}^2+\cosh^2 B\;(1-\tilde{\mu}^2)+\tilde{\mu}-\cosh B\sinh B\;(1-\tilde{\mu}^2))}{(1+\cosh^2B(1-\tilde{\mu}^2))^2}.\nonumber
\end{align}
Noting that for $0<\tilde{\mu}<1$ we have $\tilde{\mu}^2+\cosh^2 B\;(1-\tilde{\mu}^2)-\tilde{\mu}+\cosh B\sinh B\;(1-\tilde{\mu}^2)>0$, we can rearrange the last inequality and conclude that at $(\gamma_0,s_0)$ 
\begin{align}
C&\geq|\tilde{\lambda}|\cosh B(1-\tilde{\mu}^2)^{\scriptstyle{\frac{1}{2}}}\left[\frac{\tilde{\mu}^2+\cosh^2 B\;(1-\tilde{\mu}^2)+\tilde{\mu}-\cosh B\sinh B\;(1-\tilde{\mu}^2)}{\tilde{\mu}^2+\cosh^2 B\;(1-\tilde{\mu}^2)-\tilde{\mu}+\cosh B\sinh B\;(1-\tilde{\mu}^2)}\right]^{\scriptstyle{\frac{1}{2}}}\nonumber\\
&\geq|\tilde{\lambda}|\cosh B(1-\tilde{\mu}^2)^{\scriptstyle{\frac{1}{2}}}\nonumber\\
&\geq C_0\frac{1-R_0^2}{1+R_0^2}\cosh B \left(1-{\textstyle{\frac{\sup|\tilde{\sigma}|^2(1+R_0)^2}{C_0^2(1-R_0)^2}}}\right)^{\scriptstyle{\frac{1}{2}}}.\nonumber
\end{align}
But this contradicts assumption (\ref{e:smallness3}). Thus $\theta$ is never 0.

These exact same arguments go through upon replacing $\theta$ by $\psi$, since all of the relevant equations also hold for $\psi$ (see the proof of 
Proposition \ref{p:hypang1}). Thus 
$\psi\neq 0,{\textstyle{\frac{\pi}{2}}}$.
\end{proof}

\vspace{0.1in}

Having established control of the angles between the intersection and the canonical frames we now set about bounding the 2-jet of the flowing surface. 

\begin{Prop}
The time and space derivatives of the intersection and constant angle conditions yield the following relationships between the flowing and boundary
surfaces along the edge:
\begin{equation}\label{e:dbc1}
\nabla_1H=-H(\tilde{\Gamma}_{(212)}+\coth B\;\tilde{A}^{(\hat{1})}_{(12)}),
\qquad
\nabla_2H=-{\textstyle{\frac{H}{\sinh B}}}\tilde{A}_{(22)}^{(\hat{1})},
\end{equation}
\begin{equation}\label{e:dbc2}
C_{(\hat{1}\hat{2})}^{(1)}={\textstyle{\frac{1}{\sinh B}}}\tilde{A}^{(\hat{2})}_{(12)},
\qquad
C_{(\hat{1}\hat{2})}^{(2)}=-\tilde{C}_{(\hat{1}\hat{2})}^{(2)}+\coth B\tilde{A}^{(\hat{2})}_{(22)},
\end{equation}
\begin{equation}\label{e:dbc3}
A^{(\hat{1})}_{(12)}=\tilde{A}^{(\hat{1})}_{(12)},
\qquad
A^{(\hat{2})}_{(11)}=\tilde{A}^{(\hat{2})}_{(11)},
\end{equation}
\begin{equation}\label{e:dbc4}
\Gamma_{(112)}=\cosh B\;\tilde{\Gamma}_{(112)}+\sinh B\; \tilde{A}^{(\hat{1})}_{(11)},
\end{equation}
\begin{equation}\label{e:dbc5}
A^{(\hat{1})}_{(11)}=\sinh B\;\tilde{\Gamma}_{(112)}+\cosh B\; \tilde{A}^{(\hat{1})}_{(11)},
\end{equation}
\begin{equation}\label{e:dbc6}
A^{(\hat{2})}_{(12)}=\cosh B\;\tilde{A}^{(\hat{2})}_{(12)}+\sinh B\; \tilde{C}^{\;\;(\hat{2})}_{(1\hat{1})},
\end{equation}
\begin{equation}\label{e:dbc7}
C^{\;\;(\hat{2})}_{(1\hat{1})}=-\sinh B\;\tilde{A}^{(\hat{2})}_{(12)}-\cosh B\; \tilde{C}^{\;\;(\hat{2})}_{(1\hat{1})}.
\end{equation}
Here and throughout, $A$, $\Gamma$ and $C$ are 
the second fundamental form, tangent and normal connections of the flowing disc, in components of frames adapted to the intersection, so for example, $A_{(ab)}^{(\hat{c})}=<A(\mathring{e}_{(a)},\mathring{e}_{(b)}),\mathring{f}^{(c)}>$, where a hat on a frame
index means it is in the normal direction. Similar notation holds for tilded quantities.

\end{Prop}
\begin{proof}
In order to differentiate the constant angle condition we need to compute the evolution of the adapted frames. Starting with the frame 
adapted to the boundary surface we compute in  normal coordinates:
\begin{align}
\frac{\partial}{\partial s}\mathring{\tilde{e}}_{(b)}&=X^{(d)}\partial_{(d)}\mathring{\tilde{e}}_{(b)}
  =X^{(d)}\overline{\nabla}_{(d)}\mathring{\tilde{e}}_{(b)}
= X^{(d)}[\tilde{\nabla}_{(d)}\mathring{\tilde{e}}_{(b)}-\tilde{A}^{(\hat{c})}_{(db)}\mathring{\tilde{f}}_{(c)}]\nonumber\\
&=-\frac{H}{\sinh B}[\tilde{\nabla}_{(2)}\mathring{\tilde{e}}_{(b)}-\tilde{A}^{(\hat{c})}_{(2b)}\mathring{\tilde{f}}_{(c)}]
=-{\textstyle{\frac{H}{\sinh B}}}\tilde{\Gamma}_{(2b)}^{\;\;\;(d)}\mathring{\tilde{e}}_{(d)}
+{\textstyle{\frac{H}{\sinh B}}}\tilde{A}_{(2b)}^{(\hat{d})}\mathring{\tilde{f}}_{(d)},\nonumber
\end{align}
where we have used the splitting (\ref{e:connsplit1}) and (\ref{e:connsplit2}) of the ambient connection and assumed that the flow at the 
edge is orthogonal to the intersection in $\tilde{\Lambda}$ (see Note \ref{n:H}). 

Similarly we find that
\[
\frac{\partial}{\partial s}\mathring{\tilde{f}}_{(b)}=-{\textstyle{\frac{H}{\sinh B}}}\tilde{A}_{(2\hat{b})}^{\;\;\;(d)}\mathring{\tilde{e}}_{(d)}
-{\textstyle{\frac{H}{\sinh B}}}\tilde{C}_{(2\;\;\hat{b})}^{\;\;(\hat{d})}\mathring{\tilde{f}}_{(d)},
\]
\[
\frac{\partial}{\partial s}\mathring{\tilde{e}}^{(b)}=-{\textstyle{\frac{H}{\sinh B}}}\tilde{\Gamma}_{(2\;\;d)}^{\;\;(b)}\mathring{\tilde{e}}^{(d)}
+{\textstyle{\frac{H}{\sinh B}}}\tilde{A}_{(2\hat{d})}^{(b)}\mathring{\tilde{f}}^{(d)},
\]
\[
\frac{\partial}{\partial s}\mathring{\tilde{f}}^{(b)}=-{\textstyle{\frac{H}{\sinh B}}}\tilde{A}_{(2d)}^{(\hat{b})}\mathring{\tilde{e}}^{(d)}
-{\textstyle{\frac{H}{\sinh B}}}\tilde{C}_{(2\hat{d})}^{\;\;(\hat{b})}\mathring{\tilde{f}}^{(d)}.
\]

We now compute the evolution of the (co)frame of the flowing surface. To this end let
\[
\frac{\partial}{\partial s}\mathring{e}^{(a)}=M_{(b)}^{(a)}\mathring{e}^{(b)}+N_{(b)}^{(a)}\mathring{f}^{(b)},
\]
\[
\frac{\partial}{\partial s}\mathring{f}_{(a)}=K_{(a)}^{(b)}\mathring{e}_{(b)}+L_{(a)}^{(b)}\mathring{f}_{(b)},
\]
where $K,L,M$ and $N$ are $2\times 2$ matrices to be determined.

Differentiating $<\mathring{e}^{(a)},\mathring{f}_{(b)}>=0$ with respect to time implies that $N_{(b)}^{(a)}=-K_{(b)}^{(a)}$ and a standard
calculation (see e.g. \cite{chenli}) shows that
\[
K_{(b)}^{(a)}=\mathring{e}^{(a)j}\nabla_jH_{(b)}+H^{(c)}C^{(a)}_{(\hat{c}\hat{b})},
\]
and so we have 
\[
\frac{\partial}{\partial s}\mathring{e}^{(a)}=M_{(b)}^{(a)}\mathring{e}^{(b)}
-\left[\nabla^{(a)}H_{(b)}+HC^{(a)}_{(\hat{1}\hat{b})}\right]\mathring{f}^{(b)},
\]
\[
\frac{\partial}{\partial s}\mathring{f}_{(a)}=\left[\nabla^{(b)}H_{(a)}+HC^{(b)}_{(\hat{1}\hat{a})}\right]\mathring{e}_{(b)}
+L_{(a)}^{(b)}\mathring{f}_{(b)},
\]
with $L$ and $M$ to be determined. This we do by the time derivative of the intersection and constant angle conditions in adapted frames:
\[
0=\frac{\partial}{\partial s}<\mathring{e}^{(a)},\mathring{\tilde{e}}_{(b)}>=\frac{\partial}{\partial s}<\mathring{e}^{(a)},\mathring{\tilde{f}}_{(b)}>
=\frac{\partial}{\partial s}<\mathring{f}_{(a)},\mathring{\tilde{e}}^{(b)}>=\frac{\partial}{\partial s}<\mathring{f}_{(a)},\mathring{\tilde{f}}^{(b)}>.
\] 
For example, consider
\begin{align}
0&=\frac{\partial}{\partial s}<\mathring{f}_{(1)},\mathring{\tilde{e}}^{(1)}>=<\frac{\partial}{\partial s}\mathring{f}_{(1)},\mathring{\tilde{e}}^{(1)}>
+<\mathring{f}_{(1)},\frac{\partial}{\partial s}\mathring{\tilde{e}}^{(1)}>\nonumber\\
&=<(\nabla^{(b)}H_{(1)}+HC^{(b)}_{(\hat{1}\hat{1})})\mathring{e}_{(b)}+L_{(1)}^{(b)}\mathring{f}_{(b)},\mathring{\tilde{e}}^{(1)}>
-{\textstyle{\frac{H}{\sinh B}}}<\mathring{f}_{(1)},\tilde{\Gamma}_{(2\;\;d)}^{\;\;(1)}\mathring{\tilde{e}}^{(d)}
-\tilde{A}_{(2\hat{d})}^{(1)}\mathring{\tilde{f}}^{(d)}>\nonumber\\
&=\nabla^{(1)}H_{(1)}-{\textstyle{\frac{H}{\sinh B}}}\left[\sinh B\; \tilde{\Gamma}_{(2\;\;2)}^{\;\;(1)}-\cosh B\;\tilde{A}_{(2\hat{1})}^{(1)}\right]\nonumber\\
&=-\nabla_{(1)}H-H\left[\tilde{\Gamma}_{(212)}+\coth B\;\tilde{A}_{(12)}^{(\hat{1})}\right],\nonumber
\end{align}
which is the first of equations (\ref{e:dbc1}).
Similarly we find that
\[
0=L_{(1)}^{(1)}=L_{(2)}^{(2)}=M_{(1)}^{(1)}=M_{(2)}^{(2)},
\]
\[
L_{(1)}^{(2)}=-L_{(2)}^{(1)}=H(\tilde{A}^{(\hat{2})}_{(22)}-\coth B\;\tilde{C}_{(2\hat{1}\hat{2})}),
\]
\[
M_{(1)}^{(2)}=-M_{(2)}^{(1)}=H(\coth B\;\tilde{\Gamma}_{(212)}+\tilde{A}^{(\hat{1})}_{(12)}),
\]
and the rest of equations (\ref{e:dbc1}) and (\ref{e:dbc2}).

Equations (\ref{e:dbc3}) to (\ref{e:dbc6}) follow from the space derivatives of the intersection and constant angle conditions.
For example taking the covariant derivative of this w.r.t. the ambient connection we find
\[
<\overline{\nabla}_{\mathring{e}_{(1)}}\mathring{e}_{(a)},\mathring{\tilde{e}}^{(b)}>
+<\mathring{e}_{(a)},\overline{\nabla}_{\mathring{e}_{(1)}}\mathring{\tilde{e}}^{(b)}>=0.
\]
From decomposing the ambient connection as in equation (\ref{e:connsplit1}), for $a=2$ and $b=1$ we get
\[
<\overline{\nabla}_{\mathring{e}_{(1)}}\mathring{e}_{(2)},\mathring{\tilde{e}}^{(1)}>=<\nabla_{\mathring{e}_{(1)}}\mathring{e}_{(2)}
-A_{(12)}^{(\hat{c})}\mathring{f}_{(c)},\mathring{e}^{(1)}>=\Gamma_{(12)}^{\;\; (1)},
\]
and equation (\ref{e:connsplit1}) along with the relationship between the two sets of frames also gives
\begin{align}
<\mathring{e}_{(2)},\nabla_{\mathring{e}_{(1)}}\mathring{\tilde{e}}^{(1)}>&=<\cosh B\;\mathring{\tilde{e}}_{(2)}
+\sinh B \;\mathring{\tilde{f}}_{(1)},\tilde{\nabla}_{\mathring{\tilde{e}}_{(1)}}\mathring{\tilde{e}}^{(1)}
    -\tilde{A}_{(11)}^{(\hat{c})}\mathring{\tilde{f}}_{(c)}>\nonumber\\
&=\cosh B\; \tilde{\Gamma}_{\;\;(11)}^{\;\; (2)}-\sinh B \;\tilde{A}_{(11)}^{(\hat{1})}.\nonumber
\end{align}
Putting these together we find
\[
\Gamma_{(12)}^{\;\; (1)}+\cosh B\; \tilde{\Gamma}_{(11)}^{\;\; (2)}+\sinh B \;\tilde{A}_{(11)}^{(\hat{1})}=0,
\]
which can be rearranged to give equation (\ref{e:dbc4}). The other equations follow similarly by covariantly differentiating the inner product of 
other frame vectors.
\end{proof}
\vspace{0.1in}

We now bound the second derivatives  of the flowing surface by those of the boundary surface along the edge.

\begin{Thm}\label{t:bdryest3}
Suppose that the initial conditions $B$ and $C$ are chosen to satisfy inequalities (\ref{e:smallness1}), (\ref{e:smallness2}) and (\ref{e:smallness3})
and $0<\theta<\pi/2$, $0<\psi<\pi/2$. 
Then for as long as the flow exists $|d^2F_s|<C_1(|d^2\tilde{F}|)$ at the edge.
\end{Thm}
\begin{proof}
The idea of the proof is to use equations (\ref{e:dbc2}) to (\ref{e:dbc7}), together with the spatial derivative of the 
asymptotic holomorphicity condition, to bound the
derivatives of the 1-jet, as expressed by the quantities $\lambda$, $\mu$, $\phi$ and $\vartheta$. Note that by the identity (\ref{e:id1}) the
derivatives of $\vartheta={\mathbb R}{\mbox{e}}(\rho)$ are determined by those of $\lambda$, $\mu=|\sigma|/|\lambda|$ and 
$\phi={\mbox {arg}}(\sigma)$ - an artifact of partial derivatives commuting.

The 2-jet of the flowing surface will ultimately be bounded by the 2-jet of the boundary surface. 
As we proceed we gather these edge 2-jet terms together on the right-hand side and denote them by a generic function $\tilde{\mathcal F}$ which
is bounded for a $C^2$ boundary surface. That is, we write  $\tilde{\mathcal F}$ (possibly with a subscript) for quantities that only depend on 
the slopes $\tilde{\lambda}$, $\tilde{\mu}$, $\tilde{\phi}$ and $\tilde{\vartheta}$ and their first derivatives.

A complicating
factor in this scheme is that derivatives of the angles $\theta,\tilde{\theta},\psi,\tilde{\psi}$ also enter into some computations and have to
be controlled. 

The estimates will hold for any point $\gamma\in f_s(\partial D)$ and for computational simplicity we translate and rotate so that the point 
is $\xi=\eta=0$. The derivatives at such a point naturally split into normal and tangential components relative to the edge of the flowing disc 
and for any function $\Phi:D\rightarrow{\mathbb R}$ we denote the tangential and normal derivatives by
\[
\Phi_{|1}=\mathring{e}_{(1)}^j\frac{\partial \Phi}{\partial x^j},
\qquad \qquad
\Phi_{|2}=\mathring{e}_{(2)}^j\frac{\partial \Phi}{\partial x^j},
\]
respectively, where $\{\mathring{e}_{(1)},\mathring{e}_{(2)}\}$ are the adapted frames.

The expressions which link equations (\ref{e:dbc2}) to (\ref{e:dbc7}) to the derivatives of the 1-jet are the projections of the submanifold equations
(\ref{e:connsplit1}) and (\ref{e:connsplit2}):
\begin{equation}\label{e:defsff}
A^{(\hat{c})}_{(ab)}=\mathring{f}^{(c)}_{j}\;^\perp P_{k}^j\mathring{e}_{(a)}^{\;l}\overline{\nabla}_l\mathring{e}_{(b)}^{\;k},
\end{equation}
\begin{equation}\label{e:defconns}
C^{\;\;(\hat{2})}_{(a)\;\;(\hat{1})}=\mathring{e}_{(a)}^{j}\;^\parallel P_j^k\mathring{f}^{(2)}_{l}\overline{\nabla}_k\mathring{f}_{(1)}^{\;l},
\qquad\qquad
\Gamma^{\;\;(2)}_{(a)\;\;(1)}=\mathring{e}_{(a)}^{j}\;^\parallel P_j^k\mathring{e}^{(2)}_{l}\overline{\nabla}_k\mathring{e}_{(1)}^{\;l},
\end{equation}
the expression for the aligned frames
\[
\mathring{e}_{(1)}=\cos\theta {e}_{(1)}-\sin\theta {e}_{(2)},
\qquad\qquad
\mathring{e}_{(2)}=\sin\theta {e}_{(1)}+\cos\theta {e}_{(2)},
\] 
\[
\mathring{f}_{(1)}=\cos\psi {f}_{(1)}+\sin\psi {f}_{(2)},
\qquad\qquad
\mathring{f}_{(2)}=-\sin\psi {f}_{(1)}+\cos\psi {f}_{(2)},
\] 
and the coordinate expressions in Proposition \ref{p:frames} for the canonical frames. Similar expressions also hold
for the boundary surface with tilded quantities.

We want to bound $\lambda_{|1}$, $\lambda_{|2}$, $\mu_{|1}$, $\mu_{|2}$, $\phi_{|1}$ and $\phi_{|2}$, as well as the gauge terms
$\theta_{|1}$, $\tilde{\theta}_{|1}$, $\psi_{|1}$ and $\tilde{\psi}_{|1}$.

This last term can be bounded by noting that, by the first of equations (\ref{e:dbc2}) and (\ref{e:dbc7}) we have
\[
{\textstyle{\frac{1}{\sinh B}}}\tilde{A}^{(\hat{2})}_{(12)}=C_{(\hat{1}\hat{2})}^{(1)}=C^{\;\;(\hat{2})}_{(1\hat{1})}
=-\sinh B\;\tilde{A}^{(\hat{2})}_{(12)}-\cosh B\; \tilde{C}^{\;\;(\hat{2})}_{(1\hat{1})}.
\]
Now since $\tilde{A}^{(\hat{2})}_{(12)}$ consists 2-jet terms of the boundary and by the first expression in
(\ref{e:defconns}) we find that
\[
\tilde{C}^{\;\;(\hat{2})}_{(1\hat{1})}=\tilde{\psi}_{|1}+\tilde{\mathcal F}_1=-\coth B \tilde{A}^{(\hat{2})}_{(12)}=\tilde{\mathcal F}_2.
\]
Thus we have 
\[
\tilde{\psi}_{|1}=\tilde{\mathcal F}_3,
\]
and our first bound.

The tangential derivative of the asymptotic holomorphic condition can be written
\begin{equation}\label{e:dlam01}
\lambda_{|1}={\textstyle{\frac{\lambda}{\mu}}}\mu_{|1}.
\end{equation}
Moreover, we have from equations (\ref{e:dbc6}) and (\ref{e:dbc3}) that
\begin{align}
\sin{\scriptstyle{(\theta-\psi)}}A^{(\hat{1})}_{(12)} -\cos{\scriptstyle{(\theta-\psi)}}A^{(\hat{2})}_{(12)}
&=\sin{\scriptstyle{(\theta-\psi)}}\tilde{A}^{(\hat{1})}_{(12)}  -\cos{\scriptstyle{(\theta-\psi)}}
[\cosh B\;\tilde{A}^{(\hat{2})}_{(12)}+\sinh B\; \tilde{C}^{\;\;(\hat{2})}_{(1\hat{1})}]\nonumber\\
&=\tilde{\mathcal{F}}.\nonumber
\end{align}
Now using the explicit expression in Proposition \ref{p:2ndff} for the second fundamental form in terms of the 2-jet we find that we have
\begin{equation}\label{e:dlam02}
\lambda_{|2}=\frac{\lambda}{1-\mu^2}\left[(1+\mu-2\sin^2\theta)\mu_{|2}
-2\cos\theta\sin\theta\mu(1-\mu^2)^{\scriptstyle{\frac{1}{2}}}\phi_{|2}\right]+\tilde{\mathcal F}.
\end{equation}
Equations (\ref{e:dlam01}) and (\ref{e:dlam02}) bound the derivatives of $\lambda$ by the other derivatives and so we can reduce our list of derivatives, 
leaving us with $\mu_{|1}$, $\phi_{|1}$, $\theta_{|1}$, $\tilde{\theta}_{|1}$, $\psi_{|1}$, $\mu_{|2}$ and  $\phi_{|2}$. These split naturally
into tangent and normal derivatives which we now exploit.

The five equations (\ref{e:dbc2}) to (\ref{e:dbc7}) give a linear system for the five tangential derivatives which is of the form 
\[
\left[\begin{matrix}
*&* &0 &0 &0 \\
*&* &0 &0 &* \\
*&* &0 &0 &* \\
*&* &0 &1 &0 \\
*&* &-1 &0 &* 
\end{matrix}
\right]
\left[\begin{matrix}
{\textstyle{\frac{1}{2\mu(1-\mu^2)}}}\mu_{|1}\\
{\textstyle{\frac{\mu}{2\sqrt{1-\mu^2}}}}\phi_{|1}\\
\theta_{|1}\\
\psi_{|1}\\
\sinh B\tilde{\theta}_{|1}
\end{matrix}
\right]=\tilde{\mathcal F}.
\]
Note that $|d\sigma|^2=d|\sigma|^2+|\sigma|^2|d\phi|^2$ and so to bound the gradients of $\sigma$ we
need only bound $d|\sigma|$ and $|\sigma|d\phi=|\lambda|\mu\phi$.

The last two equations imply that bounds on $\mu_{|1}$, $\phi_{|1}$ and $\tilde{\theta}_{|1}$ yield bounds on $\theta_{|1}$ and $\psi_{|1}$. Thus 
we are left with a $3\times 3$ system for $\mu_{|1}$, $\phi_{|1}$ and $\tilde{\theta}_{|1}$ with matrix of coefficients
\[
\left[\begin{matrix}
1-\cos 2\theta\mu&\sin 2\theta &0 \\
 & & \\
\sin 2\psi +(\sin 2\theta \cos 2\psi+\sin 2(\theta+\psi))\mu&\cos 2\theta\cos 2\psi+\cos 2(\theta+\psi)&\cos(\theta+\psi) \\
 & & \\
\cos 2\psi-(\sin 2\theta\sin 2\psi -\cos 2(\theta+\psi))\mu&- \cos 2\theta\sin 2\psi -\sin 2(\theta+\psi) &-\sin(\theta+\psi)
\end{matrix}
\right].
\]

The determinant of this matrix is 
\[
\cos\psi\sin\theta(1+\mu)+\cos\theta\sin\psi(1-\mu)>0,
\]
where the inequality follows from (\ref{e:smallness1}), (\ref{e:smallness2}) and (\ref{e:smallness3}). We can therefore invert 
the system to yield bounds on the tangential derivatives: $\mu_{|1}$, $\phi_{|1}$ and $\tilde{\theta}_{|1}$. The final tangential derivative we 
have to bound is $\vartheta_{|1}$ and this follows readily from differentiating equation (\ref{e:hypangcon3}) along the edge.

We turn now to bounding the normal derivatives $\mu_{|2}$ and $\phi_{|2}$. These are controlled by the tangential derivatives due to the
generalized Coddazzi-Mainardi equation (\ref{e:id1}) and the condition (\ref{e:compat}). To see this, introduce
\[
\alpha_\theta=\cos\theta\alpha_1-\sin\theta\alpha_2
={\textstyle{\frac{1}{\sqrt{2}}}}\left(\frac{\cos\theta}{(|\lambda|-|\sigma|)^{\scriptstyle{\frac{1}{2}}}}
+\frac{i\sin\theta}{(|\lambda|+|\sigma|)^{\scriptstyle{\frac{1}{2}}}}\right)e^{-{\scriptstyle{\frac{i}{2}}}\phi+{\scriptstyle{\frac{i}{4}}}\pi},
\]
so that the tangential and normal derivative operators are
\[
\mathring{e}_{(1)}=\alpha_\theta\frac{\partial}{\partial \xi}+\overline{\alpha}_\theta\frac{\partial}{\partial \bar{\xi}},
\qquad\qquad
\mathring{e}_{(2)}=\alpha_{\theta-{\scriptstyle{\frac{\pi}{2}}}}\frac{\partial}{\partial \xi}
+\overline{\alpha}_{\theta-{\scriptstyle{\frac{\pi}{2}}}}\frac{\partial}{\partial \bar{\xi}},
\]
which can be inverted to
\[
\frac{\partial}{\partial \xi}=
-i(\lambda^2-|\sigma|^2)^{\scriptstyle{\frac{1}{2}}}\left[\overline{\alpha}_{\theta-{\scriptstyle{\frac{\pi}{2}}}}\mathring{e}_{(1)}
-\overline{\alpha}_\theta\mathring{e}_{(2)}\right].
\]

Equation (\ref{e:id1}) says that
\[
\partial\bar{\sigma}-\bar{\partial}\vartheta-i\bar{\partial}\lambda=\tilde{\mathcal F},
\]
which relates certain tangential and normal derivatives. Thus
\[
\vartheta_{1}=\alpha_\theta\partial\vartheta+\overline{\alpha}_\theta\bar{\partial}\vartheta
=\alpha_\theta(i\partial\lambda+\bar{\partial}\sigma)+\overline{\alpha}_\theta\bar{\partial}
(-i\bar{\partial}\lambda+\partial\bar{\sigma})+\tilde{\mathcal F}.
\]
Now a direct computation reduces this to
\[
-(1-\cos\theta)(\mu+\cos 2\theta)\mu_{|2}+\sin 2\theta(1-\cos\theta)\mu(1-\mu^2)^{\scriptstyle{\frac{1}{2}}}\phi_{|2}=\tilde{\mathcal F}.
\]
On the other hand equation (\ref{e:compat}) implies that
\[
\cos(\theta+\psi)\mu_{|2}-\sin(\theta+\psi)\mu(1-\mu^2)^{\scriptstyle{\frac{1}{2}}}\phi_{|2}=\tilde{\mathcal F}.
\]
The last two equations form a linear system for $\mu_{|2}$ and $\phi_{|2}$ whose coefficients have determinant
\[
-(1-\cos\theta)\mu(1-\mu^2)^{\scriptstyle{\frac{1}{2}}}[\cos\psi\cos\theta(1+\mu)+\sin\psi\sin\theta(1-\mu)].
\]
Inequalities (\ref{e:smallness1}), (\ref{e:smallness2}) and (\ref{e:smallness3}) ensure that this remains strictly negative during the flow and 
hence we can invert the system and get bounds for $\mu_{|2}$ and $\phi_{|2}$. This completes the proof.
\end{proof}

\vspace{0.1in}

\subsection{{Existence of a holomorphic disc}}

In this section we prove that we can attach a holomorphic disc to the set of oriented normals of any non-umbilic convex hemisphere in 
${\mathbb E}^3$, considered as a surface in $TS^2$.

We first prove a final estimate.

\begin{Prop}\label{p:shearest}
Under the mean curvature flow we have the following estimate:
\[
\left(\frac{\partial }{\partial s}-{\mathbb G}^{jk}\partial_j\partial_k\right)\left(
\frac{|\sigma|^2}{\lambda^2-|\sigma|^2}\right)\leq
    \frac{4\lambda}{|\sigma|^2}\frac{|\sigma|^4}{(\lambda^2-|\sigma|^2)^2},
\]
\end{Prop}
\begin{proof}
Our starting point is equation (\ref{e:holconv}), which we rewrite in the form
\[
\left(\frac{\partial }{\partial s}-{\mathbb G}^{jk}\partial_j\partial_k\right)\left(
\frac{|\sigma|^2}{\lambda^2-|\sigma|^2}\right)=I_1+I_2+I_3+I_4,
\]
where
\[
I_1=-2\frac{(\lambda^2+|\sigma|^2)}{(\lambda^2-|\sigma|^2)^3}
    \Big\|\lambda d|\sigma|-|\sigma|d\lambda -\frac{\lambda^2-|\sigma|^2}{\lambda^2+|\sigma|^2}\frac{\lambda|\sigma|}{1+\xi\bar{\xi}}d(1+\xi\bar{\xi})\Big\|^2,
\]
\[
I_2=2\frac{(1+\xi\bar{\xi})\lambda^2|\sigma|}{(\lambda^2-|\sigma|^2)^3}<d(1+\xi\bar{\xi}),\lambda d|\sigma|-|\sigma|d\lambda>,
\]
\[
I_3=-2\frac{\lambda^2|\sigma|^2}{(\lambda^2-|\sigma|^2)^3}\Big\|d\phi-2(1+\xi\bar{\xi})^{-1}j[d(1+\xi\bar{\xi})]\Big\|^2,
\]
\[
I_4=\frac{|\sigma|^2}{2(\lambda^2-|\sigma|^2)^2(\lambda^2+|\sigma|^2)}\Big\{-i|\sigma|\lambda^2
    (\xi^2e^{i\phi}-\bar{\xi}^2e^{-i\phi})+4\lambda\{[2-\xi\bar{\xi}]\lambda^2+2|\sigma|^2\}\Big\}.
\]
Here $\sigma=|\sigma|e^{i\phi}$ and we have introduced the flat complex structure $j(d\xi)=id\xi$.

Introduce the flat metric
\[
<d\xi,d\bar{\xi}>=1, \qquad \qquad <d\xi,d\xi>=<d\bar{\xi},d\bar{\xi}>=0,
\]
on $\Sigma$ via its coordinate $\xi$. Denote the flat norm and inner product by $|.|$ and $<\cdot,\cdot>$, and the norm and inner product 
of $g$ by $\|.\|$ and $<<\cdot,\cdot>>$. The following estimates will prove useful:

\begin{Lem}\label{l:metrics}
\[
\frac{(1+\xi\bar{\xi})^2(-\lambda-|\sigma|)}{\lambda^2-|\sigma|^2}|X|^2\leq\|X\|^2\leq\frac{(1+\xi\bar{\xi})^2(-\lambda+|\sigma|)}{\lambda^2-|\sigma|^2}|X|^2.
\]
\end{Lem}

\vspace{0.1in}

First we estimate $I_1$ using the flat metric and Lemma \ref{l:metrics}:
\[
I_1\leq-2\frac{(1+\xi\bar{\xi})^2(\lambda^2+|\sigma|^2)}{(\lambda^2-|\sigma|^2)^3(-\lambda+|\sigma|)}
    \Big|\lambda d|\sigma|-|\sigma|d\lambda -\frac{\lambda^2-|\sigma|^2}{\lambda^2+|\sigma|^2}\frac{\lambda|\sigma|}{1+\xi\bar{\xi}}
   d(1+\xi\bar{\xi})\Big|^2,
\]
and so after completing the squares
\begin{align}
I_1+I_2&\leq-2{\textstyle{\frac{(1+\xi\bar{\xi})^{2}(\lambda^2+|\sigma|^2)}{(\lambda^2-|\sigma|^2)^3(-\lambda+|\sigma|)}}}\Bigg\{
    \Big|\lambda d|\sigma|-|\sigma|d\lambda +{\textstyle{\frac{(-\lambda+|\sigma|)(-\lambda+2|\sigma|)\lambda|\sigma|}{2(\lambda^2+|\sigma|^2)(1+\xi\bar{\xi})}}}d(1+\xi\bar{\xi})\Big|^2\nonumber\\
&\qquad\qquad\qquad\qquad-{\textstyle{\frac{\lambda(-\lambda+|\sigma|)^2(-3\lambda-4|\sigma|)}{4(\lambda^2+|\sigma|^2)^2}\frac{\lambda^2|\sigma|^2}{(1+\xi\bar{\xi})^2}}}
    \Big|d(1+\xi\bar{\xi})\Big|^2\Bigg\}\nonumber.
\end{align}

Clearly $I_3$ is negative, so we discard it. To estimate $I_4$ we use
\[
-2\xi\bar{\xi}\leq i(\xi^2e^{i\phi}-\bar{\xi}^2e^{-i\phi})\leq 2\xi\bar{\xi}.
\]
Thus
\[
I_4\leq{\textstyle{\frac{|\sigma|^2}{2(\lambda^2-|\sigma|^2)^2(\lambda^2+|\sigma|^2)}}}
    \Big\{2|\sigma|\lambda^2\xi\bar{\xi}+2\lambda\{[4-2\xi\bar{\xi}]\lambda^2+4|\sigma|^2\}\Big\}.
\]
Combining the estimates
\[
\left(\frac{\partial }{\partial s}-{\mathbb G}^{jk}\partial_j\partial_k\right)\left(\frac{|\sigma|^2}{\lambda^2-|\sigma|^2}\right)\leq\frac{4\lambda|\sigma|^2}{(\lambda^2-|\sigma|^2)^2}
    +\frac{\lambda^2|\sigma|^2(\lambda^3+2\lambda^2|\sigma|-2\lambda|\sigma|^2-|\sigma|^3)\xi\bar{\xi}}{(\lambda^2-|\sigma|^2)^3(\lambda^2+|\sigma|^2)}.
\]
In fact, we can achieve $|\lambda|\geq3|\sigma|$ throughout the flow (see Corollary \ref{c:bdryests} with $C_3<1/3$), so that
\begin{align}
\left(\frac{\partial }{\partial s}-{\mathbb G}^{jk}\partial_j\partial_k\right)\left(\frac{|\sigma|^2}{\lambda^2-|\sigma|^2}\right)&\leq\frac{4\lambda|\sigma|^2}{(\lambda^2-|\sigma|^2)^2}
    +\frac{\lambda^5|\sigma|^2\xi\bar{\xi}}{9(\lambda^2-|\sigma|^2)^3(\lambda^2+|\sigma|^2)}\nonumber\\
& \leq\frac{4\lambda}{|\sigma|^2}\frac{|\sigma|^4}{(\lambda^2-|\sigma|^2)^2}\nonumber.
\end{align}
This completes the proof of the Proposition.
\end{proof}

\vspace{0.1in}

We now show that our flow is asymptotically holomorphic:

\begin{Prop}\label{p:asymhol}
The mean curvature flow satisfies:
\[
|\sigma|^2\rightarrow 0 \qquad\qquad{\mbox as} \qquad s\rightarrow\infty.
\]
\end{Prop}
\begin{proof}
At the edge this follows from the second Neumann condition. To consider the interior, recall the following estimate:
\[
\left(\frac{\partial }{\partial s}-{\mathbb G}^{jk}\partial_j\partial_k\right)\left(
\frac{|\sigma|^2}{\lambda^2-|\sigma|^2}\right)\leq
    \frac{4\lambda}{|\sigma|^2}\frac{|\sigma|^4}{(\lambda^2-|\sigma|^2)^2}.
\]
This is of the form
\[
\left(\frac{\partial }{\partial s}-{\mathbb G}^{jk}\partial_j\partial_k\right)f\leq-C_1^2f^2,
\]
for the positive function $f$ given by 
\[
f=\frac{|\sigma|^2}{\lambda^2-|\sigma|^2}, 
\]
and $C_1$ is a constant such that
\[
\frac{4|\lambda|}{|\sigma|^2}\geq C_1^2,
\]
which exists by Corollary \ref{c:bdryests}.
Following Ecker and Huisken \cite{EaH}, let $g=sf$ and compute
\[
\left(\frac{\partial }{\partial s}-{\mathbb G}^{jk}\partial_j\partial_k\right)g\leq gs^{-1}(1-C_1^2g).
\]
Now suppose that the maximum of $g$ occurs in the interior of the disc. Then, by the maximum principle we must have $1-C_1^2g\geq0$ or, equivalently,
$f\leq C_1^{-2}s^{-1}$. Returning to our notation, this means that
\[
\frac{|\sigma|^2}{\lambda^2-|\sigma|^2} \leq \frac{1}{C_1^2s},
\]
or 
\[
|\sigma|^2\leq\frac{\lambda^2-|\sigma|^2}{C_1^2s} \leq \frac{C_2}{s}.
\]
That is, the flow is asymptotically holomorphic.
\end{proof}

\vspace{0.1in}

Finally, drawing the results together, the existence of the holomorphic disc is established as follows.

\vspace{0.1in}

\begin{Thm}\label{t:ashol}
Let $S$ be a $C^{3+\alpha}$ smooth open convex surface in ${\mathbb E}^3$ without umbilic points 
and suppose that the Gauss image of $S$ contains a closed hemisphere. 
Let $\Sigma\subset TS^2$ be the oriented normals of $S$. 

Then $\exists f:D\rightarrow TS^2$ with $f\in C^{1+\alpha}_{loc}(D)\cap C^0(\overline{D})$ satisfying
\begin{enumerate}
\item[(i)] $f$ is holomorphic,
\item[(ii)] $f(\partial D)\subset\Sigma$.
\end{enumerate}
\end{Thm} 
\begin{proof}
By a rotation in ${\mathbb E}^3$ and the induced action on $TS^2$  we can take
the north pole of $\Sigma$ to $\xi=0$. Now deform $\Sigma$ to $\tilde{\Sigma}$ by adding a 
holomorphic twist. Thus if $\Sigma$ is given by the graph
function $\eta=F(\xi,\bar{\xi})$, then $\tilde{\Sigma}$ is given by the graph function $\eta=\tilde{F}(\xi,\bar{\xi})=F(\xi,\bar{\xi})-iC_0\xi$.
We choose $C_0>0$ large enough so that $\tilde{\Sigma}$ is positive at the pole.

We can now apply Theorem \ref{t:lte} to find a long-time solution 
$f\in C^{2+\alpha}_{loc}(D\times{\mathbb R}_{\geq0})\cap C^1(\overline{D}\times{\mathbb R}_{\geq0})$ to mean curvature flow with edge 
in $\tilde{\Lambda}$, so long as the initial conditions $B$ and $C$ are chosen small enough. Moreover, in Proposition \ref{p:asymhol} we 
showed that this solution is asymptotically holomorphic in time:
\[
|\sigma|^2\rightarrow 0 \qquad\qquad{\mbox as} \qquad s\rightarrow\infty.
\]
Now by parabolic Schauder estimates, (see e.g. \cite{LaS} section 6, and \cite{JJ} page 79), we have 
\[
||f(\cdot,t)||_{C_{loc}^{1+\alpha}(D)} \le 
C(||H||_{L^\infty(D \times [0,\infty))} + ||\overline{K}(\cdot,t)||_{L^\infty(D)}).
\]
Here, $\overline{K}$ involves ambient metric, Christoffel symbols, and the
gradient $f$ (all of which are bounded since the evolution takes place
in a relatively compact subset of $TS^2$), and on the mean curvature vector $H$,
which is bounded for all time. 
The right hand side being bounded in time, and using the gradient bound from
uniform positivity, we can by Arzela-Ascoli extract a subsequence 
$t_j \to \infty$ limit disc 
$\tilde{f}_\infty(D)$, where  
$\tilde{f}_\infty \in C_{loc}^{1+\alpha'}(D)\cap C^0(\overline{D})$, for $\alpha' < \alpha$. 
From asymptotic holomorphicity, proved in the next section, it now 
follows that $\tilde{f}_\infty(D)$ is holomorphic with respect to ${\mathbb J}$.
Note that we do {\it not} have smooth convergence up to the boundary, and
that (in general), the angle condition (iii) in I.B.V.P. is not retained
by the holomorphic limit $\tilde{f}_\infty(D)$.

Finally, the holomorphic disc $f_\infty(D)$ with edge lying on $\Sigma$ can now be obtained by subtracting the holomorphic twist.

\end{proof}

\vspace{0.1in}


\section{{\bf Concluding Remarks}}

\noindent{\bf Background}

The affirmation of such a venerable conjecture is deserving of some general remarks on the methods employed. 
The basic outline of the proof is as follows. 

The Carath\'eodory conjecture is an under-determined hyperbolic problem about the global nature of solutions for the Codazzi-Mainardi system on the 
sphere. By transferring to the enlarged geometric 
setting of the space of oriented lines, the neutral K\"ahler metric allows us to deform it to an elliptic boundary value problem. 
Mean curvature flow then allows us not only to solve this problem, but to get to the most rigid of elliptic objects: holomorphic curves. The 
proof of the Conjecture then follows from the global consequences of the flexibility of the totally real boundary condition in the neutral setting.

Thus, our proof brings together two currents in contemporary PDE: parabolic flows and the rigidity of holomorphic curves, both in a 
new geometric setting. That these techniques are just sufficient to prove the Carath\'eodory conjecture gives an indication of the depth at which it lies.
The geometric setting (that of neutral K\"ahler surfaces) is sufficiently new that results from both of these fields require modification. In what 
follows we sketch the salient features of the proof, highlighting the key points and offering an enlarged perspective on the work.

\vspace{0.1in}

\noindent{\bf Mean curvature flow and holomorphic discs}

In the case of the mean curvature flow, we must establish a priori gradient estimates for long-time existence in the indefinite setting with higher
codimension. In this instance the indefinite signature of the metric assists the analysis: as long as the flowing submanifold remains in a compact set, 
a mild curvature assumption ensures that singularity formation does not occur. Thus many of the difficulties associated with flowing in the definite 
case are avoided, and a general result on mean curvature flow of positive surfaces in indefinite manifolds is established. This ``good'' sign for
mean curvature flow has also been exploited in other contexts \cite{LaS}.

However, the higher codimension significantly complicates the gradient estimates required for long-time existence. In addition,
rather than working with the more usual case of compact submanifolds, we must consider the flow with mixed Dirichlet and Neumann boundary 
conditions and therefore all quantities have had to be controlled at the edge. 

One can interpret our {\bf I.B.V.P.} as a capillary problem in codimension two. Previous research on this classical problem in codimension one
uses a variety of techniques to ensure long-time existence and convergence. In the higher codimension case many of these techniques are 
inadequate. For example, the Hopf maximum principle at the edge employed in \cite{stahl} fails, as we have no obvious notion of convexity.
Similarly, barrier arguments are not available as we need arbitrary totally real free boundary conditions for which explicit barriers with edge
are difficult to construct. 

From this perspective, our assumptions on initial and boundary data can be interpreted as increasing the hyperbolic adhesion along the intersection 
of the flowing surface and the boundary surface. 
This allows us to stop the flow from leaving the boundary hemisphere and ensures that it is well-defined for all time. 

In general we do not expect convergence of the flow in $C^{2+\alpha}$, as this would over-prescribe the limit holomorphic disc at the edge. 
Passing to a convergent subsequence, with the drop in differentiability along the edge, is therefore unavoidable.

The existence of a holomorphic disc satisfying the Dirichlet condition has a topological implication for the boundary, as is well-known in symplectic 
geometry.

\vspace{0.1in}

\noindent{\bf Neutral K\"ahler surfaces}

Let us now turn to the geometric setting. The metric employed in the proof was first defined, as far as the authors are aware, by Study \cite{study}. 
As shown in \cite{gak4}, it extends to a neutral K\"ahler structure and is canonical in the sense that
it is the unique  metric on the space of oriented lines (up to addition of a spherical element) that is invariant under the action of the Euclidean group
\cite{salvai}. 

Such an invariant metric exists on the space of oriented geodesics of any 3-dimensional space form and so our method may well 
extend to a proof of the Carath\'eodory conjecture in the 3-sphere and hyperbolic 3-space. In fact, there exist such invariant metrics on
geodesic spaces of many symmetric spaces \cite{agk}. 

Perhaps some insight into the difficulty of the Conjecture in Euclidean 3-space is afforded by the following observation. The K\"ahler metric
is not K\"ahler-Einstein and so mean curvature flow does not preserve the Lagrangian condition. Thus, viewed in ${\mathbb E}^3$, our flow
twists the normal lines and we lose the orthogonal surface. It is in this extended context that we find the flexibility to prove the Conjecture.

\vspace{0.1in}

\noindent{\bf Local version of the Conjecture}

Previous published efforts at proving the Carath\'eodory conjecture have been focused on establishing a local index conjecture due to Hamburger. 
Guided by the topological fact that the sum of the indices of isolated umbilic points on a closed convex body must be 2, Hamburger conjectured and then 
sought to establish a bound on the winding number of any isolated umbilic point \cite{Ham}. 

Since isolated umbilic points with all indices less than or equal to 1 are 
easy to construct, the conjecture of Hamburger, often attributed to Loewner, is that the winding number of an isolated umbilic point 
must be less than or equal to 1 (recall that the index takes values in ${\textstyle{\frac{1}{2}}}{\mathbb Z}$). 

Historically, most approaches to the 
Carath\'eodory conjecture attempt to prove this local version in the 
case where the surface is real analytic - a recent attempt to improve the exposition of Hamburger's work can be found in \cite{Ivan}. Other work 
on isolated umbilic points without the assumption of real analyticity include \cite{GaS} and \cite{SaX}.

The methods employed in this paper can be extended to bound the index of an isolated umbilic point on a $C^{3+\alpha}$ surface. We briefly summarize 
the argument - details can be found in \cite{gak10}. 

For the sake of contradiction, suppose we have a convex surface with an isolated umbilic point of index $I=2+k/2$ for $k\geq0$. Extend the surface to 
a closed convex surface. There will be other umbilic points on the surface and the sum of their indices will be $-k/2$. By moving into general position
near these points we get $k$ umbilics each of index $-1/2$.

In $TS^2$ we would therefore have an embedded Lagrangian surface $\Sigma$ with one complex point of index $4+k$, and $k$ hyperbolic complex points each of 
index $-1$. Now a totally real version of blowing-up $\Sigma$ at these hyperbolic points yields $\tilde{\Sigma}=\Sigma\#k({\mathbb R}P^2)$. That is 
we attach an ${\mathbb R}P^2$ in place of each index $-1$ complex point and these cross-caps can be made to be totally real. In fact, $\tilde{\Sigma}$ is 
\begin{itemize}
\item compact and embedded, 
\item Lagrangian away from the cross-caps,
\item totally real except at the single complex point of index $4+k$.
\end{itemize}

Now all of the arguments from this paper can be applied: the space of Lagrangian deformations fixing the unique complex point is a Banach manifold and is
transverse for the $\bar{\partial}$ operator. Thus on generic deformations of the surface $\tilde{\Sigma}$ there cannot exist holomorphic curves 
with edge on $\tilde{\Sigma}$ unless they encircle a complex point. 

Once again, the existence of holomorphic discs which do {\it not} encircle complex points is proven by mean curvature flow 
and the contradiction implies that $k<0$, or the index of an isolated umbilic point on a convex $C^{3+\alpha}$ surface in Euclidean 3-space is 
less than 2. This raises the intriguing possibility of the existence of surfaces with an {\it exotic} umbilic point: a smooth (but not real analytic) 
surface with an isolated umbilic point of index 3/2.

Full details of these arguments can be found in \cite{gak10} and a summary of the arguments is given in the accompanying video.

\vspace{0.2in}

\appendix
\section{Notation}

Here we summarize the symbols we use:

\vspace{0.1in}
 
\noindent${\mathbb E}^3$ \hspace{0.2in} Euclidean 3-space

\noindent$TS^2$ \hspace{0.2in} total space of the tangent bundle to the 2-sphere

\noindent${\mathbb J}$ \hspace{0.2in} complex structure

\noindent$\Omega$ \hspace{0.2in} symplectic 2-form

\noindent${\mathbb G}$ \hspace{0.2in} neutral metric

\noindent$S$ \hspace{0.2in} surface in ${\mathbb E}^3$

\noindent$\Sigma$ \hspace{0.2in} surface in $TS^2$

\noindent$\pi$ \hspace{0.2in} projection $TS^2\rightarrow S^2$

\noindent${\mathbb M}$ \hspace{0.2in} $n+m$-dimensional manifold

\noindent$\gamma$ \hspace{0.2in} point in $TS^2$, oriented line in ${\mathbb E}^3$

\noindent$p$ \hspace{0.2in} point in ${\mathbb E}^3$

\noindent$i(p)$ \hspace{0.2in} index of isolated umbilic point $p\in S$ 

\noindent$\mu(TS^2,T\Sigma)$ \hspace{0.2in} Keller-Maslov index of a curve on $\Sigma$

\noindent$(\xi,\eta)$ \hspace{0.2in} canonical coordinates on $TS^2-\pi^{-1}(\infty)$

\noindent$(R,\theta)$ \hspace{0.2in} polar coordinates for $\xi$

\noindent$D$ \hspace{0.2in} the open unit disc in ${\mathbb C}$

\noindent$\partial D$ \hspace{0.2in} the unit circle in ${\mathbb C}$

\noindent$\bar{\partial}$ \hspace{0.2in} the Cauchy-Riemann operator

\noindent$C^{k+\alpha}$ \hspace{0.2in} H\"older space 

\noindent$H^{s}$ \hspace{0.2in} Sobolev space

\noindent${\mathcal L}ag$ \hspace{0.2in} set of entire Lagrangian sections

\noindent${\mathcal L}ag_0$ \hspace{0.2in} set of entire Lagrangian sections with fixed point

\noindent${\mathcal F}$ \hspace{0.2in} set of functions with edge on a surface

\noindent${\mathcal M}$ \hspace{0.2in} set of holomorphic discs with edge on a totally real surface

\noindent$I$ \hspace{0.2in} analytic index of an elliptic boundary value problem

\noindent$s$ \hspace{0.2in} parabolic time variable

\noindent$f_s$ \hspace{0.2in} flowing submanifold 

\noindent$H_\alpha$ \hspace{0.2in} mean curvature vector

\noindent$\tilde{\Sigma}$ \hspace{0.2in} boundary surface in $TS^2$

\noindent$\overline{\nabla}$ \hspace{0.2in} ambient Levi-Civita connection

\noindent$\overline{Ric}$ \hspace{0.2in} ambient Ricci tensor

\noindent$\overline{R}_{ijkl}$ \hspace{0.2in} ambient Riemann tensor

\noindent$\nabla^{\parallel}$ \hspace{0.2in} induced Levi-Civita connection

\noindent$\nabla^\bot$ \hspace{0.2in} normal connection

\noindent$\{e_i,T_\alpha\}_{i,\alpha=1}^{n,m}$ \hspace{0.2in} background orthonormal frame in ${\mathbb M}$

\noindent$\{\tau_i,\nu_\alpha\}_{i,\alpha=1}^{n,m}$ \hspace{0.2in} orthonormal frame adapted to flowing submanifold

\noindent$A_{ij\alpha}$ \hspace{0.2in} second fundamental form

\noindent$C_{i\alpha}^\beta$ \hspace{0.2in} normal connection coefficients

\noindent$\triangle$ \hspace{0.2in} Laplacian of induced connection

\noindent$ ^\| P_i^j$ \hspace{0.2in} parallel projection operator

\noindent$ ^\bot P_i^j$ \hspace{0.2in} perpendicular projection operator

\noindent$\nu$ \hspace{0.2in} generalised tilt function

\noindent$F$ \hspace{0.2in} graph function of a section of $TS^2\rightarrow S^2$ 

\noindent$\sigma$ \hspace{0.2in} shear of a 2-parameter family of oriented lines

\noindent$\phi$ \hspace{0.2in} argument of $\sigma$

\noindent$\vartheta$ \hspace{0.2in} divergence of a 2-parameter family of oriented lines

\noindent$\lambda$ \hspace{0.2in} twist of a 2-parameter family of oriented lines

\noindent$\Delta$ \hspace{0.2in} $\Delta=\lambda^2-|\sigma|^2$

\noindent$\mu $ \hspace{0.2in} $\mu=|\sigma|/|\lambda|$

\noindent$\{E_{(\mu)}\}_{\mu=1}^4$ \hspace{0.2in} orthonormal frame on $TS^2$

\noindent$\{e_{(a)}\}_{a=1}^2$ \hspace{0.2in} orthonormal frame of tangent bundle $T\Sigma$

\noindent$\{f_{(a)}\}_{a=1}^2$ \hspace{0.2in} orthonormal frame of normal bundle $N\Sigma$

\noindent$\{\mathring{e}_{(a)}\}_{a=1}^2$ \hspace{0.2in} adapted orthonormal frame to tangent bundle $T\Sigma$

\noindent$\{\mathring{f}_{(a)}\}_{a=1}^2$ \hspace{0.2in} adapted orthonormal frame to normal bundle $N\Sigma$

\noindent$M_{(\mu)}^{(\nu)}$ \hspace{0.2in} matrix of angles

\noindent$B$ \hspace{0.2in} hyperbolic angle between two intersecting positive surfaces

\noindent$\chi$ \hspace{0.2in} perpendicular distance of an oriented line to the origin

\end{document}